\documentclass[10pt,twoside]{amsart}
\usepackage{mathtools,amsmath,amssymb}
\usepackage{mathrsfs}
\usepackage{eucal}

\usepackage{libertine}       
\usepackage[a4paper,top=1in, bottom=1in, left=1in, right=1in]{geometry} 
\usepackage{graphicx} 
\usepackage[colorlinks=true,linkcolor=blue,citecolor=black]{hyperref} 

\usepackage{cleveref}
\usepackage[backend=biber,backref=true,doi=false,url=false,maxnames=9]{biblatex} 
\renewbibmacro*{pageref}{%
  \iflistundef{pageref}
    {}
    {\addspace
     \mkbibbrackets{\printlist[pageref][-\value{listtotal}]{pageref}}}} 
\usepackage{lipsum}
\usepackage{adjustbox} 
\usepackage{tikz-cd}
\usepackage{mathptmx}
\usepackage{quiver}
\usepackage{float}
\usepackage[english]{babel}
\usepackage{amsthm}
\usepackage{mathtools}
\usepackage{soul}
\usepackage{nicematrix}
\usepackage{array}   
\usepackage{booktabs}
\usepackage{float}
\usepackage{orcidlink}
\usepackage{csquotes}
\usepackage{booktabs}
\usepackage{geometry}
\usepackage{multirow}
\usepackage{nccmath}
\usepackage[normalem]{ulem}

\usepackage{nomencl}

\makenomenclature
\usepackage[toc,page]{appendix}
\usepackage{caption}
\newcommand{\GL}{\mathrm{GL}}
\newcommand{\SL}{\mathrm{SL}}
\newcommand{\tr}{\mathrm{tr}}

\newcommand{\M}{\mathrm{M}}
\newcommand{\im}{\mathrm{Im}}

\newcommand{\R}{\mathcal{O}_2}

\newtheorem{theorem}{Theorem}[section]
\theoremstyle{definition}
\newtheorem{definition}[theorem]{Definition}
\theoremstyle{definition}

\newtheorem{lemma}[theorem]{Lemma}
\newtheorem{corollary}[theorem]{Corollary}
\newtheorem{proposition}[theorem]{Proposition}
\newtheorem{remark}[theorem]{Remark}
\newtheorem{example}[theorem]{Example}
\newtheorem{question}[theorem]{Question}
\newtheorem{conjecture}{Conjecture}
\title[Engel word on $\SL_2(\mathcal{O})$ and $\mathrm{PSL}_2(\mathcal{O}_2)$]{Surjectivity of Engel Words on $\SL_2(\mathcal{O})$ and $\mathrm{PSL}_2(\mathcal{O}_2)$}
\date{\today}
\author[Roy]{Ayon Roy}
\email[(Roy)]{ayonroy1999@gmail.com}
\address{Indian Institute of Science Education and Research Pune, Dr. Homi Bhabha Road, Pashan, Pune 411 008, India}
\author[Singh]{Anupam Singh}
\email[(Singh)]{anupamk18@gmail.com}
\address{Indian Institute of Science Education and Research Pune, Dr. Homi Bhabha Road, Pashan, Pune 411 008, India}

\thanks{}
\date{\today}
\subjclass[2020]{20G40, 20D06, 20G25}
\keywords{Word maps, Special linear group, local ring, Engel word, trace map}
\thanks{Roy is supported by an IISER Pune PhD Fellowship. Singh is funded by an ANRF-MATRICS Grant number ANRF/ARGM/2025/000095/MTR}
\begin{document}
\begin{abstract}
The study of word maps and Waring-like problems has been widely pursued for finite simple groups, algebraic groups, and Lie groups. In this article, we study Engel word maps $e_{m}(x, y) = \left[\cdots\left[[x, y], y \right], \cdots, y \right]$ on certain linear groups over local rings, namely, $\mathrm{SL}_2(\mathcal R)$ and $\mathrm{PSL}_2(\mathcal R)$. We consider the commutative ring $\mathcal {R} $ to be either a complete, local principal ideal ring $\mathcal O$, or a local principal ideal ring of finite length $\mathcal O_\ell$. Suppose the characteristic of the residue field $k\cong \mathbb F_q$ is $\neq 2$. Under some mild conditions on $q$, we show that there exists a constant $q_0(m)$, such that for all $q \geq q_0(m)$, all lifts in $\SL_2(\mathcal{O})$ of non-scalar elements of $\SL_2(k)$, are in the image of the $m$-th Engel word over $\SL_2(\mathcal{O})$. We further show that all Engel word maps are surjective on $\text{PSL}_2(\mathcal{O}_2)$ where $\mathcal{O}_2$ is a local principal ideal ring of length two. This work generalizes similar results about the Engel word map over fields. 
\end{abstract}

\maketitle

\section{Introduction}

In 1951, {\O}ystein Ore conjectured (see~\cite{Ore1951}) that for a non-Abelian finite simple group, every element is a single commutator. It led to several developments in group theory and related areas over the past half-century. Ore himself proved that every element of the alternating group $\mathcal A_n$  can be written as a single commutator. Over the years, through the work of several mathematicians, notably by R. C. Thompson \cite{Thompson1961, Thompson1962};  Gow and O. Bonten \cite{Bonten1993}; J. Neb$\Ddot{\text{u}}$ser, H. Pahlings, and E. Cleuvers \cite{NeubuserPahlingsCleuvers1984};  E. W. Ellers and N. L. Gordeev \cite{EllersGordeev1998} and many more, the conjecture was partly solved. It was completely settled in 2010 by Liebeck, O’Brien, Shalev, and Tiep \cite{LiebeckBrienShalev10}, who relied on the character theory of finite groups of Lie type and probabilistic methods. We refer to Malle's survey article~\cite{Malle14} for more details. 

The Ore's conjecture led to several significant developments in the study of images of \emph{word maps} on groups and other algebraic structures. Given a group $G$ and a word $w$, a member of a free group $\mathcal F_d$ with $d$ many generators, one can define a map $\widetilde{w} \colon G^d\rightarrow G$ by evaluation, which is called \emph{word map} on $G$ induced by the word $w$. One of the fundamental questions is to find the smallest positive integer $m$ such that every element of $G$ can be written as a product of $m$ many elements of $\widetilde{w}(G):=\widetilde{w}(G^d)$. This is also called the Waring-like problem. Over the past few decades, several foundational results have been proved for word maps over various groups, including linear algebraic groups, finite groups of Lie type, classical groups, and Lie groups (see \cite{borelclassfnconjcomm}, \cite{exppowerchatterjee}, \cite{powermappralay}, 
\cite{thomgoto}, \cite{kundusinghpower}, \cite{larsenshalevdistribution}, \cite{wordimagelobotzky},
\cite{panjasinghorthosymp}). For more details, we refer to the survey article \cite{surveygkp}. In 1983, Borel proved that the image of any non-trivial word over a connected semisimple linear algebraic group $\mathcal{G}$ is dominant, i.e., it contains a Zariski dense open subset. As a consequence, when $K$ is algebraically closed field, one gets that $w(G)^2=G$, where $G=\mathcal{G}(K)$; see \cite{Borel1983}. For finite simple groups, Shalev has shown that every element of such groups, when the size is large enough, is a product of two elements from the image when non-trivial (see~\cite{shalevwaring}). Similar to the Waring problem, the result for matrix algebra over a finite field is obtained in~\cite{kishoreSingh}. 

The surjectivity and, more generally, images of several specific words have been extensively studied for various groups, especially algebraic groups defined over a field, and some rings. In this article, we focus on the \emph{Engel words} defined on the special linear group $\mathrm{SL}_n(\mathcal O)$ where $\mathcal O$ is a (commutative) local ring (with identity). Let $\mathcal F_2 = \langle x, y\rangle$ be a free group generated by two elements. The $m$-th Engel word $e_m(x, y)$ in $\mathcal F_2$ is defined recursively as follows: $$e_1(x, y)=[x, y] = xyx^{-1}y^{-1}, \ e_2(x, y)=[[x, y], y]\ \ \text{and}\ \  e_{m}(x, y)=[e_{m-1}(x, y), y]$$ 
for $m\geq 2$. The induced $m$-th Engel word map is $e_{m}^{\mathcal{O}}\colon \mathrm{SL}_n(\mathcal{O}) \times \mathrm{SL}_n(\mathcal{O}) \rightarrow \mathrm{SL}_n(\mathcal{O})$ given by $(X, Y)\mapsto e_{m}^{\mathcal{O}}(X, Y)$, where $e_1^{\mathcal{O}}(X,Y) = XYX^{-1}Y^{-1}$ is simply the commutator map. In~\cite{Thompson1961}, Thompson proved surjectivity of the commutator map on $\mathrm{SL}_n(k)$ where $k$ is a field. In 2007, Shalev conjectured (Conjecture 2.8 and 2.9 \cite{Shalev2007}) that the $n$-th \emph{Engel word} is almost surjective for any non-abelian finite simple group. Bandman, Garion, and Grunewald (see \cite{BandmanGarionGrunewald2012}) showed that the $n$-th \emph{Engel word} is surjective for the groups $\mathrm{PSL}(2, q)$ and almost surjective for $\mathrm{SL}(2, q)$ provided $q$ is large enough. Following that, in~\cite{BaZa2016} and \cite{BNT2024} the surjectivity of \emph{Engel words} for $\mathrm{PSL}(2,\mathbb{C})$ and $\mathrm{PSL}(2, K)$ where $K$ is algebraically closed field, were studied. In~\cite{Gordeev}, Gordeev studied Engel words on algebraic groups of small ranks. In Section~\ref{known-results}, we list some known results in one place.

In this article, we are interested in looking at the group $\SL_2$ and $\rm{PSL}_2$ over certain rings. This problem has been studied for some rings, such as $p$-adic. Shalev in~\cite{Shalevzeta} conjectured (Conjecture 1.3) the following:  
\begin{conjecture}[Shalev]
Every element of $\mathrm{SL}_n(\mathbb{Z}_p)$ where $p$ is a prime and $n\geq 2$ ($p>3$ when $n=2$) is a commutator.
\end{conjecture}
\noindent Further, Shalev proposed a similar problem (Problem 1.4) for $\SL_n(\mathbb Z)$ and claimed that these would be hard problems. In response, Avni, Gelander, Kassabov, and Shalev (see~\cite{AvniGelanderKassabovShalev}) proved that every lift of non-scalars is a commutator in $\mathrm{SL}_n(\mathcal{O})$, for any local ring $\mathcal{O}$ with residue field $k$ that contains more than $n+1$ elements. An approach to solve this problem would be to lift solutions of the equation from the residue field where answers are known. In~\cite{PRS2026}, \cite{PanjaRoySingh2025}, the authors initiated a study to lift elements, such as regular elements and cyclic elements of $\mathrm{GL}_n(k)$ to $\mathrm{GL}_n(\mathcal{O})$ as images of certain maps. With this in mind, let us make our question more precise, which we try to answer in this article.  

\begin{question}
Let $\mathcal{O}$ be a local principal ideal ring with residue field $k\cong \mathbb{F}_q$ of odd characteristic.
\begin{enumerate}
\item Let $A\in \mathrm{SL}_2(\mathcal{O})$ which maps to $\bar A \in \SL_2(q)$. In~\cite[Theorem 5.1]{BandmanGarionGrunewald2012}, almost surjectivity (that is except $-I$) of Engel words over $\mathbb F_q$ is proved for large enough $q$ with an estimate of $q \geq q_0(m) = 2\cdot3^{2^{m+2}}$. Under the same assumption that $\bar A\in \mathrm{Im}(e_{m+1}^k)$ in $\mathrm{SL}_2(k)$, do we have $A\in  \mathrm{Im}(e_{m+1}^{\mathcal{O}}) $ in $\mathrm{SL}_2(\mathcal{O})$?
\item The surjectivity of the Engel word is known for $\rm{PSL}_2(q)$. Is it true that the Engel word is surjective on $\rm{PSL}_2(\mathcal O)$? 
\end{enumerate}
\end{question}
\noindent We answer these questions in this article. When $\mathcal{O}$ is a complete local principal ideal ring, we show that for large enough $q$ (odd), every non-scalar element that is in the image of $m$-th Engel word over $\SL_2(k)$ is in the image of the $m$-th Engel word in $\SL_2(\mathcal O)$. This is our Theorem~\ref{th: Engel second}. Further, when $q$ is large enough, we show (see Theorem~\ref{th: Engel PSL2}) that the Engel word map is surjective on $\text{PSL}_2(\mathcal O_2)$ where $\mathcal O_2$ is a local principal ring of length $2$. This result generalizes the same result proved by Bandman et. al. in~\cite{BandmanGarionGrunewald2012} over $\mathbb F_q$. 

The method for solving the lifting problem relies on dealing with several cases separately for regular semisimple elements and other cyclic elements for $\mathcal{O}_2$ and extending the results to each finite length level. Some of the key ideas are motivated by the methods used in \cite{AvniGelanderKassabovShalev} and the properties of the trace map in \cite{BandmanGarionGrunewald2012}. We hope these ideas will evolve further and help solve Shalev's conjecture mentioned above.

\subsection{Organization of the article}

We start with a survey of existing results on the surjectivity of Engel words on $\SL_2$ and $\text{PSL}_2$ over different fields and integral domains, in Section~\ref{not-conv}. In the same section, brief details about $\R$-conjugacy classes of elements of $\SL_2(\R)$ are described with their splitting into different families. A study of polynomials and lifting of roots of multivariable polynomials (similar to Hensel's Lemma) from $k$ to $\R$ in different characteristics has been introduced in Section~\ref{mult-lift}. In Section~\ref{tr}, we discuss the trace map of the $(m+1)$-th Engel word corresponding to general and special linear groups over local rings. The Section~\ref{commutator-local ring-length two} and \ref{towards Shalev's conj} mainly deal with commutators for special linear groups over $\mathcal{O}_2$ and $\mathcal{O}$.   In Section~\ref{Mag-Uni}, we discuss the idea of Magnus embedding for matrices over $\R$, and apply the same to triangular unipotent elements as elements of $\im(e_{m+1})$ in $\SL_2(\mathcal{O}_2)$. We discuss the coadjoint action of $\SL_2(k)$ on the dual Lie algebra $\mathfrak{sl}^*_2(k)$ in Section~\ref{coadjoint lie th}, which plays a central role along with the trace map for our lifting setup. The Section~\ref{cyclic as Engel} deals with the technical details of lifting non-scalar matrices as elements of $\im(e_{m+1})$ in $\SL_2(\mathcal{O})$. The proof of our main theorems is in Section~\ref{mainth}.

\subsection{Acknowledgement}
The authors would like to thank Prof. Boris Kunyavskii and Prof. Maneesh Thakur for their interest in this work.

\section{Preliminaries}{\label{not-conv}}

We begin with a brief survey of the results known particularly about the Engel word map on $\SL_2$ and $\text{PSL}_2$. 

\subsection{Known results on Engel word for $\SL_2$ and $\rm{PSL}_2$}\label{known-results}

In this section, we focus only on the surjectivity of the Engel words $e_m(x_1, x_2)$ on the linear group, as the problem that Engel words are surjective remains quite open. 

\begin{table}[htbp]\caption{Surjectivity of Engel words \( e_m \) on $\SL_2(q)$ and $\mathrm{PSL_2}(q)$}\label{tab:engel_survey1}
\centering
\small 
\begin{tabular}{@{} l c p{4cm} p{7cm} @{}}
\toprule
\textbf{Group} & \textbf{Condition on \( m \)} & \textbf{Condition on Field} & \textbf{Surjectivity Result} \\
\midrule
\( \text{PSL}(2, q) \) & \( m \leq 4 \) & All \( q \) (any prime power) & Surjective (Bandman et. al., see \cite{BandmanGarionGrunewald2012}) \\
\addlinespace
\( \text{PSL}(2, q) \) & \( m \geq 5 \) & \( q \geq q_0(m) \) (sufficiently large) & Surjective (Shalev conjectured for all \( q \neq 2, 3\); proved by Bandman et. al. see \cite{BandmanGarionGrunewald2012}) \\
\addlinespace
\( \text{SL}(2, q) \) & \( m \geq 1 \) &  \( q \geq q_0(m) \) (sufficiently large) & Almost surjective, i.e., onto \( \text{SL}(2, q) \setminus \{-id\} \) (due to Bandman et. al. see \cite{BandmanGarionGrunewald2012}) \\
\addlinespace
\( \text{SL}(2, q) \) & \( m =1 \) & All \( q \neq 2, 3 \) & Surjective (classical result on Ore's conjecture due to Thompson, see \cite{Thompson1961}). \\
\addlinespace
\bottomrule
\end{tabular}
\end{table}
\noindent Over a finite field, the results for surjectivity are tabulated in Table~\ref{tab:engel_survey1}. When $k$ is an algebraically closed field, the known result is listed in Table~\ref{tab:engel_survey2}, and it seems the problem is largely open for a general field, including $\mathbb R$ for both $\SL_2$ and $\mathrm{PSL_2}$ when $m\geq 2$. We must note down that for $\SL_2(k)$, where $k$ is any field $|k|> 3$ of characteristic $\neq 2$, $-I$ is a commutator if and only if $-1$ is a sum of two squares in $k$; see \cite[theorem 1]{Thompson1961}. This gives a general obstruction to the surjectivity of Engel words over various fields.   

\begin{table}[htbp]\caption{Survey of Surjectivity of Engel Words \( e_m \) on $\SL_2(K)$ and $\mathrm{PSL_2}(K)$}
\centering
\label{tab:engel_survey2}
\small 
\begin{tabular}{@{} l c p{3cm} p{7cm} @{}}
\toprule
\textbf{Group} & \textbf{Condition on \( m \)} & \textbf{Condition on Field} & \textbf{Surjectivity Result} \\
\midrule
\addlinespace
\( \text{SL}(2, k) \) & \( m \geq 2 \) & \(k=\bar k \) & seems unknown; see Klimenko et. al. \cite{kunvkacmoddy} \\
\addlinespace
\( \text{PSL}(2, k) \) &  \( m \geq 1 \) & \( k=\bar k \) & Surjective (see Bandman et al.~\cite{BaZa2016}, and Bien et al.~\cite{BNT2024}) \\
 

\bottomrule
\end{tabular}
\end{table}

Some known results over certain rings are briefly listed in Table~\ref{tab:engel_survey3}. When $R = \mathbb{Z}/p^{\ell}\mathbb{Z}$, commutators are well explored over $\SL_2(R)$ and $\mathrm{PSL}_2(R)$, but results regarding surjectivity of higher Engel words don't seem to be known.   

\begin{table}[htbp]\caption{Survey of Surjectivity of Engel Words \( e_m \) on $\SL_2(R)$ and $\mathrm{PSL_2}(R)$}
\centering
\label{tab:engel_survey3}
\small 
\begin{tabular}{@{} l c p{2cm} p{7cm} @{}}
\toprule
\textbf{Group} & \textbf{Condition on \( m \)} & \textbf{Ring} & \textbf{Surjectivity Result} \\
\midrule
\( \text{SL}(2, \mathbb{Z}_p) \) & \( m \geq 1 \) & \(p>3\) & For $m=1$ partial results are known, see Avni et. al.  \cite{AvniGelanderKassabovShalev}; for $m\geq 2$ seems unknown, see Kleminko et. al. \cite{kunvkacmoddy}. \\
\addlinespace
\( \text{PSL}(2, \mathbb{Z}_p) \) & \( m \geq 1 \) & \(p>3 \) & Surjective for $m=1$, due to Avni et. al.  \cite{AvniGelanderKassabovShalev}; seems unknown for $m\geq 2$, see \cite{kunvkacmoddy}. \\
\addlinespace
\( \text{SL}(2, \mathbb{Z}), \ \ \text{PSL}(2, \mathbb{Z}) \) & \( m \geq 1 \) & \(\mathbb{Z}\) & Not surjective for $m=1$ (hence for $m\geq 2$ as well) due to Ghosh et. al.~\cite{SL2ChenMeiri} \\
\bottomrule
\end{tabular}
\end{table}

\subsection{Notation and Conventions}
In this section, we recall some definitions and set the notation for the rest of the article.
A commutative ring $\mathcal{R}$ with unity is said to be \emph{complete with respect to an ideal $I$} if the canonical map $\mathcal{R} \rightarrow \varprojlim\limits_{j\geq 1} \mathcal{R}/I^j\mathcal{R}$ is an isomorphism.
We denote by $\mathcal{O}$ a local principal ideal ring that is complete with respect to its unique maximal ideal $\mathfrak{m} = (\pi)$. We assume that the residue field $k$ has characteristic not equal to $2$. For $\ell\geq 1$, we denote the quotient ring $\mathcal{O}_{\ell} = \mathcal{O}/\pi^{\ell}\mathcal{O}$. This is a local principal ideal ring of length $\ell$ with unique maximal ideal $\mathfrak{m}_\ell = (\pi_\ell)$, where $\pi_\ell = \pi + (\pi^\ell)$. The canonical map $\mathcal{O}_{\ell + 1} \rightarrow \mathcal{O}_{\ell}$ is denoted by $\theta_\ell$. The kernel of the natural surjection $\theta_j\colon \mathcal{O}_{j+1} \rightarrow \mathcal{O}_j$ is $ker(\theta_j) = \frac{\mathfrak{m}^j}{\mathfrak{m}^{j+1}}$. Therefore, $(ker(\theta_j))^2=0$ for each $j\geq 1$. Note that $\mathcal{O}_1 = \mathcal{O}/\mathfrak{m} \cong k$ and we simply denote $\theta_1 = \theta$.

For a local ring $\mathcal{A}$ with residue field $k$ and unique maximal ideal $\mathfrak{m}$, let $\M_2(\mathcal{A})$ denotes the set of all $2\times 2$ matrices with entries from $\mathcal{A}$. The general linear group $\GL_2(\mathcal{A})$ is the set of all elements $X\in \M_2(\mathcal{A})$ such that $\det(X) \in \mathcal{A}^{\times}$. The special linear group $\SL_2(\mathcal{A})$ is the set of all elements $X\in \M_2(\mathcal{A})$ with $\det(X)=1$. The quotient map $\theta \colon \mathcal{A} \rightarrow k =\mathcal{A}/ \mathfrak{m}$ induces a canonical map $\theta \colon \M_2(\mathcal{A}) \rightarrow \M_2(k)$ denoted by $X\mapsto \overline{X}$. Thus, we get a map $\theta\colon \GL_2(\mathcal{A}) \rightarrow \GL_2(k)$ and $\theta\colon \SL_2(\mathcal{A}) \rightarrow \SL_2(k)$. Further, note that the quotient map $\theta\colon \mathcal{A} \rightarrow k$ also induces a surjective map on the polynomial ring $\theta \colon \mathcal{A}[x_1, x_2, \ldots, x_n] \rightarrow k[x_1, x_2, \ldots, x_n]$ defined by reduction of the coefficients under $\theta$. By abuse of notation, all the above reduction maps are denoted simply by $\theta$, usually clear from the context. Similarly, all the reductions from $\mathcal{O}_{\ell+1}$ level to $\mathcal{O}_{\ell}$ level are denoted by $\theta_{\ell}$. 

The commutator word map $e_1^{\mathcal{A}} \colon \SL_2(\mathcal{A}) \times
\SL_2(\mathcal{A}) \rightarrow \SL_2(\mathcal{A})$ given by  $(X, Y)\mapsto XYX^{-1}Y^{-1}=[X,Y]$ induces a commutator word map at residue field level $\overline{e}_1 \colon \SL_2(k) \times \SL_2(k) \rightarrow \SL_2(k)$. The induced $(m+1)$-th Engel word map $e_{m+1}^{\mathcal{A}} \colon \SL_2(\mathcal{A}) \times
\SL_2(\mathcal{A}) \rightarrow \SL_2(\mathcal{A})$ is defined recursively as $e_{m+1}^{\mathcal{A}}(X,Y) = e_1^{\mathcal{A}}(e_{m}^{\mathcal{A}}(X,Y), Y)$ for $m\geq 1$. We denote $e_0^{\mathcal{A}}(X, Y)=X$. The reduced Engel word map over the field level is denoted by $\Bar{e}_{m+1}\colon \SL_2(k) \times \SL_2(k)\rightarrow \SL_2(k) $ which is defined in a similar way (thus $\Bar{e}_{m+1}$ is $e_{m+1}^k$). The notation $\im(e_{m+1}^{\mathcal{A}})$ denotes the image of the $(m+1)$-th Engel word $e_{m+1}^{\mathcal{A}}$ over the group $\SL_2(\mathcal{A})$. Therefore, $\im(\bar e_{m+1})$ means the image of the $m+1$-th Engel word $\bar e_{m+1}$ over the group $\SL_2(k)$. 

For a polynomial $F(t) = \Sigma_{i=0}^r a_i t^i\in \mathcal R[t]$, its formal derivative with respect to $t$ is denoted by the notation $\frac{df}{dt}$, which is given by $\Sigma_{i=0}^r ia_i t^{i-1}\in \mathcal R[t]$. In the same spirit, for a polynomial $F(x_1,x_2,\ldots, x_n)\in \mathcal R[x_1,x_2,\ldots,x_n]$ we denote the formal partial derivative of $F$ with respect to $x_i$ by the notation $\frac{\partial F}{\partial x_i}$, which is nothing but the formal derivative of $F$ with respect to $x_i$ considering $x_1,\ldots, x_{i-1}, x_{i+1},\ldots, x_n$ as constants. For a polynomial $f(x_1,x_2,\ldots, x_n)\in k[x_1,x_2,\ldots, x_n]$, the Jacobian matrix corresponding to $f$ is given by $J_f = \left[\begin{array}{cccccccc}
\frac{\partial f}{\partial x_1} & \frac{\partial f}{\partial x_2} & \cdots & \frac{\partial f}{\partial x_n} 
\end{array}\right]$. Moreover, $J_f|_{\alpha_0}$
 denotes the Jacobian evaluated at $\alpha_0$ for some $\alpha_0\in k^n$. A polynomial $F(t)\in \R[t]$ is said to be \emph{fundamental irreducible} if $\Bar{F}(t) = f(t)$ is irreducible in $k[t]$. 

An element $X\in \GL_2(\mathcal{A})$ is said to be \emph{regular semisimple} if $\overline{X}$ is regular semisimple in $\GL_n(k)$. For a ring $\mathcal R$ and two elements $A, B \in \M_2(\mathcal R)$, the notation $A \sim_{\mathcal R} B$ means that there exists $C\in \GL_2(\mathcal R)$ such that $A = CBC^{-1}$. The $\GL$-conjugacy class of $X\in \SL_2(\mathcal{A})$ is the set containing conjugates of $X$ by the elements of $\GL_2(\mathcal{A})$. For any $X\in \M_n(\mathcal{A})$ and $g\in \GL_n(\mathcal{A})$ the notation $X^g$ is used to denote $g^{-1}Xg$ and the notation $Ad_g(X)$ denotes $gXg^{-1}$. 
For a commutative ring $\mathcal R$ with unity and a polynomial $\mathfrak f(t)=t^2+ c_1t+c_0\in \mathcal R[t]$, the companion matrix $C_{\mathfrak f}\in\M_2(\mathcal R)$ of degree $2$ is the matrix 
$$\begin{pmatrix} 0 & -c_0 \\ 1 &  -c_1 \end{pmatrix}.$$
The zero matrix in $\M_2(\mathcal R)$ is denoted by a boldfaced zero $\mathbf 0$. The characteristic polynomial of $A \in \GL_2(\mathcal R)$ is denoted by the symbol $\chi(A)(t)$, which is $\det(tI-A)$. It is well known that square matrices over commutative rings with unity satisfy their characteristic equation. Moreover, for a commutative ring $\mathcal{R}$ with unity, $\tr(AB)= \tr(BA)$ for any two matrices $A, B\in \M_n(\mathcal{R})$. For any $r\in \mathbb{N}$, the $r$-th Engel word on $\SL_2(\mathcal{R})$ will be denoted by $e_r^{\mathcal{R}}$. Thus, the trace of the $r$-th Engel word map is denoted by $\tr(e_r(X, Y))$. 
\subsection{$\GL$-conjugacy classes of elements in $\SL_2(\R)$}{\label{O_2conj}}
Let $\R$ be a local principal ideal ring of length $2$. In this section, we provide a brief description of the $\GL$-conjugacy classes of the elements of $\SL_2(\R)$ by looking at their splitting into different types of conjugacy classes. A $\GL$-conjugacy class of an element $A\in \SL_2(\R)$ is the set $\{P^{-1}AP\mid P\in \GL_2(\R)\}$ and we denote this by $[A]_{\GL}$. 
\begin{example}
Consider two matrices $A=\left(\begin{array}{cc}
0  & -1 \\ 1 & 0 \end{array}\right)$ and $B=\left(\begin{array}{cc}   0  & 1 \\  -1 & 0
\end{array}\right)$ in $\SL_2(\R)$, where $\R$ has the residue field $\mathbb{R}$. These two matrices are conjugate in $\GL_2(\R)$. More specifically, we can take $P=\left(\begin{array}{cc} 1 & 0 \\  0 & -1 \end{array}\right)$ such that $P^{-1}BP=A$. But these two matrices are not conjugate in $\SL_2(\R)$. Because, if so, then the reduced matrices $\Bar{A}$ and $\Bar{B}$ must be conjugate by an element of $\SL_2(\mathbb R)$. If possible suppose we have $P=\left(\begin{array}{cc} a & b \\ c & d
\end{array}\right)\in \SL_2(\mathbb{R})$ such that $PAP^{-1}=B$. This implies $a^2 + b^2 = -1$, i.e. $x^2+y^2=-1$ has a solution in $\mathbb{R}^2$; which is a contradiction. 
\end{example}
\noindent Therefore in what follows, we consider the $\GL$-conjugacy classes of the elements of $\SL_2(\R)$ by seeing these inside $\GL_2(\R)$. That would be enough for our purpose as we don't really require exact $\SL_2(\R)$ conjugacy classes in $\SL_2(\R)$. 
 
Let $A$ be a matrix in $\SL_2(\R)$ such that there exists $\mathbf{g}_1$ and $\mathbf{g}_2$ in $\SL_2(\R)$ such that $e_1(\mathbf{g}_1, \mathbf{g}_2) = A$. Then for any $P\in\GL_2(\R)$ we have $e_1(\widehat{\mathbf{g}}_1, \widehat{\mathbf{g}}_2) = P^{-1}AP$, where $\widehat{\mathbf{g}}_i = P^{-1}\mathbf{g}_i P$ for $i=1, 2$. Note that $\widehat{\mathbf{g}}_1, \widehat{\mathbf{g}}_2$ are in $\SL_2(\R)$ again. Hence, an element of $\SL_2(\R)$ is in $\im(e_1)$ if and only if the whole $\GL$-conjugacy class of that element is in $\im(e_1)$. The same would apply to the images of any Engel word $e_m$.

Since, $\SL_2(\R)\unlhd\GL_2(\R)$, we consider the group action via conjugation defined by $$\GL_2(\R) \times\SL_2(\R) \rightarrow \SL_2(\R)$$ $$(\textbf{g}, X)\mapsto \textbf{g}\cdot X=\textbf{g}^{-1}X\textbf{g}.$$ 
Let us denote the orbit of $X$ with respect to the above action by $\GL_2(\R)\centerdot X$.  Therefore, we can write 
\begin{align*}
\SL_2(\R) = \bigcup\limits_{X\in\SL_2(\R)} \GL_2(\R)\cdot X = \bigcup\limits_{X \in \SL_2(\R)} \{ P^{-1} X P \mid P \in \GL_2(\R) \}.
\end{align*}
In \cite[page 1290]{BarringtonCliffWen2010}; Leigh, Cliff, and Wen have listed the similarity class information for matrices in $\GL_2(\mathcal{R})$, where $\mathcal{R} $ is a finite commutative local principal ideal ring with unity, with the residue field of odd characteristic. Their idea is based on constructing a subgroup chain that remains invariant under the action of $\GL_2(\R)$ by conjugation. The length of the chain is finite, which is critical. In fact, the same construction can be carried over to any local principal ideal ring of finite length and residue field with odd characteristic exponent. We make a note of the families of $\GL$-conjugacy classes of $\SL_2(\R)$ below: 
\begin{enumerate}
\item $S(\alpha)  : \left( \begin{array}{cc} \alpha & 0 \\
0 & \alpha\end{array} \right);  \alpha\in \R^{\times}; \alpha^2=1$
\item $D(\alpha,\delta,i)  :  \left( \begin{array}{cc} \alpha & 0 \\
0 & \delta\end{array} \right);   \alpha,\delta\in \R^{\times};\alpha-\delta\in \pi_2^i\R^{\times}; \alpha\delta=1;  0\leq i<2$
\item $H(\alpha,\beta,i) :  \left( \begin{array}{cc} \alpha & \pi_2^{i+1}\beta \\
\pi_2^i & \alpha\end{array} \right); \alpha\in \R^{\times},\beta\in \R/\pi_2^{1-i}\R; \alpha^2-\pi_2^{2i+1}\beta=1;  0\leq i<2 $
\item $H'(\alpha,\beta,i)  :  \left( \begin{array}{cc} \alpha & \pi_2^i\epsilon\beta \\
\pi_2^i\beta & \alpha\end{array} \right); \alpha\in \R,\beta\in \R^{\times}, \alpha^2-\epsilon\beta^2\pi_2^{2i}=1; 0\leq i<2; \epsilon$  is a non square unit in $\R$.
\end{enumerate}
Recall, $\mathfrak{m}_2=(\pi_2)$ is the unique maximal ideal of $\R$.

\begin{remark}\label{surjpower} When $k$ is algebraically closed, the \Cref{lift1} implies that \emph{any unit in $\R$ must be a square.} Therefore a fixed unit $\epsilon$ as above must be a square. In that case, $D(\alpha, \delta, 0)$ family and $H'(\alpha, \beta, 0)$ family coincides. Therefore, any regular semisimple element in $\SL_2(\R)$ must be split. 
\end{remark}

We explore the $\R$-similarity classes of elements of $\SL_2(\R)$ in more detail and present them more clearly.

\noindent\textbf{Family (1).} Take $A = \left(\begin{array}{cc} \alpha &  \\ & \alpha
\end{array}\right)\in \SL_2(\R)$. As $\theta(A)=\Bar{A}\in \SL_2(k)$,we have $\Bar{\alpha} = \pm 1$. Therefore $\alpha$ is of the form $1 + m_1$ or $-1 + m_2$ in $\R^{\times}$, where $m_1, m_2\in \mathfrak{m}_2$. When $\alpha = 1+m_1$ using $\alpha^2=1$, we obtain $m_1=0$. Similarly, for $\alpha = -1 + m_2$, we have $m_2 = 0$. Hence, in this case, the only possibilities for $A$ are $\pm I$.

\noindent\textbf{Family (2).} Let $A = \left(\begin{array}{cc} \alpha &  \\ & \delta
\end{array}\right)\in \SL_2(\R)$. \begin{enumerate}
\item $\mathbf{i=0}:$ In this case $\delta = \alpha^{-1}$ and $\alpha - \delta\in \R^{\times}$.  
\item $\mathbf{i=1}:$ In this case we have $\alpha - \delta\in \pi_2\R^{\times}$. From the residue field information we have $\alpha$ is of the form $1+ \pi_2 a_1$ or $-1 +  \pi_2 a_2$ for $a_1, a_2 \in \R$. Therefore, $\delta$ is of the form $1 + \pi_2(a_1 + v_1)$ or $-1 + \pi_2(a_2 + v_2)$ corresponding to $\alpha = 1+ \pi_2 a_1$ or $1 + \pi_2 a_2$ respectively, for some $v_1, v_2\in \R^{\times}$. Now from the $\alpha \delta = 1$ condition, one can obtain $\delta=1-\pi_2 a_1$ when $\alpha = 1 + \pi_2 a_1$ and $\delta = -1 - \pi_2 a_2$ when $\alpha = -1 + \pi_2 a_2$ respectively. Further, $\alpha - \delta \in \pi_2\R^{\times}$ implies that $a_1, a_2$ must be in $\R^{\times}$.
\end{enumerate}

\noindent\textbf{Family (3).} Let $A=\left( \begin{array}{cc} \alpha & \pi_2^{i+1}\beta \\
\pi_2^i & \alpha\end{array} \right)\in \SL_2(\R)$.
\begin{enumerate}
\item $\mathbf{i=0}:$ In this case there are two possibilities. (a) If $\beta$ is any non-unit then $\alpha = \pm 1$ from the determinant condition.
(b) If $\beta$ is a unit then a $\pi_2^{i+1} \beta \in \pi_2^{i+1}\R$ (see the additive isomorphism $\theta$ on page 1290, \cite{BarringtonCliffWen2010}). Therefore, we can write $\pi_2^{i+1}\beta = \pi_2^{i+1}v$ for some unit $v$ in $\R$. From the determinant condition, we get $\alpha^2 - \pi_2 v = 1$. Reducing this equation mod $\mathfrak{m}$, we get $\bar{\alpha} = \pm1$. Therefore $\alpha$ can be of the form $1 + m_1$ or $-1 + m_2$ where $m_1, m_2\in \mathfrak{m}$. In these cases we obtain $m_1 = 2^{-1}\pi_2 v$ and $m_2 = -2^{-1}\pi_2 v$.

\item $\mathbf{i=1}:$ In this case $\pi_2^{i+1} \beta = 0$. Therefore, the determinant condition gives $\alpha = \pm 1$.
\end{enumerate}

\noindent\textbf{Family (4).} Let $A=\left( \begin{array}{cc} \alpha & \pi_2^i \epsilon \beta \\ \pi_2^i\beta & \alpha \end{array} \right)\in \SL_2(\R)$. 
\begin{enumerate}
\item $\mathbf{i=0}:$ In this case $A=\left( \begin{array}{cc} \alpha & \epsilon\beta \\
\beta & \alpha\end{array} \right)$.
\item $\mathbf{i=1}:$ In this case the determinant condition gives $\alpha = \pm 1$ again.
\end{enumerate}
Based on the above arguments, we list the $\GL$-conjugacy classes of elements of $\SL_2(\R)$ explicitly in Table~\ref{tab:simplified}.

\begin{table}[ht]\caption{$\GL$-conjugacy classes of elements in $\SL_2(\R)$}\label{tab:GL-conjugacy-SL}
\centering
\resizebox{\textwidth}{!}{%
\renewcommand{\arraystretch}{1.8}
\begin{tabular}{|c|c|c|c|c|c|}
\hline
 & \textbf{$S({\alpha})$} & \textbf{$D(\alpha,\delta,i)$} & \multicolumn{2}{c|}{\textbf{$H(\alpha,\label{tab:simplified}\beta,i)$}} & \textbf{$H'(\alpha,\beta,i)$} \\
\hline
\multirow{2}{*}{i=0} & 
$\pm I$ &
$\begin{pmatrix}
\alpha & 0 \\
0 & \delta
\end{pmatrix}$ &
$\begin{pmatrix}
1 & 0 \\
1 & 1
\end{pmatrix}$ &
$\begin{pmatrix}
1+2^{-1}\pi_2 v & \pi_2 v \\
1 & 1+2^{-1}\pi_2 v
\end{pmatrix}$ &
$\begin{pmatrix}
\alpha & \epsilon\beta \\
\beta & \alpha
\end{pmatrix}$ \\
\cline{4-5}
 & & & $\begin{pmatrix}
-1 & 0 \\
1 & -1
\end{pmatrix}$ & $\begin{pmatrix}
-1-2^{-1}\pi_2 v & \pi_2 v \\
1 & -1-2^{-1}\pi_2 v
\end{pmatrix}$ & \\
\hline
\multirow{2}{*}{i=1} &
$\pm I$ &
$\begin{array}{@{}c|c@{}}
\begin{pmatrix}
1+\pi_2 a & 0 \\
0 & 1-\pi_2 a
\end{pmatrix} & \begin{pmatrix}
-1+\pi_2 a & 0 \\
0 & -1-\pi_2 a
\end{pmatrix}
\end{array}$ &
$\begin{pmatrix}
1 & 0 \\
\pi_2 & 1
\end{pmatrix}$ &
$\begin{pmatrix}
-1 & 0 \\
\pi_2 & -1
\end{pmatrix}$ &
$\begin{array}{@{}c|c@{}}
\begin{pmatrix}
1 & \pi_2\epsilon\beta \\
\pi_2\beta & 1
\end{pmatrix} & \begin{pmatrix}
-1 & \pi_2\epsilon\beta \\
\pi_2\beta & -1
\end{pmatrix}
\end{array}$ \\
\hline
\end{tabular}%
}
\end{table}

\section{Lifting a solution of multivariable polynomial over $\R$}{\label{mult-lift}}

In this section, $\R$ denotes a local principal ideal ring of length $2$ which is $\mathcal{O}/ \mathfrak{m}^2$. When the residue field $k$ is perfect, there is a complete classification of local principal ideal rings of length two with residue field $k$; see \cite[Theorem 2.1]{simlengthtwoamri}. Consider a field $k$ and a monic polynomial $f(t) \in k[t]$ in one variable. A root $\alpha_0$ of $f(t)$ is a \emph{simple root} if $$\frac{df}{dt}\Large\vert_{t=\alpha_0}\neq 0$$  where $\frac{df}{dt}$ denotes the first order formal derivative of $f(t)$ with respect to $t$. Similarly, one can define the notion of a simple root for multivariable polynomials over an arbitrary field $k$; see \cite{JPSERREArithmatic}, page 15.

\begin{definition}
Let $f(x_1,x_2,\ldots,x_n)\in k[x_1,x_2,\ldots,x_n]$ be a monic polynomial over the field $k$ and $\alpha$ be a root of $f$. Then, $\alpha$ is said to be a \emph{simple root} of $f(x_1,\ldots, x_n)$ if at least one of the formal partial derivatives $\frac{\partial f}{\partial x_i}|_{\alpha}$ is nonzero for some $1\leq i \leq n$. 
\end{definition}

Now, we study when a simple root of the reduction modulo $\mathfrak{m}$ of a monic polynomial in $\R[x_1,x_2,\ldots,x_n]$ lifts to a root of the original polynomial over $\R$.
\begin{lemma}\label{liftingsimple}
Let $\R$ be a local principal ideal ring of length two with residue field $k$ with characteristic $\neq 2$. Let $F(x_1,x_2,\ldots , x_n)\in \R[x_1,x_2,\ldots, x_n]$ be a monic polynomial such that the reduced polynomial $f(x_1,x_2,\ldots, x_n)\in k[x_1,x_2,\ldots, x_n]$ has a simple root $(s_1,s_2,\ldots, s_n)\in k^n$, then this simple root can be lifted to a root $(\hat s_1,\hat s_2,\ldots, \hat s_n) \in \R^n$ of the polynomial $F(x_1, x_2, \ldots, x_n)$ such that $\theta(\hat s_i) = s_i$ for all $i$. 
\end{lemma}
\begin{proof} Let $\sigma=(s_1, \ldots, s_n)$ be a simple root of $f(x_1, \ldots, x_n)$ in $k^n$. Without loss of generality, let $\frac{\partial f}{\partial x_i}|_{\sigma}\neq 0$ for some $i\leq r$ where $1\leq r\leq n$. We make separate cases depending on the field characteristic $k$. 

\noindent\textbf{$(1)$ When the characteristic of $k$ is zero.}   Take any lift $\check \sigma=(\check s_1, \check s_2, \ldots, \check s_n)$ of $\sigma$ in $\R^n$. We provide a choice of $(m_1,m_2,\ldots, m_n)\in \mathfrak{m}_2^n$ for $\check \sigma_0 = (\check s_1+m_1, \check s_2+m_2, \ldots, \check s_n + m_n)$ in $\R^n$ such that $F(\check \sigma_0) = 0$, where $m_i\in \mathfrak{m}_2$ for all $i$. Over any commutative ring $\mathcal{R}$ with unity, and for a polynomial $F(\textbf{x})$ in $n$-variables over $\mathcal{R}$, one can write $$F(\textbf{x} + \textbf{y}) = F(\textbf{x}) + \sum_{i=1}^{n} \frac{\partial F}{\partial x_i} y_i + \sum_{1 \leq i,j \leq n} C_{ij}(\textbf{x},\textbf{y}) y_i y_j$$
where $C_{ij}(\textbf{x}, \textbf{y})\in \mathcal{R}[\textbf{x}, \textbf{y}]$; $\mathbf{x}=(x_1,x_2,\ldots, x_n)$ and $\mathbf{y}=(y_1,y_2, \ldots,y_n)$.  Now $F(\check \sigma_0)=0$ implies $F(\check \sigma)+ \Sigma_{i=1}^r m_i\frac{\partial F}{\partial x_i}|_{\check \sigma} = 0$. As the reductions of $\frac{\partial F}{\partial x_i}|_{\check \sigma}$ are $\frac{\partial f}{\partial x_i}|_{\sigma}$ for $i=1,2, \ldots, r$; which are units in $k$ by assumption. Therefore $\frac{\partial F}{\partial X_i}|_{\check \sigma}$ must be units in $\R$ for $i=1,2, \ldots, r$. Now choose 
$$m_i = -F(\check \sigma)\left(r\frac{\partial F}{\partial x_i}|_{\check \sigma}\right)^{-1} \hspace{5pt} \text{for i=1,2, \ldots,r}$$ 
and for $i=r+1,r+2, \ldots, n$ any element of $\mathfrak{m}$ will work as a choice of $m_i$. This is because $\frac{\partial F}{\partial x_i}|_{\check \sigma}$ are non-units in $\R$ . Therefore, we get the required choice for $m_i$, and we take $\hat s_i= \check s_i+m_i$ for $i=1,2,\ldots,n$.

\noindent\textbf{$(2)$ When the characteristic of $k$ is $p$, an odd prime.} 

\begin{itemize}
\item \textbf{Case I} ($p\nmid r$) :
Choose $$m_i=-F(\check\sigma)\left(r\frac{\partial F}{\partial x_i}|_{\check\sigma} \right)^{-1} \hspace{5pt} \text{for i=1,2,\ldots,r}.$$ 
\item \textbf{Case II} ($p|r$) :
Choose $$m_i = -F(\check\sigma)\left(2(r-1) \frac{\partial F}{\partial x_i}|_{\check\sigma}\right)^{-1} \hspace{5pt} \text{for i=1,2,\ldots,r-1}, \quad m_r = -F(\check\sigma)\left(2\frac{\partial F}{\partial x_i}|_{\check\sigma}\right)^{-1}$$
and for $i=r+1, r+2, \ldots, n$ any element of $\mathfrak{m}$ will work as a choice of $m_i$.  Because then $\frac{\partial F}{\partial x_i}|_{\check\sigma}$ are non-units in $\R$ . Therefore, we get proper choices for $m_i$ and we take $\hat s_i=\check s_i+m_i$ for $i=1,2,\ldots,n$. 
\end{itemize}
\end{proof}

\begin{remark}{\label{rmk lifting}}
Since $(ker(\theta_j))^2=0$ for all $j\geq 1$, this method works for lifting simple roots of polynomial equations from the residue field to any $\mathcal{O}_{\ell}$ of any finite length $\ell\geq 2$. Hence, for a complete local principal ideal ring with residue field $k$ of characteristic $\neq 2$, this lifting strategy can be applied.
\end{remark}

\begin{corollary}\label{lift1}
Let $\R$ be a local principal ideal ring of length two with perfect residue field $k$ of characteristic $\neq 2$. Then, any $L$-th power map over $\R$, of which reduction with respect to $\theta$ is surjective over $k$, is also surjective over the unit group $\R^{\times}$ where $gcd(L, p')= 1$ and $p'$ is the characteristic exponent of $k$.
\end{corollary}
\begin{proof}
Let $\alpha$ be a unit in $\R$. Consider $F(t) = t^L - \alpha$. Since the reduced polynomial $f(t)=t^L-\Bar{\alpha}$ under $\theta$ must have a solution in $k$, let us denote a solution of it by $\beta_0$. As $\Bar{\alpha}$ must be a unit in $k$, therefore $\beta_0$ must be a unit. Hence $f'(\beta_0) = L \beta_0^{L-1}$ must be non-zero in $k$. Hence, by Lemma~\ref{liftingsimple} we get $F(t) = 0$ has a solution in $\R^{\times}$.
\end{proof}
\noindent An example of a ring $\R$ used in Corollary~\ref{lift1}, one can consider those $\R$ with a perfect residue field $k$, which is algebraically closed.


\section{The Trace Polynomial}{\label{tr}}

The trace polynomial corresponding to the commutator word on $\SL_2(k)$ plays a crucial role in various aspects when the characteristic of $k$ is zero. For a given $\mu\in k$, we have $\mathcal{M}_{\mu} = \{(x_1, x_2, x_3) \in k^3 \mid x_1^2 + x_2^2 + x_3^2-x_1x_2x_3 = \mu \}$, a Markoff surface in $\mathbb{A}_k^3$. Its vanishing polynomial equation is satisfied by commutators over $\SL_2(k)$ whose trace value is $\mu - 2$. In~\cite{AmitPeter2022}, there has also been a study related to integral points and the Hasse principle over the Markoff surface $\mathcal{M}_{\mu}$ when $\mu$ is a given integer. We need to deal with this surface (for Engel words) over $\R$ naturally. We would need to recover matrices in $\SL_2(\R)$ from the existence of $\R$ points (under certain conditions) on these kinds of surfaces defined over $\R$ by using given trace values of the $m+1$-th Engel word on $\SL_2(\R)$. Further, it will require translating the process at each $\mathcal{O}_{\ell}$ level for regular semisimple elements; we discuss this in more detail in Section~\ref{reg-sem}.

Let $K$ be a field and $\textbf{h}_1, \textbf{h}_2\in \SL_2(K)$. Cossey et. al.~\cite{CosseyMacdonaldStreet} proved that $\tr(\textbf{h}_1 \textbf{h}_2) + \tr(\textbf{h}_1 \textbf{h}_2^{-1}) = \tr(\textbf{h}_1) \tr(\textbf{h}_2)$. In fact, this holds true for $\SL_2(\mathcal{R})$, where $\mathcal{R}$ is any commutative ring with unity. This can be easily seen from the proof of~\cite[Lemma 5.2.1]{CosseyMacdonaldStreet}. 

Now we look at the expression of trace of $e_{m+1}^{\mathcal{R}}(A,B)_{\GL}$ where $A,B\in \GL_2(R)$. We use the suffix $\GL$ in the notation $e_{m+1}^{\mathcal{R}}(A,B)$ to denote that $A,B$ be any arbitrary matrices in $ \GL_2(\mathcal{R})$ and to distinguish the case from the case when $A,B$ are the members of $\SL_2(\mathcal{R})$ only.
\begin{lemma}{\label{traceengelinGL}}
    Let $\mathcal{R}$ be a commutative ring with unity and $A, B$ be two elements of $\GL_2(\mathcal{R})$. Then,  $$\tr(e_1^{\mathcal{R}}(A,B)_{\GL}) =[d_Bs_{\GL}^2+u_{\GL}^2+d_Av_{\GL}^2-s_{\GL}v_{\GL}u_{\GL}-2d_Ad_B]d_A^{-1}d_B^{-1}$$ and $$ \tr(e_{m+1}^{\mathcal{R}}(A,B)_{\GL})=[d_Bs_{m_{\GL}}^2+2v_{\GL}^2-s_{m_{\GL}}v_{\GL}^2-2d_B]d_B^{-1}\ \ \ \text{for} \ m\geq 1$$
where $s_{m_{\GL}}=\tr(e_m^{\mathcal{R}}(A,B)_{\GL})$. Moreover $s_{\GL}=\tr(A)$, $u_{\GL}=\tr(AB)$, $v_{\GL}=\tr(B)$, $d_A=\det(A),\ d_B=\det(B)$.
 
\end{lemma}
\begin{proof}
    Let $A,B\in \GL_2(\mathcal{R})$. Then $\chi(A)(A)=\mathbf{0}$ implies $A^2-\tr(A)A+\det(A)I=\mathbf{0}$. This implies $$A-\tr(A)I+\det(A)A^{-1}=\mathbf{0}.$$ Multiplying both side by $B$ from the left, we obtain: $$BA-\tr(A)B+\det(A)BA^{-1}=\mathbf{0}.$$ Taking trace of both side, we get $$\tr(BA^{-1})=[\tr(A)\tr(B)-\tr(AB)]\det(A)^{-1}.$$

Now $e_1(A,B)=(AB)(BA)^{-1}$. Let $X=AB$ and $Y=BA$. Therefore $\tr(e_1(A,B))=\tr(XY^{-1})$. By the above trace expression, we get $$\tr(e_1(A,B))=[\tr(X)\tr(Y)-\tr(XY)]det(Y)^{-1}.$$ Let $s_{\GL}=\tr(A)$, $u_{\GL}=\tr(AB)$, $v_{\GL}=\tr(B)$. Then $\tr(XY)=\tr(AB^2A)=\tr(A^2B^2)$. Again from the respective characteristic equations of $A$ and $B$ we have $A^2=\tr(A)A-\det(A)I$ and $B^2=\tr(B)B-\det(B)I$. Substituting these values in the expression $\tr(XY)=\tr(A^2B^2)$, we get $$\tr(XY)=s_{\GL}v_{\GL}u_{\GL}-d_Bs_{\GL}^2-d_Av_{\GL}^2+2d_Ad_B;$$ where $d_A=\det(A),\ d_B=\det(B)$. Therefore $$\tr(e_1^{\mathcal{R}}(A,B)_{\GL})=[d_Bs_{\GL}^2+u_{\GL}^2+d_Av_{\GL}^2-s_{\GL}v_{\GL}u_{\GL}-2d_Ad_B]d_A^{-1}d_B^{-1}.$$ For $m\geq 1$, $d_{e_m^{\mathcal{R}}(A,B)_{\GL}}=1$ and $\tr(e_m^{\mathcal{R}}(A,B)_{\GL}B)=v_{\GL}$. Therefore for $m\geq 1$;  $$ \tr(e_{m+1}^{\mathcal{R}}(A,B)_{\GL})=[d_Bs_{m_{\GL}}^2+2v_{\GL}^2-s_{m_{\GL}}v_{\GL}^2-2d_B]d_B^{-1}$$
where $s_{m_{\GL}}=\tr(e_m^{\mathcal{R}}(A,B)_{\GL})$.
\end{proof}

\begin{corollary}
Let $\mathcal{R}$ be a commutative ring with unity and $A_1, A_2$ be two elements of $\SL_2(\mathcal{R})$. Then,  $$\tr(e_1^{\mathcal{R}}(A_1,A_2)) = s^2 + u^2 + v^2 - suv - 2$$ 
where $s=\tr(A_1)$, $u=\tr(A_1A_2)$ and $v=\tr(A_2)$. 
 \end{corollary}
\begin{proof}       
Consider $A=A_1$ and $B=A_2$ in Lemma \ref{traceengelinGL}. Then $d_A=1=d_B$ and the result follows immediately.
\end{proof}
\noindent Therefore for any arbitrary $X, Y\in \SL_2(\R)$, $\tr(e_1^{\R}(X, Y))$ is a $3$-variable polynomial in $\R[s, u, v]$; where $s=\tr(X)$, $u=\tr(XY)$ and $v=\tr(Y)$ are the polynomials in the entries of $X$ and $Y$. Let us denote this by 
$$\Psi_1(s, u, v) = s^2 + u^2 + v^2 - suv - 2 \in \R[s,u,v].$$ 
This is called the \emph{trace polynomial} corresponding to the word $e_1$ over $\SL_2(\R)$.

Now take $s_1 = \Psi_1(s, u, v)$. Then, $$\tr(e_2^{\R}(X, Y))= \tr(e_1^{\R}(e_1^{\R}(X, Y), Y)) = s_1^2 + 2v^2 - s_1v^2 -2$$ as $\tr(e_1^{\R}(X, Y), Y)= \tr(Y) = v$. Now, denote $\tr(e_2(X, Y))$ by $\Psi_2(s_1, v)$. We can write $s_2=\Psi_2(s_1, v)$. Inductively, we note that 
$$s_{m+1} = \Psi_{m+1}(s_m, v) = \tr(e_{m+1}^{\R}(X, Y)) = s_m^2 + 2v^2 - s_mv^2 -2$$ 
for all $m\geq 1$. As $s_1$ is a function of $s, u, v$ and $s_{m+1}$ is a function of $s_m, v$ for all $m\geq 1$ therefore $\Psi_{m+1}(s_m, v)$ can be written as a function of $s, u, v$. Thus, it will be denoted by $\Psi_{m+1}(s, u, v)$ or $s_{m+1}(s, u, v)$. We use the notation $s_{m+1}(s_m,v)$ and $\Psi_{m+1}(s_m, v)$ interchangeably. The reduction of $\Psi_i$ on the residue field will be denoted by $\psi_i$. 

\begin{lemma}{\label{jacobian1}}
Let $\bar\lambda\in k$ and $\Bar{\lambda}\neq \pm 2$. Suppose $(\alpha_0, v_0)\in k^2$ is a solution of $\Tilde{\psi}_{m+1}(s_m, v):=\psi_{m+1}(s_m, v)- \bar \lambda =0$.  
Then, the Jacobian $$J_{\Tilde{\psi}_{m+1}}|_{(\alpha_0,v_0)}=\left[\begin{array}{ccc} \frac{\partial\Tilde{\psi}_{m+1}}{\partial s_m} & \frac{\partial\Tilde{\psi}_{m+1}}{\partial v}  \end{array}\right]\Big |_{(\alpha_0,v_0)}\neq 0.$$
\end{lemma}
\begin{proof}
For a fixed scalar $\bar\lambda \in k$, let  $\Tilde{\psi}_{m+1}(s_m, v) := \psi_{m+1}(s_m, v)- \bar \lambda = s_m^2 + 2v^2 - s_mv^2 - 2 - \Bar{\lambda}$ for $m\geq 1$. We have $(\alpha_0, v_0)\in k^2$ a solution of $\Tilde{\psi}_{m+1}(s_m, v) =0$. Further, let $(s_0, u_0, v_0)\in k^3$ that satisfies $s_m(s_0, u_0, v_0) = \alpha_0$. We prove the required result by taking a contrapositive. Let us check when the Jacobian is zero, and find the condition for $\Bar{\lambda}$. 

We have the following two equations:
\begin{enumerate}
\item $\frac{\partial\Tilde{\psi}_{m+1}}{\partial s_m}\Big |_{(\alpha_0,v_0)} = 0$ gives $(2s_m-v^2)\Big |_{(\alpha_0, v_0)} = 0$.
\item $\frac{\partial\Tilde{\psi}_{m+1}}{\partial v}|_{(\alpha_0, v_0)} = 0$ implies $\frac{\partial\Tilde{\psi}_{m+1}}{\partial s_m}\Big |_{(\alpha_0,v_0)} \cdot \frac{\partial s_m}{\partial v}\Big |_{(s_{m-1}(s_0, u_0, v_0), v_0)} + \frac{\partial\Tilde{\psi}_{m+1}}{\partial v}\Big |_{(\alpha_0, v_0)} = 0$ when $m\geq 2$, and 
\item[] $\frac{\partial\Tilde{\psi}_{m+1}}{\partial s_m}\Big |_{(\alpha_0,v_0)}\cdot \frac{\partial s_m}{\partial v}\Big |_{(s_{m-1}(s_0, u_0, v_0), u_0, v_0)} + \frac{\partial\Tilde{\psi}_{m+1}}{\partial v}\Big |_{(\alpha_0,v_0)} = 0$ when  $m=1$; $s_0(s_0,u_0,v_0)=s_0$.
\end{enumerate}
Thus, we need to solve the following equations:
\begin{enumerate}
\item $2\alpha_0-v_0^2=0$.
\item $(2\alpha_0-v_0^2)\cdot \frac{\partial s_m}{\partial v}\Big |_{(s_{m-1}(s_0,u_0,v_0),v_0)}+4v_0-2\alpha_0v_0=0\ \ \text{when}\ m\geq 2$, and 
\item[] $(2\alpha_0-v_0^2)\cdot \frac{\partial s_m}{\partial v}\Big |_{(s_{m-1}(s_0,u_0,v_0),u_0,v_0)}+4v_0-2\alpha_0v_0=0\ \ \text{when}\ m=1.$
\end{enumerate}
 
Now, by $(1)$ and $(2)$ we obtain $v_0(\alpha_0 - 2) = 0$. Therefore, either $v_0 = 0$ or $\alpha_0 = 2$. 
\begin{table}[ht]\caption{Values of the tuple $(\alpha_0,v_0,\Bar{\lambda})$}
\label{tab:product s_0u_0v_0}
\scriptsize
\centering
\renewcommand{\arraystretch}{1.8}\begin{tabular}{|l|l|l|}
\hline
$\alpha_0$ & $v_0$ & $\Bar{\lambda}$ \\ \hline
2    & 2  & 2 \\ \hline
2   & -2   & 2 \\ \hline
0   & 0   & -2 \\ \hline
\end{tabular}
\end{table}

\textbf{Case I:} When $\alpha_0 = 2$, we have $v_0 = \pm 2$. This implies $\Bar{\lambda}=2$.

\textbf{Case II:} Now, let $v_0 = 0$. Then, from the above equation, we get $\alpha_0 = 0$. In this case $\Bar{\lambda} = -2$. We list the choice for $(\alpha_0, v_0)$ in the Table~\ref{tab:product s_0u_0v_0}. Hence, when $\Bar{\lambda} \neq \pm 2$, Jacobian is nonzero, i.e., at least one of the partial derivatives $\frac{\partial\Tilde{\psi}_{m+1}}{\partial s_m}$ or $ \frac{\partial\Tilde{\psi}_{m+1}}{\partial v}$ must be nonzero at $(\alpha_0, v_0)$.
\end{proof}

\subsection{Jacobian condition for the trace equation corresponding to Engel words}\label{sub-trace}
In this section, we prove certain technical results required later.

\begin{lemma}{\label{tr 2}}
Let $A_0 = \left (\begin{array}{cc} a  & b \\ c & d \end{array}\right)$ and $B_0=\left (\begin{array}{cc} e & n \\
m & w  \end{array}\right )$ are in $\SL_2(k)$ such that $\tr(\Bar{e}_r(A_0, B_0)) = \lambda_0$ for some $\lambda_0\in k\backslash\{2\}$. Let $r\geq 2$ and  $$f_r(x_1, x_2, x_3, x_4, y_1, y_2, y_3, y_4) = \tr(\Bar{e}_r(x,y))-\lambda_0$$ be a polynomial in $8$ variables over $k$. Consider matrices $x=\left (\begin{array}{cc} x_1  & x_2 \\
        x_3 & x_4
\end{array}\right)$ and $y=\left (\begin{array}{cc}
       y_1  & y_2 \\
        y_3 & y_4
\end{array}\right)$ in $\SL_2(k)$ such that $$\tr(\Bar{e}_r(x,y))-\lambda_0=\Tilde{\psi}_r(s_{r-1},v)$$ where $s_{r-1} = \tr(\Bar{e}_{r-1}(x,y))$ is also a function of $s,u$ and $v$; where $s=x_1+x_4$, $u=x_1y_1+x_2y_3+x_3y_2+x_4y_4$, $v=y_1+y_4$. 
    
Then, the matrix $\left[\begin{array}{cccccccc}
\frac{\partial f_r}{\partial x_1} & \frac{\partial f_r}{\partial x_2} & \frac{\partial f_r}{\partial x_3} & \frac{\partial f_r}{\partial x_4} & \frac{\partial f_r}{\partial y_1} & \frac{\partial f_r}{\partial y_2} & \frac{\partial f_r}{\partial y_3} & \frac{\partial f_r}{\partial y_4}
\end{array}\right]\Big|_{\alpha_0} = \mathbf{0}$ implies $\left[\begin{array}{ccc}
\frac{\partial\Tilde{\psi}_r}{\partial s_{r-1}} & \frac{\partial\Tilde{\psi}_r}{\partial v}  \end{array}\right]\Big |_{(\gamma_{r-1},v_0)}=\mathbf{0}$; where $\alpha_0=(a,b,c,d,e,n,m,w)$ and $\beta_0=(s_0,u_0,v_0)$, $\gamma_0=(s_0,u_0)$, $\beta_0=(\gamma_0,v_0)$ and $\gamma_j=s_j(\gamma_{j-1},v_0)$ for $1\leq j\leq r-1$. Moreover, $s_0 = tr(A_0) = a+d$, $u_0 = tr(A_0B_0) = ae+bm+cn+dw$ and $v_0=tr(B_0)=e+w$.
\end{lemma}
\begin{proof} 
\textbf{Step I:}
First, we prove that, for any fixed $r\geq 2$
\begin{eqnarray*}
\left[\begin{array}{ccc}
    \frac{\partial\Tilde{\psi}_r}{\partial s} & \frac{\partial\Tilde{\psi}_r}{\partial u} & \frac{\partial\Tilde{\psi}_r}{\partial v}  \end{array}\right]\Big |_{\beta_0}=\mathbf{0}\implies \left[\begin{array}{ccc}\frac{\partial\Tilde{\psi}_r}{\partial s_1} & \frac{\partial\Tilde{\psi}_r}{\partial v}  \end{array}\right]\Big |_{(\gamma_1,v_0)}=\mathbf{0} \\\implies \left[\begin{array}{ccc}
    \frac{\partial\Tilde{\psi}_r}{\partial s_2} & \frac{\partial\Tilde{\psi}_r}{\partial v}  \end{array}\right]\Big |_{(\gamma_2,v_0)}=\mathbf{0}\implies\cdots\implies \left[\begin{array}{ccc}
    \frac{\partial\Tilde{\psi}_r}{\partial s_{r-1}} & \frac{\partial\Tilde{\psi}_r}{\partial v}  \end{array}\right]\Big |_{(\gamma_{r-1},v_0)}=\mathbf{0} ; 
    \end{eqnarray*}
where $\gamma_0=(s_0,u_0)$, $\beta_0=(\gamma_0,v_0)$ and $\gamma_j=s_j(\gamma_{j-1},v_0)$ for $1\leq j\leq r-1$. Here, the chain rule, analogous to the partial derivatives, is applicable. As $s_1$ is a function of $s,u,v$ therefore the condition $\left[\begin{array}{cccccccc}
  \frac{\partial\Tilde{\psi}_r}{\partial s} & \frac{\partial\Tilde{\psi}_r}{\partial u} & \frac{\partial\Tilde{\psi}_r}{\partial v}\end{array}\right]\Big|_{\beta_0}=\mathbf{0}$ gives the following equations:
\begin{equation}{\label{lemsimp i}}
\frac{\partial\Tilde{\psi}_r}{\partial s}|_{\beta_0}=\frac{\partial\Tilde{\psi}_r}{\partial s_1}|_{(\gamma_1,v_0)}\frac{\partial s_1}{\partial s}|_{\beta_0}=0\implies \frac{\partial\Tilde{\psi}_r}{\partial s_1}|_{(\gamma_1,v_0)}(2s_0-u_0v_0)=0 \end{equation}
 \begin{equation}{\label{lemsimp ii}}
   \frac{\partial\Tilde{\psi}_r}{\partial u}|_{\beta_0}=\frac{\partial\Tilde{\psi}_r}{\partial s_1}|_{(\gamma_1,v_0)}\frac{\partial s_1}{\partial u}|_{\beta_0}=0\implies \frac{\partial\Tilde{\psi}_r}{\partial s_1}|_{(\gamma_1,v_0)}(2u_0-s_0v_0)=0 
  \end{equation}
  \begin{equation}{\label{lemsimp iii}}
   \frac{\partial\Tilde{\psi}_r}{\partial v}|_{\beta_0}=\frac{\partial\Tilde{\psi}_r}{\partial v}|_{(\gamma_1,v_0)}=0.
\end{equation}
Claim: The partial derivative $\frac{\partial\Tilde{\psi}_r}{\partial s_1}|_{(\gamma_1,v_0)}= 0$. We prove this claim by assuming contrarily as follows: If possible, let $\frac{\partial\Tilde{\psi}_r}{\partial s_1}|_{(\gamma_1,v_0)}\neq 0$. Then by Equation~(\ref{lemsimp i}) and Equation~(\ref{lemsimp ii}) we obtain $s_0(4-v_0^2)=0$. Therefore either $v_0=\pm2$ or $s_0=0$. 
Here we recall the trace formula introduced in Section~\ref{tr}, viz., $s_1=\tr(\bar e_1(x,y))=s^2+u^2+v^2-suv-2$ and $s_{m+1}=\tr(\bar e_{m+1}(x,y))=s_m^2+2v^2-s_mv^2-2$ for $m\geq 1$.
  
\noindent \textbf{Case I.} Let $s_0 = 0$. This immediately implies $u_0 = 0$, by Equation~(\ref{lemsimp ii}). Then the trace formula gives $\gamma_1=s_1(\gamma_0)=v_0^2-2$ which further implies $$s_2(\gamma_1,v_0)=(v_0^2-2)^2+2v_0^2-(v_0^2-2)v_0^2-2=2.$$   Continuing this process, inductively we obtain $\lambda_0=s_r(\gamma_{r-1},v_0)=2$; which is a contradiction.

\noindent \textbf{Case II.} When $v_0=\pm 2$ then using Equation~(\ref{lemsimp i}) and Equation~(\ref{lemsimp ii}), we obtain the Table~\ref{tab: values 5 points}.
 \begin{table}[ht]
\scriptsize
\centering
\renewcommand{\arraystretch}{1.8}\begin{tabular}{|l|l|l|l|l|}
\hline
$s_0$ & $u_0$ & $v_0$ & $s_1$ & $\lambda_0$ \\ \hline
$u_0$   & $u_0$  & 2 & 2 & 2 \\ \hline
$-u_0$   & $u_0$   & -2 & 2 & 2 \\ \hline
\end{tabular}
\caption{Values of $(s_0,u_0,v_0,s_1,\lambda_0)$}
\label{tab: values 5 points}
\end{table}
In this case, by a similar argument to case I, we obtain $\lambda_0=2$, which again contradicts the assumption. Therefore, $\frac{\partial\Tilde{\psi}_r}{\partial s_1}|_{(\gamma_1,v_0)}=0$. Hence, the claim is true.

This claim and the Equation~(\ref{lemsimp iii}) together implies $\left[\begin{array}{ccc}
    \frac{\partial\Tilde{\psi}_r}{\partial s_1} & \frac{\partial\Tilde{\psi}_r}{\partial v}  \end{array}\right]\Big |_{(\gamma_1,v_0)}=\mathbf{0}$.

\textbf{Step II:} Secondly we prove that for $2\leq j\leq r-1$; 
\begin{eqnarray*}
\left[\begin{array}{ccc}
\frac{\partial\Tilde{\psi}_r}{\partial s_{j-1}} & \frac{\partial\Tilde{\psi}_r}{\partial v}  \end{array}\right]\Big |_{(\gamma_{j-1},v_0)}=\mathbf{0}\  \text{implies}\  \left[\begin{array}{ccc}
    \frac{\partial\Tilde{\psi}_r}{\partial s_j} & \frac{\partial\Tilde{\psi}_r}{\partial v}  \end{array}\right]\Big |_{(\gamma_j,v_0)}=\mathbf{0}.
     \end{eqnarray*}
Now the condition $\left[\begin{array}{ccc}
    \frac{\partial\Tilde{\psi}_r}{\partial s_{j-1}} & \frac{\partial\Tilde{\psi}_r}{\partial v}  \end{array}\right]\Big |_{(\gamma_{j-1},v_0)}=\mathbf{0}$
    gives the following equations:
    \begin{equation}{\label{lemsimp iv}}
    \frac{\partial\Tilde{\psi}_r}{\partial s_{j-1}}\big |_{(\gamma_{j-1},v_0)}=\frac{\partial\Tilde{\psi}_r}{\partial s_{j}}\big |_{(\gamma_j,v_0)}\frac{\partial s_{j}}{\partial s_{j-1}}\big |_{(\gamma_{j-1},v_0)}=0\implies \frac{\partial\Tilde{\psi}_r}{\partial s_{j}}\big |_{(\gamma_j,v_0)}(2s_{j-1}(\gamma_{j-2},v_0)-v_0^2)=0
  \end{equation}
\begin{equation}{\label{lemsimp v}}
\frac{\partial\Tilde{\psi}_r}{\partial v}\big |_{(\gamma_{j-1},v_0)} =0\implies \frac{\partial\Tilde{\psi}_r}{\partial v}\big |_{(\gamma_{j},v_0)}=0\implies\frac{\partial\Tilde{\psi}_r}{\partial s_{j-1}}\big|_{(\gamma_{j-1},v_0)}\frac{\partial s_{j-1}}{\partial v}\big|_{(\gamma_{j-2},v_0)}+ 2v_0(2-s_{r-1}(\gamma_{r-2},v_0))=0.
\end{equation}
By Equation~(\ref{lemsimp iv}) and Equation~(\ref{lemsimp v}) we have $v_0(2-s_{r-1}(\gamma_{r-2},v_0))=0$. This implies $v=0$ or $s_{r-1}(\gamma_{r-2},v_0)=2$. If possible let  $s_{r-1}(\gamma_{r-2},v_0)=2$. Then $\lambda_0=s_r(\gamma_{r-1},v_0)=2$; which is a contradiction. Therefore $v_0=0$. 
 
If possible let $v_0^2=2s_{j-1}(\gamma_{j-2},v_0)$ in Equation~(\ref{lemsimp iv}). Then we immediately get $s_{j-1}(\gamma_{j-2},v_0)=0$. This implies $s_j(\gamma_{j-1},v_0)=-2$. As $j\leq r-1$, therefore $\lambda_0=s_r(\gamma_{r-1},v_0)=2$ which is again a contradiction. Therefore the only possibility is $v_0^2\neq 2s_{j-1}(\gamma_{j-2},v_0)$. This immediately implies $\frac{\partial\Tilde{\psi}_r}{\partial s_{j}}\big |_{(\gamma_j,v_0)}=0$. Now, considering this together with the first part, it is easy to see that 
\begin{eqnarray*}
\left[\begin{array}{ccc}
\frac{\partial\Tilde{\psi}_r}{\partial s} & \frac{\partial\Tilde{\psi}_r}{\partial u} & \frac{\partial\Tilde{\psi}_r}{\partial v}  \end{array}\right]\Big |_{\beta_0}=\mathbf{0}\implies \left[\begin{array}{ccc}\frac{\partial\Tilde{\psi}_r}{\partial s_1} & \frac{\partial\Tilde{\psi}_r}{\partial v}  \end{array}\right]\Big |_{(\gamma_1,v_0)}=\mathbf{0} \implies\cdots\implies \left[\begin{array}{ccc}
\frac{\partial\Tilde{\psi}_r}{\partial s_{r-1}} & \frac{\partial\Tilde{\psi}_r}{\partial v}  \end{array}\right]\Big |_{(\gamma_{r-1},v_0)}=\mathbf{0}.
\end{eqnarray*}

\textbf{Step III:} As the final step, we prove that 
\begin{eqnarray*}
\left[\begin{array}{cccccccc}
\frac{\partial f}{\partial x_1} & \frac{\partial f}{\partial x_2} & \frac{\partial f}{\partial x_3} & \frac{\partial f}{\partial x_4} & \frac{\partial f}{\partial y_1} & \frac{\partial f}{\partial y_2} & \frac{\partial f}{\partial y_3} & \frac{\partial f}{\partial y_4}
\end{array}\right]\Big|_{\alpha_0} = \mathbf{0}\  \text{implies}\  \left[\begin{array}{cccccccc}
\frac{\partial\Tilde{\psi}_r}{\partial s} & \frac{\partial\Tilde{\psi}_r}{\partial u} & \frac{\partial\Tilde{\psi}_r}{\partial v}\end{array}\right]\Big|_{\beta_0}=\mathbf{0}.
\end{eqnarray*}
As $\Tilde{\psi}_r(s,u,v)=f(x_1,x_2,x_3,x_4,y_1,y_2,y_3,y_4)$ and $s,u,v$ all are the polynomials in $x_1,x_2,x_3,x_4,y_1,y_2,y_3,y_4$ therefore 
\begin{eqnarray*}
\frac{\partial f}{\partial x_i}=\frac{\partial f}{\partial s}\frac{\partial s}{\partial x_i}+\frac{\partial f}{\partial u}\frac{\partial u}{\partial x_i}+\frac{\partial f}{\partial v}\frac{\partial v}{\partial x_i} \ \text{and}\  \frac{\partial f}{\partial y_i}=\frac{\partial f}{\partial s}\frac{\partial s}{\partial y_i}+\frac{\partial f}{\partial u}\frac{\partial u}{\partial y_i}+\frac{\partial f}{\partial v}\frac{\partial v}{\partial y_i}\  \text{for}\  i=1,2,3,4.
       \end{eqnarray*}
This gives the following equations:
\begin{equation}{\label{lemsimp vi}}
\frac{\partial f}{\partial x_1}=\frac{\partial f}{\partial s}\frac{\partial s}{\partial x_1}+\frac{\partial f}{\partial u}\frac{\partial u}{\partial x_1}=\frac{\partial \Tilde{\psi}_r}{\partial s}+\frac{\partial \Tilde{\psi}_r}{\partial u}y_1 \implies \frac{\partial f}{\partial x_1}|_{\alpha_0}=\frac{\partial \Tilde{\psi}_r}{\partial s}|_{\beta_0}+e\frac{\partial \Tilde{\psi}_r}{\partial u}|_{\beta_0}. 
\end{equation} 
\begin{equation}{\label{lemsimp vii}}
\frac{\partial f}{\partial x_2}=\frac{\partial \Tilde{\psi}_r}{\partial u}y_3 \implies \frac{\partial f}{\partial x_2}|_{\alpha_0}=m\frac{\partial \Tilde{\psi}_r}{\partial u}|_{\beta_0}.
       \end{equation}
\begin{equation}{\label{lemsimp viii}}
\frac{\partial f}{\partial x_3}=\frac{\partial \Tilde{\psi}_r}{\partial u}y_2\implies \frac{\partial f}{\partial x_3}|_{\alpha_0}=n\frac{\partial \Tilde{\psi}_r}{\partial u}|_{\beta_0}.
\end{equation} 
\begin{equation}{\label{lemsimp ix}}
      \frac{\partial f}{\partial x_4}=\frac{\partial \Tilde{\psi}_r}{\partial s}+\frac{\partial \Tilde{\psi}_r}{\partial u}y_4\implies \frac{\partial f}{\partial x_4}|_{\alpha_0}=\frac{\partial \Tilde{\psi}_r}{\partial s}|_{\beta_0}+w\frac{\partial \Tilde{\psi}_r}{\partial u}|_{\beta_0}.
      \end{equation}
       \begin{equation}{\label{lemsimp x}}
       \frac{\partial f}{\partial y_1}=\frac{\partial f}{\partial v}\frac{\partial v}{\partial y_1}+\frac{\partial f}{\partial u}\frac{\partial u}{\partial y_1}=\frac{\partial \Tilde{\psi}_r}{\partial v}+\frac{\partial \Tilde{\psi}_r}{\partial u}x_1\implies \frac{\partial f}{\partial y_1}|_{\alpha_0}=\frac{\partial \Tilde{\psi}_r}{\partial v}|_{\beta_0}+a\frac{\partial \Tilde{\psi}_r}{\partial u}|_{\beta_0}
       \end{equation}
       \begin{equation}{\label{lemsimp xi}}
       \frac{\partial f}{\partial y_2}=\frac{\partial \Tilde{\psi}_r}{\partial u}x_3\implies \frac{\partial f}{\partial y_2}|_{\alpha_0}=c\frac{\partial \Tilde{\psi}_r}{\partial u}|_{\beta_0}. 
       \end{equation}
       \begin{equation}{\label{lemsimp xii}}
       \frac{\partial f}{\partial y_3}=\frac{\partial \Tilde{\psi}_r}{\partial u}x_2\implies \frac{\partial f}{\partial y_3}|_{\alpha_0}=b\frac{\partial \Tilde{\psi}_r}{\partial u}|_{\beta_0}.
       \end{equation}
       \begin{equation}{\label{lemsimp xiii}}
       \frac{\partial f}{\partial y_4}=\frac{\partial \Tilde{\psi}_r}{\partial v}+\frac{\partial \Tilde{\psi}_r}{\partial u}x_4\implies\frac{\partial f}{\partial y_4}|_{\alpha_0}=\frac{\partial \Tilde{\psi}_r}{\partial v}|_{\beta_0}+d\frac{\partial \Tilde{\psi}_r}{\partial u}|_{\beta_0}. 
       \end{equation}
Now, $\left[\begin{array}{cccccccc}
\frac{\partial f}{\partial x_1} & \frac{\partial f}{\partial x_2} & \frac{\partial f}{\partial x_3} & \frac{\partial f}{\partial x_4} & \frac{\partial f}{\partial y_1} & \frac{\partial f}{\partial y_2} & \frac{\partial f}{\partial y_3} & \frac{\partial f}{\partial y_4}
\end{array}\right]\Big|_{\alpha_0} = \mathbf{0}$ implies \begin{enumerate}
\item  $(e-w)\frac{\partial\Tilde{\psi}_r}{\partial u}|_{\beta_0}=0$ by  Equation~(\ref{lemsimp vi}) and Equation~(\ref{lemsimp ix}).
\item $(a-d)\frac{\partial\Tilde{\psi}_r}{\partial u}|_{\beta_0}=0$ by  Equation~(\ref{lemsimp x}) and Equation~(\ref{lemsimp xiii}).
\item $m\frac{\partial \Tilde{\psi}_r}{\partial u}|_{\beta_0}=0=n\frac{\partial \Tilde{\psi}_r}{\partial u}|_{\beta_0}$ by  Equation~(\ref{lemsimp vii}) and Equation~(\ref{lemsimp viii}).
\item $c\frac{\partial \Tilde{\psi}_r}{\partial u}|_{\beta_0}=0=b\frac{\partial \Tilde{\psi}_r}{\partial u}|_{\beta_0}$ by  Equation~(\ref{lemsimp xi}) and Equation~(\ref{lemsimp xii}).
\end{enumerate}
If possible let $\frac{\partial \Tilde{\psi}_r}{\partial u}|_{\beta_0}\neq 0$. Then $A_0=aI$ and $B_0=eI$. As $A_0,B_0$ are in $\SL_2(k)$ therefore these conditions give $a=\pm 1$ and $e=\pm 1$. Therefore $\Bar{e}_2(A_0,B_0)=I$ and $\lambda_0=2$ which is a contradiction. Therefore  $\frac{\partial \Tilde{\psi}_r}{\partial u}|_{\beta_0}= 0$. 
       
This and the condition
$$\left[\begin{array}{cccccccc}
\frac{\partial f}{\partial x_1} & \frac{\partial f}{\partial x_2} & \frac{\partial f}{\partial x_3} & \frac{\partial f}{\partial x_4} & \frac{\partial f}{\partial y_1} & \frac{\partial f}{\partial y_2} & \frac{\partial f}{\partial y_3} & \frac{\partial f}{\partial y_4}
\end{array}\right]\Big|_{\alpha_0} = \mathbf{0}$$  together with the Equation~(\ref{lemsimp vi}) and Equation~(\ref{lemsimp x}), implies that $\left[\begin{array}{cccccccc}
\frac{\partial\Tilde{\psi}_r}{\partial s} & \frac{\partial\Tilde{\psi}_r}{\partial u} & \frac{\partial\Tilde{\psi}_r}{\partial v}\end{array}\right]\Big|_{\beta_0}=\mathbf{0}$. Hence the lemma.
  \end{proof}

\begin{lemma}{\label{sum1}}
For $\ell\geq 1$, consider two matrices $A, B \in \GL_n(\mathcal{O}_{\ell+1})$ of the form $A = A_0 + \pi_{\ell+1}^{\ell} X$ and $B = B_0 + \pi_{\ell+1}^{\ell} Y$ for some $A_0, B_0\in \GL_n(\mathcal{O}_{\ell+1})$ and $X, Y$ in $\M_n(\mathcal{O}_{\ell+1})$. Then for every $r\geq 1$, the trace function corresponding to $e_{r+1}^{\mathcal{O}_{\ell+1}}(A,B)_{\GL}$ can be written as $$\tr(e_{r+1}^{\mathcal{O}_{\ell+1}}(A, B)_{\GL}) = \tr(e_{r+1}^{\mathcal{O}_{\ell+1}}(A_0, B_0)_{\GL}) + \pi_{\ell+1}^{\ell} L_{r+1}(X,Y)$$
where $L_{r+1}$ is a function $L_{r+1}\colon \M_n(\mathcal{O}_{\ell+1}) \times \M_n(\mathcal{O}_{\ell+1}) \rightarrow \mathcal{O}_{\ell+1}$ that satisfies $L_{r+1}(X_1+ X_2, Y_1 + Y_2) = L_{r+1}(X_1, Y_1) + L_{r+1}(X_2, Y_2)$ for all $X_1, X_2, Y_1, Y_2 \in M_n(\mathcal{O}_{\ell+1})$. 
\end{lemma}
\begin{proof}
In this proof we simply write $e_{r+1}^{\mathcal{O}_{\ell+1}}(A, B)$ instead of $e_{r+1}^{\mathcal{O}_{\ell+1}}(A, B)_{\GL}$ for the convenience of notation. We prove the result by using induction on $r$. Since $A = A_0 + \pi_{\ell+1}^{\ell} X$ and $B = B_0 + \pi_{\ell+1}^{\ell} Y$ are in $\GL_n(\mathcal{O}_{\ell+1})$, we get $A^{-1}=A_0^{-1}-\pi_{\ell+1}^{\ell} A_0^{-1}XA_0^{-1}$ and $B^{-1}=B_0^{-1}-\pi_{\ell+1}^{\ell} B_0^{-1}YB_0^{-1}$. Therefore,
\begin{eqnarray*}
e_1^{\mathcal{O}_{\ell+1}}(A,B) &=& (A_0 + \pi_{\ell+1}^{\ell} X)(B_0+\pi_{\ell+1}^{\ell} Y)(A_0^{-1}-\pi_{\ell+1}^{\ell}A_0^{-1}XA_0^{-1})(B_0^{-1}-\pi_{\ell+1}^{\ell} B_0^{-1}YB_0^{-1}) \\ 
&=& A_0B_0A_0^{-1}B_0^{-1}+\pi_{\ell+1}^{\ell}(XB_0A_0^{-1}B_0^{-1}+A_0YA_0^{-1}B_0^{-1}-A_0B_0A_0^{-1}XA_0^{-1}B_0^{-1}-A_0B_0A_0^{-1}B_0^{-1}YB_0^{-1})\\
&=& e_1^{\mathcal{O}_{\ell+1}}(A_0, B_0) + \pi_{\ell+1}^{\ell} Q_1(X,Y) 
\end{eqnarray*}
where $Q_1(X,Y) = XB_0A_0^{-1}B_0^{-1} + A_0YA_0^{-1}B_0^{-1} -A_0 B_0 A_0^{-1}XA_0^{-1} B_0^{-1} -A_0 B_0 A_0^{-1} B_0^{-1} Y B_0^{-1}$. It is easy to see 
\begin{equation}{\label{coordinate linearity of Q1}}
     Q_1(X_1+X_2,Y_1+Y_2)=Q_1(X_1,Y_1)+Q_1(X_2,Y_2)
\end{equation}
for all $X_1,X_2,Y_1,Y_2\in M_n(\mathcal{O}_{\ell+1})$. Now a similar calculation gives the following:
\begin{eqnarray*}
e_2^{\mathcal{O}_{\ell+1}}(A,B) = e_1^{\mathcal{O}_{\ell+1}}(e_1^{\mathcal{O}_{\ell+1}}(A,B),B) &=& e_1^{\mathcal{O}_{\ell+1}}(e_1^{\mathcal{O}_{\ell+1}}(A_0,B_0)+\pi_{\ell+1}^{\ell} Q_1(X,Y),B_0+\pi_{\ell+1}^{\ell} Y) \\ 
&=& e_2^{\mathcal{O}_{\ell+1}}(A_0,B_0)+\pi_{\ell+1}^{\ell} Q_2(X,Y),
\end{eqnarray*}
where 

\begin{eqnarray*}
&&Q_2(X,Y)=Q_1(X,Y)B_0e_1^{\mathcal{O}_{\ell+1}}(A_0,B_0)^{-1}B_0^{-1}+e_1^{\mathcal{O}_{\ell+1}}(A_0,B_0)Ye_1^{\mathcal{O}_{\ell+1}}(A_0,B_0)^{-1}B_0^{-1}- \\ &&e_1^{\mathcal{O}_{\ell+1}}(A_0,B_0)B_0e_1^{\mathcal{O}_{\ell+1}}(A_0,B_0)^{-1}Q_1(X,Y)e_1^{\mathcal{O}_{\ell+1}}(A_0,B_0)^{-1}B_0^{-1}-e_1^{\mathcal{O}_{\ell+1}}(A_0,B_0)B_0e_1^{\mathcal{O}_{\ell+1}}(A_0,B_0)^{-1}B_0^{-1}YB_0^{-1}.
\end{eqnarray*}
By using induction on $r$, first we prove the following statement: For all positive integer $r\geq 1$, $$Q_{r+1}(X_1+X_2,Y_1+Y_2)=Q_{r+1}(X_1,Y_1)+Q_{r+1}(X_2,Y_2)$$ for all $X_1,X_2,Y_1,Y_2\in M_n(\mathcal{O}_{\ell+1})$.
    
\noindent \textbf{Base Case} $\mathbf{(r=1)}.$
Equation~(\ref{coordinate linearity of Q1}) together with the expression of $Q_2(X,Y)$ implies $$Q_2(X_1+X_2,Y_1+Y_2)=Q_2(X_1,Y_1)+Q_2(X_2,Y_2)$$ for all $X_1,X_2,Y_1,Y_2\in M_n(\mathcal{O}_{\ell+1})$. Hence, the base case is true.

\textbf{Induction Hypothesis.} Let the statement is true for $r= 2, 3,\ldots, j-1$.

\textbf{Inductive step} $\mathbf{(r=j)}.$ In this case, one can write 
\begin{eqnarray}{\label{expression ej+1}}
e_{j+1}^{\mathcal{O}_{\ell+1}}(A,B)= e_1^{\mathcal{O}_{\ell+1}}(e_j^{\mathcal{O}_{\ell+1}}(A,B),B)          &=& e_1^{\mathcal{O}_{\ell+1}}(e_j^{\mathcal{O}_{\ell+1}}(A_0,B_0)+\pi_{\ell+1}^{\ell} Q_j(X,Y),B_0+\pi_{\ell+1}^{\ell} Y) \\ \nonumber
 &=& e_{j+1}^{\mathcal{O}_{\ell+1}}(A_0,B_0)+\pi_{\ell+1}^{\ell} Q_{j+1}(X,Y),
\end{eqnarray}
where 
\begin{eqnarray*}
&&Q_{j+1}(X,Y)=Q_j(X,Y)B_0e_j^{\mathcal{O}_{\ell+1}}(A_0,B_0)^{-1}B_0^{-1}+e_j^{\mathcal{O}_{\ell+1}}(A_0,B_0)Ye_j(A_0,B_0)^{-1}B_0^{-1}-\\
&& e_j^{\mathcal{O}_{\ell+1}}(A_0,B_0)B_0e_j^{\mathcal{O}_{\ell+1}}(A_0,B_0)^{-1}Q_j(X,Y)e_j^{\mathcal{O}_{\ell+1}}(A_0,B_0)^{-1}B_0^{-1}-e_j^{\mathcal{O}_{\ell+1}}(A_0,B_0)B_0e_j^{\mathcal{O}_{\ell+1}}(A_0,B_0)^{-1}B_0^{-1}YB_0^{-1}.
      \end{eqnarray*}
By the induction hypothesis we have $$Q_j(X_1+X_2,Y_1+Y_2) = Q_j(X_1,Y_1) + Q_j(X_2,Y_2)$$
for all $X_1, X_2, Y_1, Y_2\in M_n(\mathcal{O}_{\ell+1})$. This, together with the expression of $ Q_{j+1}(X, Y)$, implies that the statement holds for $r=j$ and hence the statement is true for all $r\geq 1$.

Consider $L_{r+1}(X, Y) = \tr(Q_{r+1}(X, Y))$ for any $r\geq 1$. In the other words, for fixed $A_0, B_0\in \GL_n(\mathcal{O}_{\ell+1})$, we have the map 
$$L_{r+1}\colon \M_n(\mathcal{O}_{\ell+1}) \times \M_n(\mathcal{O}_{\ell+1}) \rightarrow \mathcal{O}_{\ell+1}$$ 
given by $(X,Y) \mapsto \tr(Q_{r+1}(X,Y))$. Therefore for any $r\geq 1$, $L_{r+1}(X_1+ X_2,Y_1 + Y_2) = L_{r+1}(X_1, Y_1) + L_{r+1}(X_2, Y_2)$ for all $X_1, X_2, Y_1, Y_2\in M_n(\mathcal{O}_{\ell+1})$. This and Equation~(\ref{expression ej+1}) together complete the proof.
\end{proof}

\begin{corollary}{\label{scl}}
Suppose $X= \mu_1 A_0$ and $Y=\mu_2B_0$ for some fixed $\mu_1, \mu_2\in \mathcal{O}_{\ell+1}$ $(\ell\geq 1)$. Then, $Q_{r+1}(X,Y)=\mathbf{0}$, hence $L_{r+1}(X,Y)=\tr(Q_{r+1}(X,Y))=0$.
\end{corollary}
\begin{proof}
By using the expression, we can verify that $Q_1(X, Y) = \mathbf{0}$. Therefore $Q_2(X,Y) = \mu_2e_2^{\mathcal{O}_{\ell+1}}(A_0, B_0)-\mu_2e_2^{\mathcal{O}_{\ell+1}}(A_0, B_0) = \mathbf{0}$. Continuing this process inductively for any $r\geq 1$ we can show $Q_{r+1}(X,Y)=\mathbf{0}$, hence $L_{r+1}(X,Y)=\tr(Q_{r+1}(X,Y))=0$.
\end{proof}

Now, we introduce a specific subgroup of $\GL_2(\mathcal{A})$ and the trace expression of the $m+1$-th Engel word over this subgroup, which will be used further to prove the main result of Section~\ref{reg-sem}. 

\subsection{A very special subgroup of $\GL_2(\mathcal{A})$.}{\label{trace in GL}} Let $\mathcal{A}$ be any local ring with residue field $k$ of characteristic other than $2$. Consider the set $\mathfrak{H}=\{P\in \GL_2(\mathcal{A})\mid\ \theta(P)\in \SL_2(k)\}$. Clearly, with respect to matrix multiplication, $\mathfrak{H}$ forms a subgroup of $\GL_2(\mathcal{A})$. Moreover, the reduction map $\mathfrak{H}\rightarrow \SL_2(k)$ is a surjection.

\noindent\textbf{Expression of $\tr(e_{m+1}(X,Y)_{\mathfrak{H}})$ in terms of the entries of $X,Y$ when $X,Y\in \mathfrak{H}$ are matrix variables:} Let $X=\left (\begin{array}{cc}
     x_1& x_2 \\
     x_3&x_4 
\end{array}\right )$ and $Y=\left (\begin{array}{cc}
     y_1&y_2  \\
     y_3&y_4 
\end{array}\right )$ be two matrix variables from $\mathfrak{H}$. As $s_{1_{\mathfrak{H}}}=\tr(e_1(X,Y)_{\mathfrak{H}})$ is a function of $s_{\mathfrak{H}}=\tr(X)$, $u_{\mathfrak{H}}=\tr(XY)$, $v_{\mathfrak{H}}=\tr(Y)$ and all these three are the functions in the entries of $X,Y$ therefore one can think $\tr(e_{m+1}(X,Y)_{\mathfrak{H}})$ as a function in the entries of $X,Y$ when $X,Y$ are matrix variables from $\mathfrak{H}$. Now for any given $\lambda\in \mathcal{A}$, consider  $\Phi_{m+1}(s_{m_{\mathfrak{H}}},v_{\mathfrak{H}}, d_{Y})=\tr(e_{m+1}(X,Y)_{\mathfrak{H}})-\lambda$. Using this expression together with Lemma~\ref{traceengelinGL} we get $$d_Y\Phi_{m+1}(s_{m_{\mathfrak{H}}},v_{\mathfrak{H}}, d_{Y})=d_Ys_{m_{\mathfrak{H}}}^2+2v_{\mathfrak{H}}^2-s_{m_{\mathfrak{H}}}v_{\mathfrak{H}}^2-2d_Y-\lambda d_Y.$$ Substituting the expressions of $s_{m_{\mathfrak{H}}}$ successively we obtain $$d_X^{\mu_1} d_Y^{\mu_2}\Phi_{m+1}(s_{m_{\mathfrak{H}}},v_{\mathfrak{H}}, d_{Y})=\text{Polynomial in}\ x_1,x_2,x_3,x_4,y_1,y_2,y_3,y_4$$ where $\mu_1,\mu_2$ are positive integers and $s_{m_{\mathfrak{H}}}=\tr(e_{m+1}(X,Y)_{\mathfrak{H}})$ for $m\geq 1$.

We denote this polynomial (in the right hand side of the above equation) by $\widetilde F_{m+1}^{\mathcal{A}}(x_1,x_2,x_3,x_4,y_1,y_2,y_3,y_4)$ and $d_X^{\mu_1} d_Y^{\mu_2}\Phi_{m+1}(s_{m_{\mathfrak{H}}},v_{\mathfrak{H}}, d_{Y})$ by $\widetilde \Phi(s_{m_{\mathfrak{H}}},v_{\mathfrak{H}}, d_{Y},d_X)$. Therefore if we consider the equation $\Phi_{m+1}(s_{m_{\mathfrak{H}}},v_{\mathfrak{H}}, d_{Y})=0$, it is equivalent to consider the equation $$\widetilde F_{m+1}^{\mathcal{A}}(x_1,x_2,x_3,x_4,y_1,y_2,y_3,y_4)=0$$ in terms of the entries of $X$ and $Y$. As $\theta (X),\theta(Y)\in \SL_2(k)$ therefore $d_Y,d_X$  always take values in $\mathcal{A}^{\times}$.

\section{Lift of Scalars as Commutators in $\SL_2(\R)$}{\label{commutator-local ring-length two}}

In this section, we look at the lift of $\pm I \in \SL_2(k)$ to $\SL_2(\R)$ if they are commutators. Recall, $\R$ is a local principal ideal ring of length $2$. An element $A \in \GL_2(\R)$ is said to be regular semisimple if its reduction under $\theta$, i.e., $\Bar{A}$, is regular semisimple. That is, the characteristic polynomial of $\Bar{A}$ is separable over $k$. If $A \in \SL_2(\R)$ is regular semisimple, then $\tr(\Bar{A})$ cannot be $\pm 2$. This is because, if it is $\pm 2$ then the characteristic polynomial of $\bar A$ will be $(x \pm 1)^2$, which means $\bar A$ has a repeated eigenvalue. We note that the smallest degree monic polynomial that annihilates $A$ is said to be a \emph{minimal polynomial} for $A$. It turns out that (see \cite[section 3]{PanjaRoySingh2025}), for a regular semisimple $A$, its minimal polynomial is uniquely determined, which is equal to the characteristic polynomial $\chi(A)(t)$.

\begin{lemma}{\label{sim1}}
Let $A, B$ be two regular semisimple elements in $\GL_2(\R)$ such that $B = TAT^{-1}$ for some $T\in \GL_2(\R)$. Let $x - \alpha$ be a divisor of the characteristic polynomial of $A$ where $\alpha \in \R^{\times}$. Then, for each $s\in \R^{\times}$ there exists a matrix $S\in \GL_2(\R)$ with $det(S)=s$ such that $B = SAS^{-1}$.
\end{lemma}
\begin{proof}
As $A\sim _{\R}  B$, both regular semisimples, by $\R$-conjugacy family information described in Table~\ref{tab:simplified} and by Lemma 2.8 \cite{PanjaRoySingh2025} there exists $T_1, T_2 \in \GL_2(\R)$ such that  $T_1AT_1^{-1} = T_2BT_2^{-1} = \left(\begin{array}{cc}  \alpha & 0  \\ 0 & \beta \end{array}\right)$ for some $\beta\in \R^{\times}$. Let us take $s_0 = s (det(T_2 T_1^{-1}))$ and consider $S_0 = \left(\begin{array}{cc} s_0 & 0  \\ 0 & 1    \end{array} \right)$. Therefore, $S_0 T_1 A T_1^{-1} S_0^{-1} = T_1 A T_1^{-1} = T_2 B T_2^{-1} \implies T_2^{-1} S_0 T_1 A T_1^{-1} S_0^{-1} T_2 = B$. The element $S = T_2^{-1}S_0T_1$ works. 
\end{proof}

\begin{remark}{\label{existence of delta not}}
When $|k|\geq 5$ of characteristic $\neq 2$ there always exists some $\delta \in k$ such that $\delta^2\neq 1,0$. Pick any element from $k^{\times}\backslash\{\pm 1\}$ and denote it by $\delta$. Take a lift say, $\Delta_0$ of $\delta$ in $\R^{\times}$. Then, $\Delta_0\neq \Delta_0^{-1}$. Because otherwise $\delta^2$ will be $1$, which is a contradiction.
\end{remark}

\begin{lemma}{\label{lift I}}
For $\beta\in \R^{\times}$ and $\epsilon$ a non square unit in $\R$, the element $ \left(\begin{array}{cc} 1 & \pi_2\epsilon\beta \\  \pi_2\beta &1 \end{array}\right)$ is a commutator, i.e., it belongs to $\im(e_1^{\R})$ in $\SL_2(\R)$. 
\end{lemma}
\begin{proof}
Consider  $A= \left(\begin{array}{cc} 1 & \pi_2\epsilon\beta \\  \pi_2\beta & 1 \end{array} \right)$ where $\epsilon$ is a non-square unit and $\beta$ is an unit in $\R$. Take, $\lambda\in \R^{\times}$ such that $\lambda^2\neq 1$ and consider $D = \left(\begin{array}{cc} \Delta_0  &  \pi_2 \\ 0& \Delta_0^{-1} \end{array}\right)$. There exists  $\Delta_0 \in \R^{\times}$ such that $\bar \Delta_0 \neq \bar\Delta_0^{-1}$ by Remark~\ref{existence of delta not}. Now, $D$ and $AD$ are regular semisimples with the same characteristic polynomial, which has a factor $t -\Delta_0$. Therefore, $AD \sim_{\R} D$ by the $\R$-conjugacy class information listed in Table~\ref{tab:simplified} and Lemma 2.8 \cite{PanjaRoySingh2025}. Hence invoking Lemma~\ref{sim1}, there exists $S\in \SL_2(\R)$ such that $A=e_1^{\R}(S^{-1}, D)$.
\end{proof}

The following is a well-known result in the field case, which is true more generally. 
\begin{proposition}{\label{minus Id}}
The element $-I$ is a commutator, i.e., it belongs to $\im(e_1^{\R})$ in $\SL_2(\R)$ if and only if $-1$ can be written as a sum of two squares in the residue field $k$ of $\R$.
\end{proposition}
\begin{proof}
Suppose we have $A_0, B_0\in \SL_2(\R)$ such that $[A_0, B_0]= A_0B_0A_0^{-1}B_0^{-1} =-I$. This implies $A_0B_0A_0^{-1}=-B_0$, hence, $\tr(B_0) = 0$ since characteristic of $k\neq 2$. Hence, $B_0$ satisfies its characteristic polynomial $x^2 + 1$, thus, $B_0^2 = -I$. Over $k$, this gives $\Bar{B}_0^2 = -I$ and $\Bar{B}_0 = P \left(\begin{array}{cc}  0 & -1 \\ 1  & 0 \end{array}\right)P^{-1}$ for some $P\in GL_2(k)$. Let $P^{-1}\Bar{A}_0P= \left(\begin{array}{cc}  a_0 & b_0 \\  c_0  & d_0
\end{array}\right)$. Then $\Bar{A}_0 \Bar{B}_0 = -\Bar{B}_0 \Bar{A}_0$ implies $(P^{-1}\Bar{A}_0P)\left(\begin{array}{cc}
    0 & -1 \\
   1  & 0
\end{array}\right)=-\left(\begin{array}{cc}
    0 & -1 \\
   1  & 0
\end{array}\right)(P^{-1}\Bar{A}_0P)$. Solving this equation together with the condition $det(\Bar{A}_0)=1$ we get $a_0^2 + b_0^2 = -1$ in $k$. 

Conversely, let there exist $a, b$ in $k$ such that $a^2 + b^2 = -1$. Consider the polynomial $F(x, y)= x^2 + y^2 + 1$ in $ \R[x,y]$. Then the reduced polynomial $f$ over $k$ has a root $(a, b) \in k^2$. Since $a^2+ b^2 = -1$, therefore at least one of $ a$ and $ b$ must be nonzero. Hence, at least one of $\frac{\partial f}{\partial x}|_{(a,b)}, \frac{\partial f}{\partial y}|_{(a,b)}$ must be nonzero in $k$ as $2$ is unit in $k$. Thus, $(a, b)$ is a simple root of $f(x, y) = x^2 + y^2 + 1$ in $k^2$. By taking $n=2$ in Lemma~\ref{liftingsimple} we get $F(x, y) = 0$ has a solution $(\alpha_0, \beta_0) \in \R$ such that $\bar{\alpha}_0 = a$ and $\bar{\beta}_0 = b$. This gives, $\alpha_0^2 + \beta_0^2 = -1$ in $\R$. Now consider the matrices $Q_1=\left(\begin{array}{cc} 0 &  1\\ -1 & 0 \end{array}\right)$ and $Q_2 = \left(\begin{array}{cc} \alpha_0 &  \beta_0\\
\beta_0 & -\alpha_0 \end{array}\right)$ in $\SL_2(\R)$, then $[Q_1, Q_2] = -I$.
\end{proof}

\begin{corollary}
If $\R$ is finite with residue field of odd characteristic, then $-I$ must be in $\im(e_1^{\R})$. (This is because the norm map on the residue field level is surjective, which is a finite field).
\end{corollary}

Now we are ready to talk about the lifts of $I$ in $\SL_2(\mathcal{O}_2)$ being commutator. Similar questions for the lifts of $-I$ will be discussed in Section~\ref{towards Shalev's conj}.  
\begin{proposition}\label{prop:lift of I}
Let $\R$ be a local principal ideal ring of length two with residue field $k$, perfect of odd characteristic, and $|k|\geq 5$. Then, any lift of $I$ in $\SL_2(\mathcal{O}_2)$ must be a commutator in $\SL_2(\mathcal{O}_2)$. 
\end{proposition}
\begin{proof}
All possible lifts of scalars are listed in row $2$ corresponding to $i = 1$ in Table~\ref{tab:simplified}, up to $\GL_2(\R)$ conjugacy. The matrix $ \left(\begin{array}{cc}
1 & \pi_2\epsilon\beta \\ \pi_2\beta & 1 
\end{array}\right)\in \im(e_1^{\R})$ by Lemma~\ref{lift I}. Now let us take $a_1 = -1$ and $a_2 = 1+ \pi_2$ in $\R^{\times}$. As the residue field has a characteristic not equal to $2$, therefore $2\in \R^{\times}$. Let us consider  $Z_i=\left(\begin{array}{cc}  a_i  & 0 \\ 1  & a_i^{-1} \end{array}\right)$ in $\SL_2(\R)$ for $i=1, 2$. Then $e_1^{\R}(Z_1, Z_2) = \left(\begin{array}{cc} 1  & 0 \\ -2\pi_2  & 1 \end{array} \right) \sim_{\R}\left(\begin{array}{cc} 1  & 0 \\ \pi_2  & 1 \end{array} \right)$. The only remaining element as a lift of $I$ is $B = \left(\begin{array}{cc} 1+\pi_2 a &  \\  & 1-\pi_2 a  \end{array}\right)$ where $a \in \R^{\times}$. Consider $D = \left(\begin{array}{cc} 1-\frac{\pi_2 a}{2} &  \\  & 1+\frac{\pi_2 a}{2} \end{array} \right)$. Then $BD = \left(\begin{array}{cc} 1+\frac{\pi_2 a}{2} &  \\
& 1-\frac{\pi_2 a}{2} \end{array}\right)$. Note that $BD$ and $D$ have the same eigenvalues, but neither is regular semisimple. However, in this case, a similar result to Lemma~\ref{sim1} can be proved. There exists $T_1=\left(\begin{array}{cc} 0 & 1 \\ 1 & 0 \end{array}\right)\in \GL_2(\R)$ such that $T_1 B DT_1^{-1} = D$. Consider,  $s_0 = det(T_1^{-1})$ and $S_0 = diag(s_0, 1)$. Then $S_0 T_1 B D T_1^{-1} S_0^{-1} = D$. Now take $S = S_0 T_1$, then $S\in \SL_2(\R)$. Therefore $B=e_1^{\R}(S^{-1},D)$. 
\end{proof}

\section{Magnus Embedding and the triangular Unipotent elements as image of Engel Words}{\label{Mag-Uni}}

In this section, we use a generalization of the ``Magnus embedding"(see~\cite{BaZa2016}, \cite{BNT2024}) to investigate if an unipotent element of $\SL_2(\R)$ is in $\im(e_r^{\R})$.  An element $A\in \SL_2(\R)$ is said to be \emph{unipotent} if $A-I$ is nilpotent in $\M_2(\R)$ (see page 22, \cite{HahnOMeara}).  Up to conjugacy, non-trivial unipotent elements can be read from the family $(3)$ listed in Table~\ref{tab:GL-conjugacy-SL}. These are either the lower triangular unipotent elements of the form   $\left(\begin{array}{cc} 1  & 0 \\ 1  & 1 \end{array}\right )$ and $\left(\begin{array}{cc} 1  & 0 \\ \pi_2 & 1 \end{array}\right )$ listed in Table~\ref{tab:simplified}, or of the form $\begin{pmatrix} 1 + 2^{-1}\pi v &\pi_2 v \\ 1 & 1+2^{-1}\pi v \end{pmatrix}$ up to $\GL$-conjugates; where $v$ is an unit.

Consider $\Omega_2 = \R[t_1, t_1^{-1}, t_2, t_2^{-1}]$, the commutative ring of Laurent polynomials over $\R$. Let us consider the subgroups 
$$ ST_2(\Omega_2) = \left\{\left(\begin{array}{cc} a  & 0 \\  b & a^{-1} \end{array} \right)\mid  a, b\in \Omega_2, a\in\Omega_2^{\times} \right\}$$ 
of $\SL_2(\R)$ and the subgroup 
$$T_2(\Omega_2) = \left\{\left(\begin{array}{cc}
   a  & 0 \\   b & 1 \end{array}\right)\mid a,b \in \Omega_2, a \in \Omega_2^{\times} \right\}$$
of $\GL_2(\R)$. 
We have the group homomorphism $\mu_1\colon \mathcal F_2 \rightarrow ST_2(\Omega_2)$ defined by $x_i \mapsto \left(\begin{array}{cc} t_i  & 0 \\ 1 & t_i^{-1} \end{array}\right) $ (see page 383, \cite{BNT2024}). Then by Proposition 2.4~\cite{BNT2024};  $ker(\mu_1) = \mathcal F_2^{(2)}(\mathcal F_2^{(1)})^p$ where $p$ is the characteristic of $\R$.  

\begin{lemma}{\label{kernel}}
Let $\R$ be a local principal ideal ring of length two with residue field $k$ of characteristic $\neq 2$ and $|k|\geq 5$. Let $\mathcal F_2$ be the free group generated by $x_1$ and $x_2$. Consider the group homomorphism $\mu_1 \colon \mathcal F_2 \rightarrow ST_2(\Omega_2)$ defined by $x_i \mapsto \left( \begin{array}{cc}
 t_i  & 0 \\ 1 & t_i^{-1} \end{array}\right) $ where $\Omega_2 = \R[t_1, t_1^{-1}, t_2, t_2^{-1}]$. Then, $e_r(x_1, x_2) \notin ker(\mu_1)$ for any $r\geq 1$.
\end{lemma}
\begin{proof}
Since $\mu_1$ is a group homomorphism, we get $$\mu_1([x_1, x_2]) = \begin{pmatrix} t_1  & 0 \\ 1 & t_1^{-1} \end{pmatrix} \begin{pmatrix}
 t_2  & 0 \\ 1 & t_2^{-1} \end{pmatrix} \begin{pmatrix} t_1^{-1}  & 0 \\ -1 & t_1 \end{pmatrix} \begin{pmatrix} t_2^{-1}  & 0 \\ -1 & t_2 \end{pmatrix} = \left(\begin{array}{cc} 1  & 0 \\ \mathcal{L}_{e_1}(t_1, t_2) & 1
\end{array} \right)$$ 
where $\mathcal{L}_{e_1}(t_1, t_2) = (t_2 - t_1)(1 + t_1t_2)(t_1^{-1}t_2^{-1})^2$ is a Laurent polynomial in $\Omega_2$. Now, a similar computation gives 
$$\mu_1(e_2(x_1, x_2)) = \begin{pmatrix} 1  & \\ \mathcal{L}_{e_1}(t_1, t_2) & 1 \end{pmatrix} \begin{pmatrix}
 t_2  & 0 \\ 1 & t_2^{-1} \end{pmatrix} \begin{pmatrix} 1  & 0 \\ -\mathcal{L}_{e_1}(t_1, t_2) & 1 \end{pmatrix} \begin{pmatrix} t_2^{-1}  & 0 \\ -1 & t_2 \end{pmatrix} = \left(\begin{array}{cc} 1  & 0 \\   \mathcal{L}_{e_2}(t_1, t_2) & 1 \end{array} \right)$$ 
 where $\mathcal{L}_{e_2}(t_1, t_2) = \mathcal{L}_{e_1}(t_1, t_2)(1-t_2^{-2}) \in \Omega_2.$
Inductively, we obtain $\mu_1(e_r(x_1, x_2))=\left(\begin{array}{cc}
   1  & 0 \\  \mathcal{L}_{e_r}(t_1, t_2) & 1
\end{array}\right)$, where $\mathcal{L}_{e_r}(t_1,t_2) = \mathcal{L}_{e_1}(t_1, t_2)(1-t_2^{-2})^{r-1}\in \Omega_2$, for any $r\geq 1$. 

Now, if possible let $e_r(x_1, x_2)\in ker(\mu_1)$ for $r\geq 1$. Consider the following diagram 
$$\begin{tikzcd}
	{\mathcal F_2} & {ST_2(\Omega_2)} \\
	& {ST_2(\bar{\Omega}_2)}
	\arrow["{\mu_1}", from=1-1, to=1-2]
	\arrow["{\theta\circ \mu_1}"', from=1-1, to=2-2]
	\arrow["\theta", from=1-2, to=2-2]
\end{tikzcd}$$ and we get $e_r(x_1, x_2)\in ker(\theta \circ \mu_1)$; where $\bar{\Omega}_2 = k[t_1, t_1^{-1}, t_2, t_2^{-1}]$. Hence, $\mu_1(e_r(x_1, x_2))=I$ gives $\mathcal{L}_{e_r}(t_1, t_2) = 0$, the zero polynomial over the residue field $k$, i.e. $\mathcal{L}_{e_r}(t_1, t_2) = 0$ for all $(t_1, t_2)\in k^2$. Which is a contradiction if we show that $\mathcal L_{e_r}(x_1, x_2)$ is a non-trivial polynomial. To show this, let us consider the following two cases.
\begin{enumerate}
\item \textbf{$k$ has characteristic $0$.} In this case, as the prime subfield of $k$ is isomorphic to $\mathbb{Q}$, we can check this for $t_1 = 1$, $t_2 = 2$. Note that in this case $\mathcal{L}_{e_r}(1, 2)\neq 0$.
\item \textbf{$k$ has characteristic $p$, an odd prime.} In this case the prime subfield of $k$ is isomorphic to $\mathbb{F}_p$. When $p\neq 3$, we can choose $t_1 = 1$ and $t_2 = 2$, so that $\mathcal{L}_{w_c}(1, 2)\neq 0$ and hence $\mathcal{L}_{e_r}(1,2)\neq 0$. 
    
When $p=3$, as we are considering $|k| \geq 5$ therefore $|k^{\times}| \geq 4$. Now we want $a, b$ in $k^{\times}$ such that $a\neq b$ and $ab\neq 2$. Consider $a=1$. Then $b$ must not take the values $ 1$ or $ 2$. Consider $\mathcal{S}=\{1, 2\}$. Then $|k^{\times}\backslash \mathcal{S}| \geq 2$. Pick $b$ from $k^{\times}\backslash \mathcal{S}$. This implies $\mathcal{L}_{e_r}(1, b)\neq 0$.
\end{enumerate}
Therefore, there exists $(t_1, t_2)$ in $k^{\times}\times k^{\times}$, so that $\mathcal{L}_{e_r}(t_1, t_2)$ is nonzero in $k$ for any $r\geq 1$. This proves the result.
\end{proof}

We note here that the assumption that $|k|$ is sufficiently large at the field level is needed. For $k=\mathbb{F}_3$, $\mathcal{L}_{e_r}(t_1, t_2) = 0$ always for any $(t_1, t_2)\in k^2$ such that $t_1, t_2$ are units and $r\geq 1$. Some issues regarding the case $k = \mathbb{F}_3$ have been discussed in Section~\ref{F_3 case}.

\begin{lemma}{\label{uni2}}
Suppose $|k|\geq 5$. There exists an element $d\in \R^{\times}$ such that the unipotent $\left(\begin{array}{cc} 1  & 0 \\ d & 1 \end{array}\right)\in \im(e_r^{\R})$ for all $r\geq 1$.
\end{lemma}
\begin{proof}
In the view of computation in the proof of Lemma~\ref{kernel}, we can get a lower triangular unipotent of the required kind provided $\mathcal{L}_{e_r}(t_1,t_2)$ takes a unit value in $\R^{\times}$. On contrary, suppose all the values of $\mathcal{L}_{e_r}(t_1, t_2)$ are nonunits when we evaluate them for $t_1, t_2 \in \R^{\times}$. This implies for every nonzero $t_1, t_2\in k$, the reduction of $\mathcal{L}_{e_r}(t_1, t_2)$, i.e, $\theta(\mathcal{L}_{e_r}(t_1, t_2))$ takes the value $0$. This contradicts the fact that $e_r(x_1, x_2) \notin ker(\theta \circ \mu_1)$ in Lemma~\ref{kernel}. Therefore there exists at least one $d\in \R^{\times}$ and a corresponding tuple $(a_1, a_2)\in \R^{\times} \times \R^{\times}$ such that $\mathcal{L}_{e_r}(a_1, a_2) = d$. Now take the matrices $Z_i=\left(\begin{array}{cc} a_i  & 0 \\
1  & a_i^{-1} \end{array}\right)$ in $\SL_2(\R)$ for $i=1, 2$. This immediately gives $e_r^{\R}(Z_1, Z_2) = \left(\begin{array}{cc} 1  & 0 \\ d & 1 \end{array}\right)$.
\end{proof}

\begin{lemma}{\label{triang unipotent}} The elements $\left(\begin{array}{cc} 1   & 0 \\ 1   & 1 \end{array}\right)$ for $|k|\geq 5$, and $\left(\begin{array}{cc} 1   & 0 \\ \pi_2   & 1 \end{array} \right)$ for $|k|\geq 7$ belong to $\im(e_r^{\R})$ in $\SL_2(\R)$ for all $r\geq 1$.
\end{lemma}
\begin{proof}
Since $\SL_2(\R)\trianglelefteq \GL_2(\R)$, we consider the group action via conjugation, viz.
$$\GL_2(\R)\times \SL_2(\R)\longrightarrow \SL_2(\R) $$ $$(\mathbf{g},X)\mapsto \mathbf{g}^{-1}X\mathbf{g}.$$
Note that the action is trivial over the scalar matrices. Now for any $r\geq 1$, $e_r^{\R}(Z_1, Z_2)$ obtained in Lemma~\ref{uni2} can be written as $I + \left(\begin{array}{cc}
  0  & 0 \\  d & 0 \end{array}\right)$. Take $\left( \begin{array}{cc} 1  & 0 \\  0 & d \end{array} \right) \in \GL_2(\R)$ and we get $\left(\begin{array}{cc} 1  & 0 \\
0 & d \end{array} \right)^{-1} \left(\begin{array}{cc} 0  & 0 \\
d & 0 \end{array}\right)\left(\begin{array}{cc} 1  & 0 \\ 0 & d
\end{array}\right) = \left(\begin{array}{cc} 0  & 0 \\
 1 & 0 \end{array}\right)$. Let us consider $\Tilde{Z}_i = \left(\begin{array}{cc} 1  & 0 \\ 0 & d \end{array} \right)^{-1} Z_i \left(\begin{array}{cc}  1  & 0 \\ 0 & d
\end{array} \right)$ in $\SL_2(\R)$ for $i=1, 2$. Then $e_r(\Tilde{Z}_1, \Tilde{Z}_2) = \left(\begin{array}{cc}
1  & 0 \\ 1 & 1 \end{array}\right)$. Therefore, $ \left( \begin{array}{cc} 1  & 0 \\ 1 & 1 \end{array} \right)$ belongs to $\im(e_r^{\R})$ for any $r\geq 1$, when $|k|\geq 5$.

Consider $\mathcal{S} = \{\pm 1\}$ if $-1$ is not a square in $k$ and $\mathcal{S} = \{\pm1, \gamma_1, \gamma_2\}$ if $-1$ is a square in $k$; where $\gamma_i$ are the roots of $x^2 = -1$ in $k^{\times}$ in the second case. Since $|k|\geq 7$, the set $|k^{\times}\backslash\mathcal{S}| \geq 2$. We choose a $\delta \in k^{\times} \backslash \mathcal{S}$. Take $a_1\in \R^{\times}, a_2 = a_1 + \pi\in \R^{\times}$ such that $\theta(a_1) = \delta$. Therefore, $\mathcal{L}_{e_r}(a_1, a_2) = \mathcal{L}_{e_1}(a_1, a_2)(1 - a_2^{-2})^{r-1} = \pi(1 + a_1^2 + \pi a_1)(a_1a_2)^{-2} (1 - a_2^{-2})^{r-1}$. As the residue field has characteristic not equal to $2$ and $\delta \notin \mathcal{S}$, therefore $(1 + a_1^2+\pi a_1)(a_1a_2)^{-2}(1 - a_2^{-2})^{r-1}\in \R^{\times}$ for every $r\geq 1$. Let us denote $u_r = (1+a_1^2+\pi a_1)(a_1a_2)^{-2}(1-a_2^{-2})^{r-1}$ and $Z_i=\left(\begin{array}{cc} a_i  & 0 \\ 1  & a_i^{-1}
\end{array}\right)$ in $\SL_2(\R)$ for $i=1, 2$. Therefore, we have $e_r^{\R}(Z_1, Z_2) = \left(\begin{array}{cc} 1  & 0 \\  \pi_2 u_r  & 1 \end{array}\right)$. Moreover, $\left( \begin{array}{cc} 1  & 0 \\ 0 & u_r \end{array} \right)^{-1} \left( \begin{array}{cc} 0  & 0 \\ \pi_2 u_r & 0 \end{array} \right) \left( \begin{array}{cc} 1  & 0 \\ 0 & u_r \end{array} \right) = \left(\begin{array}{cc} 0  & 0 \\ \pi_2 & 0
\end{array} \right)$. Again, let us consider $\Tilde{Z}_i = \left(\begin{array}{cc} 1  & 0 \\ 0 & u_r \end{array} \right)^{-1} Z_i \left(\begin{array}{cc} 1  & 0 \\ 0 & u_r
\end{array} \right)$ in $\SL_2(\R)$ for $i=1, 2$. Then $e_r^{\R}(\Tilde{Z}_1, \Tilde{Z}_2) = \left(\begin{array}{cc} 1  & 0 \\
   \pi_2 & 1 \end{array} \right)$ and hence, $\left(\begin{array}{cc} 1  & 0 \\ \pi_2 & 1
\end{array}\right)$ belongs to $\im(e_r^{\R})$ in $\SL_2(\R)$ for any $r\geq 1$.
\end{proof} 
For the third kind of unipotent element that we have described at the beginning of this section, this technique of Magnus embedding doesn't work well. We handle this differently in the following sections.

\section{Commutators in $\SL_n(\mathcal O)$}{\label{towards Shalev's conj}}

Aner Shalev in~\cite{Shalevzeta}, conjectured that every element of $\mathrm{SL}_n(\mathbb{Z}_p)$ for $p$ a prime and $n\geq 2$ is a commutator (for $n=2$ one needs to assume $p>3$). Following that, Avni, Gelander, Kassabov, and Shalev (see \cite{AvniGelanderKassabovShalev}) proved that every lift of non-scalars is a commutator in $\mathrm{SL}_n(\mathcal{O})$, for any local ring $\mathcal{O}$ with residue field $k$ that contains more than $n+1$ elements. That leaves the difficult task of dealing with the lift of scalars. Indeed, part of this was done when the residue field contains an $n$-th root of unity. This section recalls these results for commutators over $\SL_n(\mathbb{Z}_p)$, fixes some gaps in the proof given earlier, and looks at a few more fresh cases of lifting scalars, enough for our purpose of the case $n=2$ at hand. For this, we recall and refine some of the techniques they developed.  

The Lie algebra of $\SL_n(k)$ is denoted by $\mathfrak{sl}_n(k) = \{P \in \M_n(k) \mid \tr(P) = 0\}$. We have a nondegenerate symmetric bilinear form given by $\langle P, Q \rangle = \tr(PQ)$ over the $k$ vector space $\mathfrak{sl}_n(k)$, when $char(k)$ is odd and $char(k) \nmid n$. Note that for $X, Y \in \mathfrak{sl}_n(k)$, the non-degeneracy of the Killing form $\langle X, Y\rangle_{\text{Killing}} = \tr(ad_X  ad_Y)$ does not guarantee the non-degeneracy of this bilinear form when $char(k)\mid n$.
\noindent With respect to this bilinear form, the orthogonal complement of a subspace $\mathfrak{U} \subset \mathfrak{sl}_n(k)$ will be denoted by $\mathfrak{U}^{\perp}$. Moreover, for an element $D\in \SL_n(k)$, we use the notation $Z(D)$ to denote the centralizer of $D$ in $\mathfrak{sl}_n(k)$, i.e., $Z(D) = Z_{\M_n(k)}(D) \cap \mathfrak{sl}_n(k)$. For any finite dimensional vector space $\mathcal{V}$ if there is a non degenerate symmetric bilinear form $\mathfrak{B}$ associated with it, then for any subspace $\mathcal{W}$ of $\mathcal{V}$ we have $dim(\mathcal{V}) = dim(\mathcal{W}) + dim(\mathcal{W}^{\perp})$, where $\mathcal{W}^{\perp}$ is the orthogonal complement of $\mathcal{W}$ with respect to the bilinear form $\mathfrak{B}$.
Moreover, this dimension formula and $\mathcal{W} \subset (\mathcal{W}^{\perp})^{\perp}$ together implies $(\mathcal{W}^{\perp})^{\perp} = \mathcal{W}$, in this case.
For any $D, g\in \SL_2(k)$, the notation $Z(D)^g = Z(g^{-1}Dg) = g^{-1}Z(D)g$.

The co-adjoint action is given as follows: 
\begin{eqnarray*}
\mathrm{Ad}^* \colon & \rm{SL}_n(k) \times \mathfrak{sl}_n^*(k) \rightarrow \mathfrak{sl}_n^*(k) \\ & (h, \phi) \mapsto \mathrm{Ad}^*_h\phi
\end{eqnarray*}
where $\mathrm{Ad}^*_{h}\phi(X) = \phi(\mathrm{Ad}_{h^{-1}}(X))$, for all $X\in \mathfrak{sl}_n(k)$. 
Here $\mathfrak{sl}_n^*(k)$  is the dual of $\mathfrak{sl}_n(k)$ as a $k$-vector space. A subgroup $H$ of $\SL_n(k)$ has a non-trivial fixed vector in $\mathfrak{sl}_n^*(k)$ if there exists a non-zero element in $\mathfrak{sl}_n^*(k)$ which is fixed by every element of $H$ via co-adjoint action. That is, $H$ has a fixed vector $\phi$ in $\mathfrak{sl}_n^*(k)$ means $\mathrm{Ad}_h^*\phi = \phi\ \forall h\in H$. We begin with recalling the following from \cite{AvniGelanderKassabovShalev} regarding lifting of commutators for special linear groups over local rings.

\begin{proposition}[Avni, Gelander, Kassabov, Shalev]{\label{adelic 1}}
Let $\mathcal{O}$ be a local ring with residue field $k$ of any characteristic. Let $\mathbf{\bar g}_1, \mathbf{\bar g}_2 \in \SL_n(K)$ be elements such that the group $H = \langle \mathbf{\bar g}_1, \mathbf{\bar g}_2 \rangle \subset \SL_n(k)$ does not have any fixed vectors in $\mathfrak {sl}_n^*(k)$. Then, for any lift $\mathbf{g}$ of $[\mathbf{\bar g}_1, \mathbf{\bar g}_2]$ to $\SL_n(\mathcal O)$ there exist $\mathbf{g}_1$ and $\mathbf{g}_2$ lifts of $\mathbf{\bar g}_1$ and $\mathbf{\bar g}_2$ respectively, such that $\mathbf{g} = [\mathbf{g}_1, \mathbf{g}_2]$.
\end{proposition}

Now we review the work on lift of scalars for $\SL_n(k)$. We remark that, in \cite[Corollary 3.4]{AvniGelanderKassabovShalev}, the element $\bar g_1$ is taken as $diag(1, \lambda, \lambda^2,\ldots,\lambda^{n-1})$ where $\lambda$ is a primitive $n$-th root of unity. Note that $det(\bar g_1) = \lambda^{n(n-1)/2}$ which is $1$ when $n$ is odd, thus works in that case. Whereas when $n$ is even $det(\bar g_1) = \lambda^{n(n-1)/2} = -1$ since $\lambda^{n/2} = -1$ as $\lambda$ is a primitive $n$-th root of unity. Hence $\bar g_1$ is not an element of $\SL_n(k)$ in that case. We carefully review this situation and provide an alternate proof of this case here. 

\begin{proposition}{\label{lifts of minus I}}
Let $\mathcal{O}$ be a local ring with residue field $k$ of odd characteristic $p$. Let $A = b I\in \SL_n(k)$ where $n$ is even and $b$ is a primitive $n$-th root of unity. Then, any lift of $A$ in $\SL_n(\mathcal{O})$ is a commutator in $\SL_n(\mathcal{O})$. 
\end{proposition}
\begin{proof}
Existence of the primitive $n$-th root in the field ensures that $p\nmid n$. By~\cite[Theorem 1]{Thompson1961} we know that $A$ is a commutator in $\SL_n(k)$. Thus, there exists $y_1, y_2,\ldots, y_n \in k$ such that $\det(Y_1)=-1$ where \[
Y_1 = \begin{pmatrix}
y_1 & -y_n & -y_{n-1} & \cdots & -y_2 \\
y_2 & y_1 & -y_n & \cdots & -y_3 \\
y_3 & y_2 & y_1 & \cdots & -y_4 \\
\vdots & \vdots & \vdots & \ddots & \vdots \\
y_n & y_{n-1} & y_{n-2} & \cdots & y_1
\end{pmatrix}
\]
and  $bI_n = PYP^{-1}Y^{-1}$ where $P$ is the companion matrix of $f(x) = x^n + 1$ and 
\[ Y = \begin{pmatrix}
y_1 & -b y_n & -b^2 y_{n-1} & \cdots & -b^{n-1} y_2 \\
y_2 & b y_1 & -b^2 y_n & \cdots & -b^{n-1} y_3 \\
\vdots & \vdots & \vdots & \ddots & \vdots \\
y_s & b y_2 & b^2 y_1 & \cdots & -b^{n-1} y_4 \\
\vdots & \vdots & \vdots & \ddots & \vdots \\
y_n & b y_{n-1} & b^2 y_{n-2} & \cdots & b^{n-1} y_1
\end{pmatrix}
\] 
where $P, Y\in \SL_n(k)$; see \cite[Proof of Theorem 1]{Thompson1961}. Let $\mathbf{\bar g}_1 = P$ and $\mathbf{\bar g}_2 = Y$ and consider the subgroup $H = \langle\mathbf{\bar g}_1, \mathbf{\bar g}_2\rangle$ of $\SL_n(k)$. To complete the proof, we use the Proposition~\ref{adelic 1}, i.e., we show that the co-adjoint action of $H$ has no fixed point. 

If possible let there exists $\phi \in \mathfrak {sl}^*_n(k)$ such that $\phi$ is a fixed vector of $\mathfrak{sl}^*_n(k)$ with respect to the group $H=\langle \mathbf{\bar g}_1, \mathbf{\bar g}_2\rangle\subset \SL_n(k)$. That is, for $X\in \mathfrak{sl}_n(k)$, we have $\phi(X) = \mathbf{\bar g}_i.\phi(X) = \mathrm{Ad}^*_{\mathbf{\bar g}_i}\phi(X) = \phi(\mathrm{Ad}_{\mathbf{\bar g}_i^{-1}}(X)) = \phi(\mathbf{\bar g}_i X \mathbf{\bar g}_i^{-1})$. Now, using the non-degeneracy of the bilinear form, defined as $\langle P, Q\rangle=\tr(PQ)$, there exists a unique $E_{\phi}\in \mathfrak{sl}_n(k)$ (for the given $\phi$) such that $\phi(M)=\langle E_{\phi}, M\rangle$. Thus, $\langle E_{\phi}, X\rangle = \phi(X) = \phi(\mathbf{\bar g}_i X \mathbf{\bar g}_i^{-1}) = \langle E_{\phi}, \mathbf{\bar g}_i X \mathbf{\bar g}_i^{-1}\rangle = \langle \mathbf{\bar g}_i^{-1} E_{\phi} \mathbf{\bar g}_i, X\rangle $. Since this happens for all $X$, we get $\mathbf{\bar g}_i E_{\phi} \mathbf{\bar g}_i^{-1} = E_{\phi}$ for $i=1, 2$. Therefore there exists $E_{\phi} \in \mathfrak{sl}_n(k)$ such that $E_{\phi} \in Z_{\M_n(k)}(P)\cap Z_{\M_n(k)}(Y)$ (noting $\mathbf{\bar g}_1 = P$ and $\mathbf{\bar g}_2 = Y$). Now $P$ is cyclic as its minimal and characteristic polynomial both are $f(x)$, hence $Z_{\M_n(k)}(P)\cong k[P]$. Thus, we can take $E_{\phi} = \psi(P)$ for some polynomial $\psi(t)\in k[t]$. 
Now $YE_{\phi}Y^{-1} = E_{\phi}$ implies $E_{\phi} = YE_{\phi} Y^{-1} = \psi (YPY^{-1}) = \psi(b^{-1}P)$. Therefore, we obtain $\psi(P) = \psi(b^{-1}P)$. Now, going over an algebraically closed field $\bar k$, there exists $Q\in \GL_n(\bar k)$ such that $QPQ^{-1} = diag(\beta, b\beta, \ldots, b^{n-1}\beta)$. Conjugating both side of the equation $\psi(P)=\psi(b^{-1}P)$ by $Q$, we obtain $\psi(\beta) = \psi(b\beta) = \ldots =\psi(b^{n-1}\beta)$. This implies $E_{\phi} =  \psi(P) = cI$ for some scalar $c$. Trace zero condition immediately implies $E_{\phi} = \mathbf{0}$. Thus $\phi=0$, and hence $H$ doesn't have any fixed vectors. Thus, invoking Proposition~\ref{adelic 1}, the result follows.
\end{proof}
\begin{remark}
When $H=\langle h_1,h_2\rangle$ the subgroup generated by $h_1, h_2\in \SL_2(k)$; it turns out that $H$ has no nontrivial fixed vector in $\mathfrak{sl}^*_2(k)$ is equivalent to say $Z(h_1) \cap Z(h_2) = \{\mathbf{0}\}$ in $\mathfrak{sl}_2(k)$ (see the proof of Proposition~\ref{lifts of minus I}).
\end{remark}

This finally ensures the surjectivity of the commutator map over $\SL_2(\mathcal{O}_2)$ as follows.
\begin{proposition}\label{prop:surjectivity-comm}
Let $\R$ be a local principal ideal ring of length two with perfect residue field $k$  of odd characteristic and $|k|\geq5$. Then the commutator word map $e_1 \colon \SL_2(\R)\times \SL_2(\R) \rightarrow \SL_2(\R)$ defined by $(g_1, g_2) \mapsto g_1g_2g_1^{-1}g_2^{-1}$ is surjective.
\end{proposition}
\begin{proof}
All the lifts of non-scalar matrices are in the image of $e_1^{\R}$ in $\SL_2(\R)$ by \cite[Theorem 3.5]{AvniGelanderKassabovShalev}. All the lifts of $I$ are also in $\im(e_1^{\R})$ in $\SL_2(\mathcal{O}_2)$ by Proposition~\ref{prop:lift of I}. The remaining elements are the lifts of $-I$. As $k$ contains a prime subfield $\mathbb{F}_p$, which is a finite field, $x^2 + y^2 = -1$ has a solution in $k$. Invoking Proposition~\ref{lifts of minus I} we obtain all the lifts of $-I$ are in $\im(e_1^{\R})$ in $\SL_2(\R)$.  
\end{proof}

\section{Co-adjoint action of $\rm{SL}_2(k)$ on $\mathfrak{sl}_2^*(k)$ and lifting solution}{\label{coadjoint lie th}}

We recall the notation set in Section~\ref{not-conv}. For $\ell\geq 1$, we denote the ring $\mathcal{O}_{\ell} = \mathcal{O} / \pi^{\ell}\mathcal{O}$ and the canonical map is $\theta_{\ell} \colon \mathcal{O}_{\ell+1} \rightarrow \mathcal{O}_{\ell}$. Let us recall the Lie algebra $\mathfrak{sl}_2(\mathcal{O}_{j+1}) = \{P\in \M_2(\mathcal{O}_{j+1}) \mid \tr(P)\in \mathfrak{m}_{j+1} \}$ for $j\geq 1$. In this section, we explore the ideas in~\cite{AvniGelanderKassabovShalev} that we used in the previous section for the commutator map via its action on its Lie algebra, and extend them to higher Engel words.  

Note that $\mathfrak{sl}_2(\mathcal{O}_{j+1})$ is generated by $\left\{ \left(\begin{array}{cc} 1 & 0 \\  0 & -1 \end{array}\right), \left(\begin{array}{cc}  0 & 1 \\ 0 & 0
\end{array}\right), \left(\begin{array}{cc} 0 & 0 \\ 1 & 0 \end{array}\right), \left(\begin{array}{cc} 0 & 0 \\ 0 & \pi_{j+1}
\end{array}\right)\right\}$ as a $\mathcal{O}_{j+1}$ module and hence it is a finitely generated $\mathcal{O}_{j+1}$ module for each $j \geq 1$. As each $\mathcal{O}_{j+1}$ is a principal ideal ring, it is Noetherian. Hence any $\mathcal{O}_{j+1}$-submodule of the finitely generated $\mathcal{O}_{j+1}$-module $\mathfrak{sl}_2(\mathcal{O}_{j+1})$ is again finitely generated by \cite[Proposition 6.2 and 6.5]{Atiyahmacdonald}.

Let us recall the notion of cyclic elements.
\begin{definition}
Let $\mathcal{A}$ be a local ring with residue field $k$. An element $g\in \GL_n(\mathcal{A})$ is said to be cyclic if $\bar g$ is a cyclic matrix in $\GL_n(k)$. In other words, $k^n$ is a cyclic $k[t]$ module through the action of $t$ via $\bar g$. 
\end{definition}
Consider the canonical isomorphism $\mathcal{T^*} \colon  \M_n(\mathcal{O}) \rightarrow \varprojlim\limits _{j\geq 1} \M_n(\mathcal{O}_j)$ defined by $A\mapsto (A_1,A_2,\ldots )$ where $A_j=[a_{u,v}  + \mathfrak{m}^j]$ for $j\geq 1$; (here $a_{u,v}$ is the $(u,v)$-th entry of the matrix $A=[a_{u,v}]$). Thus, at times, we think of an element $A\in \SL_n(\mathcal{O})$, as an element of $\M_n(\mathcal{O})$ and visualize it as $(A_1,A_2,\ldots)$ under the above isomorphism $\mathcal{T}^*$; see \cite[section 2]{PRS2026}. Note that $A_{j+1}$ is a lift of $A_j$ for each $j$, thus the determinant of each $A_j$ is $1$ and $A_j\in \SL_n(\mathcal{O}_j)$ for all $j\geq 1$.
The following Lemma relates the lifting of the solution to each $j$-th level ring $\mathcal O_{j}$ to that over $\mathcal O$. 
\begin{lemma}{\label{lifting complete level}}
Let $\mathcal{O}$ be a local principal ideal ring, complete with respect to its maximal ideal $\mathfrak{m}=(\pi)$, and it has a residue field $k$ of characteristic $\neq 2$. Let $e_{m+1}$ be the $(m+1)$-th Engel word in $\mathcal{F}_2$ for $m\geq 1$ and $A=(A_1, A_2, \ldots) \in \SL_n(\mathcal{O})$. Suppose $A_1 = \bar e_{m+1}(X_1, Y_1)$ for some $(X_1, Y_1)\in \SL_n(k)^2$. If for each $A_{j+1}\in \SL_n(\mathcal{O}_{j+1})$ there exists lift  $(X_{j+1}, Y_{j+1})\in \SL_n(\mathcal{O}_{j+1})^2$ of $(X_j, Y_j)$ with $A_{j+1} = e_{m+1}^{\mathcal{O}_{j+1}}(X_{j+1},Y_{j+1})$;  where $A_j=e_{m+1}^{\mathcal{O}_j}(X_j,Y_j)$ then, $A\in  \im(e_{m+1}^{\mathcal{O}})$.  
\end{lemma}
\begin{proof}
We have, for each $A_{j+1}\in \SL_n(\mathcal{O}_{j+1})$ there exists lift  $(X_{j+1},Y_{j+1})\in \SL_n(\mathcal{O}_{j+1})^2$ of $(X_j,Y_j)$ such that $A_{j+1}=e_{m+1}^{\mathcal{O}_{j+1}}(X_{j+1}, Y_{j+1})$; where $A_j = e_{m+1}^{\mathcal{O}_j}(X_j, Y_j)$. Consider $\widehat{X}=(X_1, X_2,\ldots)$ and $\widehat{Y} = (Y_1, Y_2,\ldots)$ in $\varprojlim\limits _{j\geq 1} \M_n(\mathcal{O}_j)$. Let us denote $X=\mathcal{T}^{*-1}(\widehat{X})$ and $Y = \mathcal{T}^{*-1}(\widehat{Y})$. Then, from the definition of $\mathcal{T}^*$ it is easy to see that $\mathcal{T}^*(e_{m+1}^{\mathcal{O}}(X, Y))=(\bar e_{m+1}(X_1, Y_1), e_{m+1}^{\R}(X_2, Y_2),\ldots) = (A_1, A_2,\ldots)$. As $\mathcal{T}^*$ is an isomorphism therefore $A = e_{m+1}^{\mathcal{O}}(X, Y)$. From the canonical isomorphism $\mathcal{O} \rightarrow \varprojlim\limits_{j\geq 1}\mathcal{O}_j$, it is easy to observe that the image of $det(X)$ and $det(Y)$ is $(1,1,\ldots)$ which is also the image of $1$. Therefore $(X, Y)\in \SL_n(\mathcal{O})^2$.
\end{proof}
\noindent If $A\in\GL_n(\mathcal{O})$ is a cyclic matrix, then in the above notation $A_j\in \GL(\mathcal{O}_j)$ must be cyclic for all $j$, and they have the same minimal and characteristic polynomial for each $j\geq 1$; see \cite[corollary 2.3]{PRS2026}. We explain some notation before proceeding further. Consider the Engel map $\bar e_{m+1} \colon \GL_2(\mathcal{O}_2)^2 \rightarrow \GL_2(k)^2$ induced by the natural reduction of coefficients. For $(\mathbf{\bar g}_1, \mathbf{\bar g}_2) \in \GL_2(k)^2$ we have $\bar e_{m+1}(\mathbf{\bar g}_1,\mathbf{\bar g}_2)\in \GL_2(k)$. Take any lift $\mathbf{g}$ of $\bar e_{m+1}(\mathbf{\bar g}_1, \mathbf{\bar g}_2)$ in $\SL_2(\mathcal{O}_2)$. Then, there exists $\mathbf{g}_1, \mathbf{g}_2$ (lifts of $\mathbf{\bar g}_1, \mathbf{\bar g}_2$ respectively) in $\SL_2(\mathcal{O}_2)$ such that $\mathbf{g} = e_{m+1}^{\R}(\mathbf{g}_1, \mathbf{g}_2)(I+ \pi_2 \mathcal{D})$, for some $\mathcal D$ where $\tr(\mathcal{D})\in \mathfrak{m}_2$ and $\pi_2$ is the generator of $\mathfrak{m}_2$ with $\pi_2^2=0$. 
This is easy to see that there exists $h_1, h_2\in \GL_2(\mathcal{O}_2)$ which are the lifts of $\mathbf{\bar g}_1$ and $\mathbf{\bar g}_2$ respectively such that $e_{m+1}(h_1, h_2) + \pi_2 C = \mathbf{g}$ for some $C\in M_2(\mathcal{O}_2)$. Now, by Lemma~\ref{matsl} we can write $h_1 = \mathbf{g}_1 + \pi_2 C_1$ and $h_2 = \mathbf{g}_2 + \pi_2 C_2$, where $\mathbf{g}_1,\mathbf{g}_2\in \SL_2(\mathcal{O}_2)$ and $C_1, C_2 \in \M_2(\mathcal{O}_2)$. Now, $e_{m+1}(h_1, h_2) = e_{m+1}^{\R}(\mathbf{g}_1, \mathbf{g}_2) + \pi_2 D_1$ for some $D_1\in \M_2(\mathcal{O}_2)$. Hence, $\mathbf{g} = e_{m+1}^{\R}(\mathbf{g}_1,\mathbf{g}_2)\left(I+ \pi_2 e_{m+1}^{\R}(\mathbf{g}_1, \mathbf{g}_2)^{-1}(D_1 + C)\right)$. Now write $\mathcal{D}= e_{m+1}^{\R}(\mathbf{g}_1, \mathbf{g}_2)^{-1}(D_1 + C)$. Finally, the Lie algebra $\mathfrak{sl}_2(\mathcal{O}_2) = \{P \in \M_2(\mathcal{O}_2) \mid \tr(P) \in ker(\mathcal{O}_2 \rightarrow K) = \mathfrak{m}_2\}$.

\begin{lemma}{\label{lem-adelic 2}}
Let $\mathcal{O}_2$ be a local principal ideal ring of length $2$ with residue field $k$ of characteristic $\neq 2$. Let $\mathbf{\bar g}_1, \mathbf{\bar g}_2 \in \SL_2(k)$ be such that the groups $H_0 = \langle \mathbf{\bar g}_1,\mathbf{\bar g}_2\rangle$ and $ H_1=\langle \bar e_1(\mathbf{\bar g}_1,\mathbf{\bar g}_2),\mathbf{\bar g}_2\rangle$ do not have any non-trivial fixed vector under the co-adjoint action on the dual Lie algebra $\mathfrak{sl}_2^*(k)$. Further, assume that $Z\left(\mathbf{\bar g}_1\mathbf{\bar g}_2\mathbf{\bar g}_1^{-1}\right)\cap Z\left(\mathbf{\bar g}_2\right)^{\perp}$ is trivial in $\mathfrak{sl}_2(k)$. Then, for any lift $\mathbf{h}$ of $\bar e_{2}(\mathbf{\bar g}_1, \mathbf{\bar g}_2)$ in $\SL_2(\mathcal O_2)$ there exists $\mathbf{h}_1, \mathbf{ h}_2\in\SL_2(\mathcal{O}_2)$ such that $\mathbf{h} = e_{2}^{\mathcal{O}_2}(\mathbf{h}_1, \mathbf{h}_2)$ where $\mathbf{h}_1, \mathbf{ h}_2$ are lifts of $\mathbf{\bar g}_1$ and $\mathbf{\bar g}_2$ respectively.
\end{lemma}
\begin{proof}
It is enough to show surjectivity of the derivative map at $(\mathbf{\bar g}_1, \mathbf{\bar g}_2)$, corresponding to $\bar e_{2}$ at the Lie algebra level. Take any lift $\mathbf{g}$ of $\bar e_{2}(\mathbf{\bar g}_1, \mathbf{\bar g}_2)$ in $\SL_2(\mathcal{O}_2)$. As explained above, there exists $\mathbf{g}_1, \mathbf{g}_2$ (lifts of $\mathbf{\bar g}_1, \mathbf{\bar g}_2$ respectively) in $\SL_2(\mathcal{O}_2)$ such that $\mathbf{g} = e_{2}^{\R}(\mathbf{g}_1, \mathbf{g}_2)(I+ \pi_2 \mathcal{D})$, for some $\mathcal D$, where $\tr(\mathcal{D})\in \mathfrak{m}_2$ and $\pi_2$ is the generator of $\mathfrak{m}_2$ with $\pi_2^2=0$. Let $X, Y\in \mathfrak{sl}_2(\mathcal{O}_2)$. 
    
We have, the subgroups $H_0= \langle \mathbf{\bar g}_1, \mathbf{\bar g}_2\rangle$ and $H_1 = \langle [\mathbf{\bar g}_1, \mathbf{\bar g}_2], \mathbf{\bar g}_2\rangle\subset \SL_2(k)$ do not have any fixed vector in $\mathfrak {sl}_2^*(k)$. Then, $e_2^{\R}(\mathbf{g}_1(1+ \pi_2 X), \mathbf{g}_2(1+ \pi_2 Y)) = e_2^{\R}(\mathbf{g}_1, \mathbf{g}_2)\left(1 + \pi_2 (L^{\mathbf{g}_2 e_1^{-1} \mathbf{g}_2^{-1}} + Y^{e_1^{-1} \mathbf{g}_2^{-1}} -L^{e_1^{-1} \mathbf{g}_2^{-1}} - Y^{\mathbf{g}_2^{-1}})\right)$. Here $e_1: = e_1^{\R}(\mathbf{g}_1, \mathbf{g}_2) = [\mathbf{g}_1, \mathbf{g}_2]$ and $L: = L(X, Y) = X^{\mathbf{g}_2 \mathbf{g}_1^{-1} \mathbf{g}_2^{-1}} + Y^{\mathbf{g}_1^{-1} \mathbf{g}_2^{-1}} -X^{\mathbf{g}_1^{-1} \mathbf{g}_2^{-1}}-Y^{\mathbf{g}_2^{-1}} = \left((X^{\mathbf{g}_2}-X)^{\mathbf{g}_1^{-1}} + (Y^{\mathbf{g}_1^{-1}}-Y)\right)^{\mathbf{g}_2^{-1}}$ where the notation $A^g$ denotes $g^{-1}Ag$. Moreover, $e_1(\mathbf{g}_1(1 + \pi_2 X), \mathbf{g}_2(1 + \pi_2 Y)) = e_1(\mathbf{g}_1, \mathbf{g}_2)\left(1 + \pi_2 L(X,Y)\right)$. Therefore, the derivative of the map $e_2$, denoted as $De^{\R}_2$, at $(\mathbf{g}_1, \mathbf{g}_2)$ is obtained from the expression of $L(X, Y)$ above:
$$De^{\R}_2\colon (X, Y) \mapsto \left((X^{\mathbf{g}_2}-X)^{\mathbf{g}_1^{-1}} + (Y^{\mathbf{g}_1^{-1}}-Y)\right)^{e_1^{-1}\mathbf{g}_2^{-1}} - \left((X^{\mathbf{g}_2}-X)^{\mathbf{g}_1^{-1}}+(Y^{\mathbf{g}_1^{-1}} -Y)\right)^{\mathbf{g}_2^{-1}e_1^{-1}\mathbf{g}_2^{-1}}+(Y^{e_1^{-1}}-Y)^{\mathbf{g}_2^{-1}}.$$ 
The notation $De^{\R}_2$ indicates that we are working over $\mathcal{O}_2$. We need to show that this map is surjective. 

Recall Proposition 3.1~\cite{AvniGelanderKassabovShalev} that for any $\bar g\in \SL_2(k)$, the image of $Z\mapsto Z^{\bar g}-Z$ is the orthogonal complement of the centralizer of $\bar g$ in $\mathfrak{sl}_2(k)$, with respect to the non-degenerate symmetric bilinear form $\langle P, Q\rangle = \tr(PQ)$. Further, the $\mathrm{Ad}$-map leaves this form invariant. We use the notation $\overline{De_2}$ to denote the same derivative map after reduction mod $\mathfrak{m}_2$, i.e., over $\mathfrak{sl}_2(k)$. We show, under the given assumptions, that the orthogonal complement $\im(\overline{De}_2)^{\perp}$ of $\im(\overline{De}_2)$ is trivial. For this, let $T$ be any vector in $\im(\overline{De}_2)^{\perp}$. Therefore, $\langle T, \overline{De}_2(X, Y)\rangle = 0 \ \forall X, Y\in \mathfrak{sl}_2(k)$. Now,
\begin{eqnarray*}
\overline{D}e_2(X,Y)&=& \mathrm{Ad}_{\mathbf{\bar g}_2 \bar e_1}\left(\mathrm{Ad}_{\mathbf{\bar g}_1}(\mathrm{Ad}_{\mathbf{\bar g}_2^{-1}}(X)-X)+(\mathrm{Ad}_{\mathbf{\bar g}_1}(Y)-Y)\right)-\mathrm{Ad}_{\mathbf{\bar g}_2 \bar e_1\mathbf{\bar g}_2}\Big(\mathrm{Ad}_{\mathbf{\bar g}_1}(\mathrm{Ad}_{\mathbf{\bar g}_2^{-1}}(X)-X) \\ && +(\mathrm{Ad}_{\mathbf{\bar g}_1}(Y)-Y)\Big) + \mathrm{Ad}_{\mathbf{\bar g}_2}(\mathrm{Ad}_{ \bar e_1}(Y)-Y).
\end{eqnarray*}
\noindent Let $V := \mathrm{Ad}_{\mathbf{\bar g}_2\bar e_1}^{-1}(T) = \mathrm{Ad}_{(\mathbf{\bar g}_2\bar e_1)^{-1}}(T)$ and $S := \mathrm{Ad}_{\mathbf{\bar g}_1}(\mathrm{Ad}_{\mathbf{\bar g}_2^{-1}}(X) - X)+(\mathrm{Ad}_{\mathbf{\bar g}_1}(Y)-Y)$. 

We consider the following two cases.
\begin{enumerate}
\item \textbf{Case I. When $(Y=\mathbf{0})$:} In this case $\langle T, \overline{De}_2(X, 0) \rangle = 0$. That is, for all $X$, we have
\begin{eqnarray*}
0=\langle T, \overline{De}_2(X,\mathbf{0}) \rangle &=& \left\langle T,\mathrm{Ad}_{\mathbf{\bar g}_2 \bar e_1}\left(\mathrm{Ad}_{\mathbf{\bar g}_1}(\mathrm{Ad}_{\mathbf{\bar g}_2^{-1}}-I)\right)(X)\right\rangle - \left\langle T, \mathrm{Ad}_{\mathbf{\bar g}_2 \bar e_1}\mathrm{Ad}_{\mathbf{\bar g}_2}\Big(\mathrm{Ad}_{\mathbf{\bar g}_1}(\mathrm{Ad}_{\mathbf{\bar g}_2^{-1}}-I)\Big)(X)\right\rangle\\
&=& \left\langle \mathrm{Ad}_{\mathbf{\bar g}_2 \bar e_1}^{-1} (T),\left(\mathrm{Ad}_{\mathbf{\bar g}_1}(\mathrm{Ad}_{\mathbf{\bar g}_2^{-1}}-I)\right)(X)\right\rangle - \left\langle  \mathrm{Ad}_{\mathbf{\bar g}_2}^{-1}\mathrm{Ad}_{\mathbf{\bar g}_2 \bar e_1}^{-1} (T),\Big(\mathrm{Ad}_{\mathbf{\bar g}_1}(\mathrm{Ad}_{\mathbf{\bar g}_2^{-1}}-I)\Big)(X)\right\rangle\\
&=& \left\langle V -  \mathrm{Ad}_{\mathbf{\bar g}_2}^{-1}(V), \mathrm{Ad}_{\mathbf{\bar g}_1}(\mathrm{Ad}_{\mathbf{\bar g}_2^{-1}}-I)(X)\right\rangle =  \left\langle \mathrm{Ad}_{\mathbf{\bar g}_1^{-1}}(I- \mathrm{Ad}_{\mathbf{\bar g}_2^{-1}})(V), (\mathrm{Ad}_{\mathbf{\bar g}_2^{-1}}-I)(X)\right\rangle.
\end{eqnarray*} 
Note that $Im(\mathrm{Ad}_{\mathbf{\bar g}_2^{-1}}-I)\perp ker(\mathrm{Ad}_{\mathbf{\bar g}_2^{-1}}-I)$, through the non-degeneracy of the bilinear form we get $\mathrm{Ad}_{\mathbf{\bar g}_1^{-1}}(I- \mathrm{Ad}_{\mathbf{\bar g}_2^{-1}})(V)\in Ker (I - \mathrm{Ad}_{\mathbf{\bar g}_2^{-1}}) = Z(\mathbf{\bar g}_2^{-1}) = Z(\mathbf{\bar g}_2)= Ker (I - \mathrm{Ad}_{\mathbf{\bar g}_2})$. Hence  $(I - \mathrm{Ad}_{\mathbf{\bar g}_2})\mathrm{Ad}_{\mathbf{\bar g}_1^{-1}}(I- \mathrm{Ad}_{\mathbf{\bar g}_2^{-1}})(V) =0$ which we rewrite as:
\begin{equation}{\label{ortho1}}
(I-\mathrm{Ad}_{\mathbf{\bar g}_2})\mathrm{Ad}_{\mathbf{\bar g}_1^{-1}}(U)=\mathbf{0} \ \ \text{where}\  U:=(I-\mathrm{Ad}_{\mathbf{\bar g}_2^{-1}})(V).  
\end{equation}

\item \textbf{Case II. When $(X=\mathbf{0})$:} In  this case $\langle T, \overline{De}_2(X, Y)\rangle = 0$. That is for all $Y$, we have,
\begin{eqnarray*}
0&=& \langle T,\overline{De}_2(\mathbf{0},Y)\rangle = \langle T, \mathrm{Ad}_{\mathbf{\bar g}_2 \bar e_1}\left(\mathrm{Ad}_{\mathbf{\bar g}_1}-I\right)(Y)-\mathrm{Ad}_{\mathbf{\bar g}_2 \bar e_1}\mathrm{Ad}_{\mathbf{\bar g}_2}\Big(\mathrm{Ad}_{\mathbf{\bar g}_1}-I\Big)(Y)+ \mathrm{Ad}_{\mathbf{\bar g}_2}\Big(\mathrm{Ad}_{\mathbf{\bar e}_1}-I\Big)(Y) \rangle \\
&=& \langle T,\mathrm{Ad}_{\mathbf{\bar g}_2 \bar e_1}\left(\mathrm{Ad}_{\mathbf{\bar g}_1}-I\right)(Y)\rangle-\langle T,\mathrm{Ad}_{\mathbf{\bar g}_2 \bar e_1}\mathrm{Ad}_{\mathbf{\bar g}_2}\Big(\mathrm{Ad}_{\mathbf{\bar g}_1}-I\Big)(Y)\rangle +\langle T,\mathrm{Ad}_{\mathbf{\bar g}_2}\Big(\mathrm{Ad}_{\mathbf{\bar e}_1}-I\Big)(Y)\rangle \\
&=& \langle V,\left(\mathrm{Ad}_{\mathbf{\bar g}_1}-I\right)(Y)\rangle-\langle \mathrm{Ad}_{\mathbf{\bar g}_2^{-1}}V,\left(\mathrm{Ad}_{\mathbf{\bar g}_1}-I\right)(Y)\rangle+\langle \mathrm{Ad}_{\mathbf{\bar g}_2^{-1}}(T),\left(\mathrm{Ad}_{\bar e_1}-I\right)(Y)\rangle \\
&=& \langle (I - \mathrm{Ad}_{\mathbf{\bar g}_2^{-1}})(V), (\mathrm{Ad}_{\mathbf{\bar g}_1}-I)(Y)\rangle + \langle \mathrm{Ad}_{\mathbf{\bar g}_2^{-1}}(T),(\mathrm{Ad}_{\bar e_1}-I)Y \rangle 
= \langle U, (\mathrm{Ad}_{\mathbf{\bar g}_1}-I)(Y)\rangle  + \langle \mathrm{Ad}_{\mathbf{\bar e}_1}(V),(\mathrm{Ad}_{\bar e_1}-I) Y \rangle\\
&=& \langle U,\mathrm{Ad}_{\mathbf{\bar g}_1}(Y)\rangle-\langle U,Y\rangle + \langle  \mathrm{Ad}_{\mathbf{\bar e}_1}(V), \mathrm{Ad}_{\mathbf{\bar e}_1}(Y)\rangle-\langle  \mathrm{Ad}_{\mathbf{\bar e}_1}(V),Y\rangle\\
&=&\langle \mathrm{Ad}_{\mathbf{\bar g}_1^{-1}}(U),Y\rangle-\langle U,Y\rangle + \langle V,Y\rangle-\langle  \mathrm{Ad}_{\mathbf{\bar e}_1}(V),Y\rangle = \langle \mathrm{Ad}_{\mathbf{\bar g}_1^{-1}}(U)-U  + V- \mathrm{Ad}_{\mathbf{\bar e}_1}(V), Y\rangle
\end{eqnarray*}
where we have used $\mathrm{Ad}_{\mathbf{\bar g}_2^{-1}}(T) = \mathrm{Ad}_{\bar e_1}(V)$ and invariance of form under $Ad$-map. The non-degeneracy of the bilinear form immediately implies that
\begin{equation}{\label{ortho2}}
(I-\mathrm{Ad}_{\mathbf{\bar g}_1^{-1}})(U)=(I-\mathrm{Ad}_{\bar e_1})(V).
\end{equation}
\end{enumerate}

Now, we claim that $U= \mathbf{0}$. We have the following properties of $U$: (1) From Equation~(\ref{ortho1}) it follows that $U\in Z(\mathbf{\bar g}_2)^{\mathbf{\bar g}_1^{-1}}$ and (2) from the definition of $U$ it is clear that $U\in Z(\mathbf{\bar g}_2)^{\perp}$. These together imply $U\in Z(\mathbf{\bar g}_2)^{\mathbf{\bar g}_1^{-1}}\cap Z(\mathbf{\bar g}_2)^{\perp} $. By assumption $Z(\mathbf{\bar g}_1\mathbf{\bar g}_2\mathbf{\bar g}_1^{-1})\cap Z(\mathbf{\bar g}_2)^{\perp}=\{\mathbf{0}\}$ in $\mathfrak{sl}_2(k)$. Therefore, the claim is true.

Since $U=0 = (I-\mathrm{Ad}_{\mathbf{\bar g}_2^{-1}})(V)$, then $V\in Z(\mathbf{\bar g}_2)$ from the definition of $U$. From Equation~(\ref{ortho2}), we get $V\in Z( \bar e_1)$. Therefore $V\in Z(\mathbf{\bar g}_2)\cap Z( \bar e_1) $ in $\mathfrak{sl}_2(k)$. The assumption $H_1 = \langle [\mathbf{\bar g}_1, \mathbf{\bar g}_2], \mathbf{\bar g}_2\rangle\subset \SL_2(k)$ does not have any fixed vector in $\mathfrak {sl}_2^*(k)$, implies $V = \mathbf{0}$. This further implies $T = \mathbf{0}$. Therefore, if $U=\mathbf{0}$, then $\im(\overline{De}_2)^\perp \cap \mathfrak{sl}_2(k) = \{\mathbf{0}\}$. Now, the dimension formula $dim(\im(\overline{De}_2)) + dim(\im(\overline{De}_2)^\perp) = dim(\mathfrak{sl}_2(k))$ implies $coker(\overline{De}_2) = \mathfrak{sl}_2(k)/ \im(\overline{De}_2) = \{0\}$. 

Note that $\im(De^{\mathcal O_2}_2)$ is an $\mathcal{O}_2$ submodule of the $\mathcal{O}_2$ module $\mathfrak{sl}_2(\mathcal{O}_2)$. Now, consider the short exact sequence 
\[\begin{tikzcd}
\mathbf{0} & {\im(De^{\R}_2)} & {\mathfrak{sl}_2(\mathcal{O}_2)} & {\mathfrak{sl}_2(\mathcal{O}_2)/\im(De^{\R}_2)} & \mathbf{0}.
	\arrow[from=1-1, to=1-2]
	\arrow[from=1-2, to=1-3]
	\arrow[from=1-3, to=1-4]
	\arrow[from=1-4, to=1-5]
\end{tikzcd}\]
Tensoring with $\mathcal{O}_2/\mathfrak{m}_2$, and using the right exactness we obtain $\left(\mathfrak{sl}_2(\mathcal{O}_2)/\im(De^{\R}_2)\right) \bigotimes_{\mathcal{O}_2}\mathcal{O}_2/\mathfrak{m}_2\cong coker(\overline{De}_2) = \{\mathbf{0}\}$. Let us denote $\mathcal{M} = \mathfrak{sl}_2(\mathcal{O}_2)/ \im(De^{\R}_2)$; clearly it is a finitely generated $\mathcal{O}_2$ module. Therefore $\mathcal{M}/ \mathfrak{m}_2\mathcal{M} = \{\mathbf{0}\}$. Invoking Proposition 2.6 \cite{Atiyahmacdonald} we obtain $\im(De^{\R}_2) = \mathfrak{sl}_2(\mathcal{O}_2)$. This completes the proof.
\end{proof}

\begin{proposition}{\label{adelic 2}}
Let $\mathcal{O}$ be a local principal ideal ring, complete with respect to its maximal ideal $\mathfrak{m}=(\pi)$ with residue field $k$ of characteristic $\neq 2$. Let $\mathbf{\bar g}_1, \mathbf{\bar g}_2 \in \SL_2(k)$ be such that the groups $H_0=\langle \mathbf{\bar g}_1,\mathbf{\bar g}_2\rangle, H_1=\langle \bar e_1(\mathbf{\bar g}_1,\mathbf{\bar g}_2),\mathbf{\bar g}_2\rangle, \ldots, H_m = \langle \bar e_m(\mathbf{\bar g}_1,\mathbf{\bar g}_2), \mathbf{\bar g}_2\rangle$ do not have any non-trivial fixed vector under the co-adjoint action on the dual Lie algebra $\mathfrak{sl}_2^*(k)$. Further assume, $Z\left(\mathbf{\bar g}_1\mathbf{\bar g}_2\mathbf{\bar g}_1^{-1}\right)\cap Z\left(\mathbf{\bar g}_2\right)^{\perp}$,   $Z\left(\bar e_1(\mathbf{\bar g}_1, \mathbf{\bar g}_2) \mathbf{\bar g}_2\bar e_1(\mathbf{\bar g}_1, \mathbf{\bar g}_2)^{-1}\right) \cap Z\left(\mathbf{\bar g}_2\right)^{\perp}, \ldots,$ $Z\left(\bar e_{m}(\mathbf{\bar g}_1, \mathbf{\bar g}_2)\mathbf{\bar g}_2\bar e_{m}(\mathbf{\bar g}_1, \mathbf{\bar g}_2)^{-1}\right)\cap Z\left(\mathbf{\bar g}_2\right)^{\perp}$ are trivial in $\mathfrak{sl}_2(k)$. Then, for any lift $\mathbf{h}$ of $\bar e_{m+1}(\mathbf{\bar g}_1, \mathbf{\bar g}_2)$ in $\SL_2(\mathcal O)$ there exists $\mathbf{h}_1, \mathbf{ h}_2\in\SL_2(\mathcal{O})$ such that $\mathbf{h} = e_{m+1}^{\mathcal{O}}(\mathbf{h}_1, \mathbf{h}_2)$ where $\mathbf{h}_1, \mathbf{ h}_2$ are lifts of $\mathbf{\bar g}_1$ and $\mathbf{\bar g}_2$ respectively.
\end{proposition}
\begin{proof}
It is enough to show the surjectivity of the derivative map of $\bar e_{m+1}$ at $(\mathbf{\bar g}_1, \mathbf{\bar g}_2)$ at the Lie algebra level. At first, we work over $\mathcal O_2$. Take any lift $\mathbf{g}$ of $\bar e_{m+1}(\mathbf{\bar g}_1,\mathbf{\bar g}_2)$ in $\SL_2(\mathcal{O}_2)$. Then (as explained before Lemma~\ref{lem-adelic 2}) there exists $\mathbf{g}_1,\mathbf{g}_2$ (lifts of $\mathbf{\bar g}_1,\mathbf{\bar g}_2$ respectively) in $\SL_2(\mathcal{O}_2)$ such that $\mathbf{g}=e_{m+1}^{\R}(\mathbf{g}_1,\mathbf{g}_2)(I+\pi_2 \mathcal{D})$ for some $\mathcal D$ where $\tr(\mathcal{D})\in \mathfrak{m}_2$. Recall that $\pi_2$ is the generator of $\mathfrak{m}_2$ with $\pi_2^2=0$. Let $X, Y\in \mathfrak{sl}_2(\mathcal{O}_2)$. When $m=1$, we have proved this in Lemma~\ref{lem-adelic 2}.

{\bf Step I:}  Along the line of proof in the previous Lemma, we look at the equation $\langle T, \overline{De}_{m+1}(X, Y)\rangle = 0$ for $m+1$-th Engel word $ e_{m+1}^{\R}$ for a given $m\geq 2$, where $T\in \im(\overline{De}_{m+1})^{\perp}, V=\mathrm{Ad}_{\bar e_m^{-1}\mathbf{\bar g}_2^{-1}}(T)$; when we put $Y=\mathbf{0}$ and $X=\mathbf{0}$ separately. Let us list down the two equations for $De^{\R}_{m+1}\colon  \mathfrak{sl}_2(\mathcal{O}_2) \times\mathfrak{sl}_2(\mathcal{O}_2) \rightarrow \mathfrak{sl}_2(\mathcal{O}_2)$ defined by $$(X, Y)\mapsto [((De_m^{\R}(X,Y))^{\mathbf{g}_2} - De_m^{\R}(X,Y))^{e_m^{-1}}+(Y^{e_m^{-1}}-Y)]^{\mathbf{g}_2^{-1}}.$$  
as follows

\begin{enumerate}
\item \textbf{Case I. When $ (Y=\mathbf{0})$:} In this case $\langle T, \overline{De}_{m+1}(X,Y)\rangle=0$.
That is, for all $X$, we have
\begin{eqnarray}{\label{eqnortho}}\nonumber
0=\langle T, \overline{De}_{m+1}(X,\mathbf{0}) \rangle &=& \left\langle T,\mathrm{Ad}_{\mathbf{\bar g}_2 \bar e_m}\left(\mathrm{Ad}_{\mathbf{\bar g}_2^{-1}}\left(\overline{De}_m(X,\mathbf{0})\right)-\overline{De}_m(X,\mathbf{0})\right)\right\rangle \\\nonumber
&=& \left\langle \mathrm{Ad}_{\bar e_m^{-1}\mathbf{\bar g}_2^{-1}}(T), \mathrm{Ad}_{\mathbf{\bar g}_2^{-1}}\left(\overline{De}_m(X,\mathbf{0})\right)\right\rangle- \left\langle \mathrm{Ad}_{\bar e_m^{-1}\mathbf{\bar g}_2^{-1}}(T), \overline{De}_m(X,\mathbf{0})\right\rangle \\
&=& \left\langle V, \mathrm{Ad}_{\mathbf{\bar g}_2^{-1}}\left(\overline{De}_m(X,\mathbf{0})\right)\right\rangle-\left\langle V,\overline{De}_m(X,\mathbf{0})\right\rangle=\left\langle \mathrm{Ad}_{\mathbf{\bar g}_2}(V)-V,\overline{De}_m(X,\mathbf{0}) \right\rangle.
\end{eqnarray} 
Now, \begin{eqnarray*}
\overline{De}_m(X,\mathbf{0}) =((\overline{De}_{m-1}(X,\mathbf{0}))^{\mathbf{g}_2}-\overline{De}_{m-1}(X,\mathbf{0}))^{e_{m-1}^{-1}\mathbf{\bar g}_2^{-1}} 
= \mathrm{Ad}_{\mathbf{\bar g}_2}\left(\mathrm{Ad}_{\bar e_{m-1}}\left(\mathrm{Ad}_{\mathbf{\bar g}_2^{-1}}\left(\overline{De}_{m-1}(X,\mathbf{0})\right)-\overline{De}_{m-1}(X,\mathbf{0})\right)\right).
 \end{eqnarray*}
Therefore,
\begin{eqnarray*}
&& \left\langle \mathrm{Ad}_{\mathbf{\bar g}_2}(V)-V,\overline{De}_m(X,\mathbf{0}) \right\rangle = \left\langle \mathrm{Ad}_{\mathbf{\bar g}_2}(V)-V,\mathrm{Ad}_{\mathbf{\bar g}_2}\left(\mathrm{Ad}_{\bar e_{m-1}}\left(\mathrm{Ad}_{\mathbf{\bar g}_2^{-1}}\left(\overline{De}_{m-1}(X,\mathbf{0})\right)-\overline{De}_{m-1}(X,\mathbf{0})\right)\right)\right\rangle \\
&=& \left\langle U,\mathrm{Ad}_{\bar e_{m-1}}\left(\mathrm{Ad}_{\mathbf{\bar g}_2^{-1}}\left(\overline{De}_{m-1}(X,\mathbf{0})\right)-\overline{De}_{m-1}(X,\mathbf{0})\right)\right\rangle \ \ [\text{where}\ U=(I-\mathrm{Ad}_{\mathbf{\bar g}_2^{-1}})V] \\
&=& \left\langle \mathrm{Ad}_{\bar e_{m-1}^{-1}}(U),\mathrm{Ad}_{\mathbf{\bar g}_2^{-1}}\left(\overline{De}_{m-1}(X,\mathbf{0})\right)\right\rangle-\left\langle \mathrm{Ad}_{\bar e_{m-1}^{-1}}(U), \overline{De}_{m-1}(X,\mathbf{0})\right\rangle = \left\langle \left(\mathrm{Ad}_{\mathbf{\bar g}_2}-I\right)\left(\mathrm{Ad}_{\bar e_{m-1}^{-1}}(U)\right), \overline{De}_{m-1}(X,\mathbf{0})\right\rangle
 \end{eqnarray*}
Again, using the expression 
\begin{eqnarray*}
\overline{De}_{m-1}(X,\mathbf{0}) =((\overline{De}_{m-2}(X,\mathbf{0}))^{\mathbf{g}_2}-\overline{De}_{m-2}(X,\mathbf{0}))^{e_{m-2}^{-1}\mathbf{\bar g}_2^{-1}} = \mathrm{Ad}_{\mathbf{\bar g}_2}\left(\mathrm{Ad}_{\bar e_{m-2}}\left(\mathrm{Ad}_{\mathbf{\bar g}_2^{-1}}\left(\overline{De}_{m-2}(X,\mathbf{0})\right)-\overline{De}_{m-2}(X,\mathbf{0})\right)\right)
 \end{eqnarray*}
we obtain 
\begin{eqnarray*}
&& \left\langle \left(\mathrm{Ad}_{\mathbf{\bar g}_2}-I\right)\left(\mathrm{Ad}_{\bar e_{m-1}^{-1}}(U)\right), \overline{De}_{m-1}(X,\mathbf{0})\right\rangle\\
&=& \left\langle \left(\mathrm{Ad}_{\mathbf{\bar g}_2}-I\right)\left(\mathrm{Ad}_{\bar e_{m-1}^{-1}}(U)\right), \mathrm{Ad}_{\mathbf{\bar g}_2}\left(\mathrm{Ad}_{\bar e_{m-2}}\left(\mathrm{Ad}_{\mathbf{\bar g}_2^{-1}}\left(\overline{De}_{m-2}(X,\mathbf{0})\right)-\overline{De}_{m-2}(X,\mathbf{0})\right)\right) \right\rangle\\
&=& \left\langle \left(I-\mathrm{Ad}_{\mathbf{\bar g}_2^{-1}}\right)\left(\mathrm{Ad}_{\bar e_{m-1}^{-1}}(U)\right), \mathrm{Ad}_{\bar e_{m-2}}\left(\mathrm{Ad}_{\mathbf{\bar g}_2^{-1}}\left(\overline{De}_{m-2}(X,\mathbf{0})\right)-\overline{De}_{m-2}(X,\mathbf{0})\right)\right\rangle\\
&=& \left\langle U_{m-1}, \mathrm{Ad}_{\bar e_{m-2}}\left(\mathrm{Ad}_{\mathbf{\bar g}_2^{-1}}\left(\overline{De}_{m-2}(X,\mathbf{0})\right)-\overline{De}_{m-2}(X,\mathbf{0})\right)\right\rangle\\
&=& \left\langle\mathrm{Ad}_{\bar e_{m-2}^{-1}} (U_{m-1}), \mathrm{Ad}_{\mathbf{\bar g}_2^{-1}}\left(\overline{De}_{m-2}(X,\mathbf{0})\right)-\overline{De}_{m-2}(X,\mathbf{0})\right\rangle 
= \left\langle \left(\mathrm{Ad}_{\mathbf{\bar g}_2}-I\right)\left(\mathrm{Ad}_{\bar e_{m-2}^{-1}}(U_{m-1})\right), \overline{De}_{m-2}(X,\mathbf{0})\right\rangle\ \\ && \text{where}\ U_{m-1}= (I-\mathrm{Ad}_{\mathbf{\bar g}_2^{-1}})\mathrm{Ad}_{\bar e_{m-1}^{-1}}(U).
 \end{eqnarray*}

Continuing this process, inductively we obtain \begin{eqnarray*}
&& \left\langle \left(\mathrm{Ad}_{\mathbf{\bar g}_2}-I\right)\left(\mathrm{Ad}_{\bar e_{m-2}^{-1}}(U_{m-1})\right), \overline{De}_{m-2}(X,\mathbf{0})\right\rangle \\ &=& \left\langle (\mathrm{Ad}_{\mathbf{\bar g}_2}-I)\mathrm{Ad}_{\bar e_1^{-1}} (I-\mathrm{Ad}_{\mathbf{\bar g}_2^{-1}})\mathrm{Ad}_{\bar e_2^{-1}} (I-\mathrm{Ad}_{\mathbf{\bar g}_2^{-1}})\mathrm{Ad}_{\bar e_3^{-1}}(I-\mathrm{Ad}_{\mathbf{\bar g}_2^{-1}}) 
     \cdots 
      \mathrm{Ad}_{\bar e_{m-1}^{-1}}(I-\mathrm{Ad}_{\mathbf{\bar g}_2^{-1}})(V), \overline{De}_1(X,\mathbf{0})\right\rangle
 \end{eqnarray*}
Now $\overline{De}_1(X,\mathbf{0})=\bar L(X,\mathbf{0})=\mathrm{Ad}_{\mathbf{\bar g}_2}\left(\mathrm{Ad}_{\mathbf{\bar g}_1}\left(\mathrm{Ad}_{\mathbf{\bar g}_2^{-1}}(X)-X\right)\right)$; see Lemma~\ref{lem-adelic 2}. Therefore \begin{eqnarray*}
&& \left\langle \left(\mathrm{Ad}_{\mathbf{\bar g}_2}-I\right)\left(\mathrm{Ad}_{\bar e_{m-2}^{-1}}(U_{m-1})\right), \overline{De}_{m-2}(X,\mathbf{0})\right\rangle \\ &=& \left\langle (I-\mathrm{Ad}_{\mathbf{\bar g}_2^{-1}})\mathrm{Ad}_{\bar e_1^{-1}} (I-\mathrm{Ad}_{\mathbf{\bar g}_2^{-1}})\mathrm{Ad}_{\bar e_2^{-1}} (I-\mathrm{Ad}_{\mathbf{\bar g}_2^{-1}})\mathrm{Ad}_{\bar e_3^{-1}}(I-\mathrm{Ad}_{\mathbf{\bar g}_2^{-1}}) 
     \cdots 
      \mathrm{Ad}_{\bar e_{m-1}^{-1}}(I-\mathrm{Ad}_{\mathbf{\bar g}_2^{-1}})(V), \mathrm{Ad}_{\mathbf{\bar g}_1}\left(\mathrm{Ad}_{\mathbf{\bar g}_2^{-1}}(X)-X\right)\right\rangle.
 \end{eqnarray*}
\noindent Therefore the Equation~(\ref{eqnortho}) finally turns out to be
\begin{eqnarray*}
&\left\langle (I-\mathrm{Ad}_{\mathbf{\bar g}_2^{-1}})\mathrm{Ad}_{\bar e_1^{-1}}(I-\mathrm{Ad}_{\mathbf{\bar g}_2^{-1}})\mathrm{Ad}_{\bar e_2^{-1}}(I-\mathrm{Ad}_{\mathbf{\bar g}_2^{-1}})\mathrm{Ad}_{\bar e_3^{-1}}(I-\mathrm{Ad}_{\mathbf{\bar g}_2^{-1}})\cdots  \cdots \mathrm{Ad}_{\bar e_{m-1}^{-1}}(I-\mathrm{Ad}_{\mathbf{\bar g}_2^{-1}})(V),\mathrm{Ad}_{\mathbf{\bar g}_1}(\mathrm{Ad}_{\mathbf{\bar g}_2^{-1}}-I)(X)\right\rangle \\& =0 \ \forall X\in \mathfrak{sl}_2(k). 
\end{eqnarray*}
\noindent This further implies 
\begin{eqnarray*}
&\mathrm{Ad}_{\mathbf{\bar g}_1^{-1}}(I-\mathrm{Ad}_{\mathbf{\bar g}_2^{-1}})\mathrm{Ad}_{\bar e_1^{-1}}(I-\mathrm{Ad}_{\mathbf{\bar g}_2^{-1}})\mathrm{Ad}_{\bar e_2^{-1}}(I-\mathrm{Ad}_{\mathbf{\bar g}_2^{-1}})\mathrm{Ad}_{\bar e_3^{-1}}(I-\mathrm{Ad}_{\mathbf{\bar g}_2^{-1}})\cdots \mathrm{Ad}_{\bar e_{m-1}^{-1}}(I-\mathrm{Ad}_{\mathbf{\bar g}_2^{-1}})(V)\in Z(\mathbf{\bar g}_2^{-1})\\&=Z(\mathbf{\bar g}_2)=Ker (I-\mathrm{Ad}_{\mathbf{\bar g}_2}).
\end{eqnarray*}
\noindent Therefore, 
\begin{equation}{\label{eqmain1}}
(I-\mathrm{Ad}_{\mathbf{\bar g}_2})\mathrm{Ad}_{\mathbf{\bar g}_1^{-1}}(U_1)=\mathbf{0}\ ;
\end{equation}
where $U_1 = (I-\mathrm{Ad}_{\mathbf{\bar g}_2^{-1}})\mathrm{Ad}_{\bar e_1^{-1}}(I-\mathrm{Ad}_{\mathbf{\bar g}_2^{-1}})\mathrm{Ad}_{\bar e_2^{-1}}(I-\mathrm{Ad}_{\mathbf{\bar g}_2^{-1}})\mathrm{Ad}_{\bar e_3^{-1}}(I-\mathrm{Ad}_{\mathbf{\bar g}_2^{-1}})\cdots \mathrm{Ad}_{\bar e_{m-1}^{-1}}(I-\mathrm{Ad}_{\mathbf{\bar g}_2^{-1}})(V)$. 
    
\noindent In the further steps, we denote $U_i=(I-\mathrm{Ad}_{\mathbf{\bar g}_2^{-1}})\mathrm{Ad}_{\bar e_i^{-1}}(U_{i+1})$ for $i=1, 2, \ldots, m-2$, where $U_{m-1}=(I- \mathrm{Ad}_{\mathbf{\bar g}_2^{-1}}) \mathrm{Ad}_{\bar e_{m-1}^{-1}}(U)$ and $U=(I - \mathrm{Ad}_{\mathbf{\bar g}_2^{-1}})(V)$.

\item \textbf{Case II. When $ (X=\mathbf{0})$ :} 
In  this case $\langle T, \overline{De}_{m+1}(X,Y)\rangle = 0$. That is, for all $Y\in \mathfrak{sl}_2(k)$ we have 
\begin{eqnarray*}
0 &=& \left\langle T,\overline{De}_{m+1}(\mathbf{0}, Y) \right\rangle = \left\langle T, \mathrm{Ad}_{\mathbf{\bar g}_2 \bar e_m}\left(\mathrm{Ad}_{\mathbf{\bar g}_2^{-1}}\left(\overline{De}_m(\mathbf{0},Y)\right)-\overline{De}_m(\mathbf{0},Y)\right)+\mathrm{Ad}_{\mathbf{\bar g}_2}\left(\left(\mathrm{Ad}_{\bar e_m}-I\right)Y\right)\right\rangle \\
&=& \left\langle \mathrm{Ad}_{\bar e_m^{-1}\mathbf{\bar g}_2^{-1}}(T), \mathrm{Ad}_{\mathbf{\bar g}_2^{-1}}\left(\overline{De}_m(\mathbf{0},Y)\right)\right\rangle- \left\langle \mathrm{Ad}_{\bar e_m^{-1}\mathbf{\bar g}_2^{-1}}(T), \overline{De}_m(\mathbf{0},Y)\right\rangle+\left\langle T,\mathrm{Ad}_{\mathbf{\bar g}_2}\left(\left(\mathrm{Ad}_{\bar e_m}-I\right)Y\right)\right\rangle \\
&=& \left\langle V, \mathrm{Ad}_{\mathbf{\bar g}_2^{-1}}\left(\overline{De}_m(\mathbf{0},Y)\right)\right\rangle-\left\langle V,\overline{De}_m(\mathbf{0},Y)\right\rangle+\left\langle T,\mathrm{Ad}_{\mathbf{\bar g}_2}\left(\left(\mathrm{Ad}_{\bar e_m}-I\right)Y\right)\right\rangle \\ &=&\left\langle \mathrm{Ad}_{\mathbf{\bar g}_2}(V)-V,\overline{De}_m(\mathbf{0},Y) \right\rangle +\left\langle T,\mathrm{Ad}_{\mathbf{\bar g}_2}\left(\left(\mathrm{Ad}_{\bar e_m}-I\right)Y\right)\right\rangle 
= \left\langle \mathrm{Ad}_{\mathbf{\bar g}_2}(U),\overline{De}_m(\mathbf{0},Y) \right\rangle +\left\langle T,\mathrm{Ad}_{\mathbf{\bar g}_2}\left(\left(\mathrm{Ad}_{\bar e_m}-I\right)Y\right)\right\rangle.
\end{eqnarray*}
\noindent Therefore, we get the equation
\begin{equation}\label{eq56}
\left\langle \mathrm{Ad}_{\mathbf{\bar g}_2}(U),\overline{De}_m(\mathbf{0},Y) \right\rangle +\left\langle T,\mathrm{Ad}_{\mathbf{\bar g}_2}\left(\left(\mathrm{Ad}_{\bar e_m}-I\right)Y\right)\right\rangle=0\ \forall \ Y\in \mathfrak{sl}_2(k).
\end{equation}
        
\noindent Now, $\mathrm{Ad}_{\mathbf{\bar g}_2^{-1}}(T) = \mathrm{Ad}_{\bar e_m}(V)$ implies $\left\langle T, \mathrm{Ad}_{\mathbf{\bar g}_2} \left( \left(\mathrm{Ad}_{\bar e_m}-I\right)Y \right) \right\rangle = \left\langle \left(I-\mathrm{Ad}_{\bar e_m} \right)V, Y \right\rangle$. Using this value in the previous Equation~(\ref{eqmain1}), we obtain 
\begin{equation}
\left\langle \mathrm{Ad}_{\mathbf{\bar g}_2}(U),\overline{De}_m(\mathbf{0},Y) \right\rangle +\left\langle \left(I-\mathrm{Ad}_{\bar e_m}\right)V,Y\right\rangle=0\ \forall \ Y\in \mathfrak{sl}_2(k)
\end{equation}
 \end{enumerate}

{\bf Step II:} Now, we claim that $U_1 = \mathbf{0}$. We have the following properties of $U_1$: (1) From Equation~(\ref{eqmain1}) it follows that $U_1 \in Z(\mathbf{\bar g}_2)^{\mathbf{\bar g}_1^{-1}}$, and (2) from the definition of $U_1$ it is clear that $U_1\in Z(\mathbf{\bar g}_2)^{\perp}$. These together imply $U_1\in Z(\mathbf{\bar g}_2)^{\mathbf{\bar g}_1^{-1}}\cap Z(\mathbf{\bar g}_2)^{\perp} $. By assumption $Z(\mathbf{\bar g}_1\mathbf{\bar g}_2\mathbf{\bar g}_1^{-1})\cap Z(\mathbf{\bar g}_2)^{\perp} = \{\mathbf{0} \}$ in $\mathfrak{sl}_2(k)$. Therefore, the claim is true.

We have the following properties of $U_2$: (1) As $U_1=(I-\mathrm{Ad}_{\mathbf{\bar g}_2^{-1}})\mathrm{Ad}_{\bar e_1^{-1}}(U_{2})$, therefore $U_1=\mathbf{0} \implies (I- \mathrm{Ad}_{\mathbf{\bar g}_2^{-1}}) \mathrm{Ad}_{\bar e_1^{-1}}(U_{2}) = \mathbf{0}$ and (2) from the definition of $U_2$ (in case I) it is clear that $U_2\in Z(\mathbf{\bar g}_2)^{\perp}$. 
Now, the assumption $Z(\bar e_1(\mathbf{\bar g}_1, \mathbf{\bar g}_2)\mathbf{\bar g}_2\bar e_1(\mathbf{\bar g}_1,\mathbf{\bar g}_2)^{-1})\cap Z(\mathbf{\bar g}_2)^{\perp}=\{\mathbf{0}\}$ immediately implies $U_2=\mathbf{0}$.

Following a similar process, we can prove that the conditions $Z(\mathbf{\bar g}_1\mathbf{\bar g}_2\mathbf{\bar g}_1^{-1})\cap Z(\mathbf{\bar g}_2)^{\perp}$, $Z(\bar e_1(\mathbf{\bar g}_1, \mathbf{\bar g}_2)\mathbf{\bar g}_2\bar e_1(\mathbf{\bar g}_1,\mathbf{\bar g}_2)^{-1})\cap Z(\mathbf{\bar g}_2)^{\perp},\ldots$, $Z(\bar e_{m-2}(\mathbf{\bar g}_1,\mathbf{\bar g}_2)\mathbf{\bar g}_2\bar e_{m-2}(\mathbf{\bar g}_1,\mathbf{\bar g}_2)^{-1})\cap Z(\mathbf{\bar g}_2)^{\perp}$, $Z(\bar e_{m-1}(\mathbf{\bar g}_1,\mathbf{\bar g}_2)\mathbf{\bar g}_2\bar e_{m-1}(\mathbf{\bar g}_1,\mathbf{\bar g}_2)^{-1})\cap Z(\mathbf{\bar g}_2)^{\perp}$  are trivial in $\mathfrak{sl}_2(k)$ implies that $U_1, U_2, \ldots, U_{m-1}$ and $U$ are all $\mathbf{0}$. As $U = \mathbf{0}$, therefore Equation~(\ref{eq56}) turns out to be 
$$\left\langle\left(I-\mathrm{Ad}_{\bar e_m}\right)V,Y\right\rangle=0\ \forall Y\in \mathfrak{sl}_2(k).$$ 
Consequently, by the non-degeneracy of the bilinear form, we get 
 
\begin{equation}{\label{eqmain2}}
 \left(I-\mathrm{Ad}_{\bar e_m}\right)V =\mathbf{0}.
\end{equation}
 
By using the condition $U = \mathbf{0}$ together with Equation~(\ref{eqmain2}), we obtain $V\in Z(\mathbf{\bar g}_2) \cap Z(\bar e_m)$. Invoking the assumption that $H_m$ does not have any nontrivial fixed vector in $\mathfrak{sl}_2^*(k)$, we obtain $V = \mathbf{0}$. This further implies $T = \mathbf{0}$ and hence $\im(\overline{De}_{m+1})^\perp \cap \mathfrak{sl}_2(k) = \{\mathbf{0}\}$. Therefore, the dimension formula $dim(\im(\overline{De}_{m+1})) + dim(\im(\overline{De}_{m+1})^\perp) = dim (\mathfrak{sl}_2(k))$ implies $coker(\overline{De}_{m+1}) = \mathfrak{sl}_2(k)/ \im(\overline{De}_{m+1}) = \{\mathbf{0}\}$. Now consider the short exact sequence 
\[\begin{tikzcd}
	\mathbf{0} & {\im(De^{\R}_{m+1})} & {\mathfrak{sl}_2(\mathcal{O}_2)} & {\mathfrak{sl}_2(\mathcal{O}_2)/\im(De^{\R}_{m+1})} & \mathbf{0}.
	\arrow[from=1-1, to=1-2]
	\arrow[from=1-2, to=1-3]
	\arrow[from=1-3, to=1-4]
	\arrow[from=1-4, to=1-5]
\end{tikzcd}\]
Tensoring with $\mathcal{O}_2/\mathfrak{m}_2$, and using right exactness we obtain $\mathfrak{sl}_2(\mathcal{O}_2)/\im(De^{\R}_{m+1})\bigotimes_{\mathcal{O}_2}\mathcal{O}_2/\mathfrak{m}_2\cong coker(\overline{De}_{m+1}) = \{\mathbf{0}\}$. Let us denote $\mathcal{M} = \mathfrak{sl}_2(\mathcal{O}_2)/ \im(De^{\R}_{m+1})$; clearly it is an $\mathcal{O}_2$ module. Therefore $\mathcal{M}/ \mathfrak{m}_2\mathcal{M} = \{\mathbf{0}\}$. Invoking \cite[Proposition 2.6]{Atiyahmacdonald} we obtain $\im(De^{\R}_{m+1}) = \mathfrak{sl}_2(\mathcal{O}_2)$.

{\bf Step III:} Therefore, if $A\in \SL_2(\mathcal{O})$ is a lift of $\bar e_{m+1}(\mathbf{\bar g}_1, \mathbf{\bar g}_2)$ then there exists $(\mathbf{h}_1, \mathbf{h}_2)\in \SL_2(\mathcal{O}_2)^2$ such that $\mathbf{h}_i$ are lifts of corresponding $\mathbf{\bar g}_i$ for $i=1,2$ and  $e_{m+1}(\mathbf{h}_1, \mathbf{h}_2) = A_2$. The kernel of the reduction map $\SL_2(\mathcal{O}_{j+1}) \rightarrow \SL_2(\mathcal{O}_j)$ is $\{I + \delta_j P \mid  \delta_j\in ker (\theta_j), \ P\in \mathfrak{sl}_2(\mathcal{O}_{j+1})\}$ here $\delta_j$ is the generator of $ker(\theta_j)$. Note that $(ker(\theta_j))^2=0$ implies $\delta_j^2=0$. Therefore in the context of lifting of elements from $\im(e_{m+1}^{\mathcal{O}_j})$ to $\im(e_{m+1}^{\mathcal{O}_{j+1}})$, one can observe $e_{m+1}^{\mathcal{O}_{j+1}}(B_1(I + \delta_jX), B_2(I + \delta_jY)) = e_{m+1}^{\mathcal{O}_{j+1}}(B_1, B_2)[1 + \delta_j De_{m+1}^{\mathcal{O}_{j+1}}(X, Y)]$ where $B_1, B_2\in \SL_2(\mathcal{O}_{j+1})$ are the lifts of $\mathbf{\bar g}_1$ and $\mathbf{\bar g}_2$ respectively (achieved through successive lifting of $\mathbf{\bar g}_1$ and $\mathbf{\bar g}_2$ at each $\mathcal{O}_{\ell}$ level for $\ell \leq j+1$), moreover $X, Y\in \mathfrak{sl}_2(\mathcal{O}_{j+1})$ and $De_{m+1}^{\mathcal{O}_{j+1}}\colon  \mathfrak{sl}_2(\mathcal{O}_{j+1}) \times \mathfrak{sl}_2(\mathcal{O}_{j+1}) \rightarrow \mathfrak{sl}_2( \mathcal{O}_{j+1}) $ is the derivative map corresponding to $e_{m+1}$ at $(B_1, B_2)$ in $j+1$-th level defined earlier at level $2$. As $\mathfrak{sl}_2(\mathcal{O}_{j+1})/ \im(De^{\mathcal{O}_{j+1}}_{m+1})\bigotimes_{\mathcal{O}_{j+1}}\mathcal{O}_{j+1}/\mathfrak{m}_{j+1}\cong coker(\overline{De}_{m+1}) = \{\mathbf{0}\}$, therefore Nakayama's lemma \cite{Atiyahmacdonald} guarantees that all these $De_{m+1}^{\mathcal{O}_{j+1}}$ are surjective for $j\geq 1$. Therefore there exists $(P_{j+1}, Q_{j+1})\in \SL_2(\mathcal{O}_{j+1})^2$ which is the lift of corresponding $(P_j, Q_j)\in \SL_2(\mathcal{O}_j)$ such that $A_{j+1} = e_{m+1}^{\mathcal{O}_{j+1}}(P_{j+1}, Q_{j+1})$ for each $j\geq 1$; where $(P_1, Q_1) = (\mathbf{\bar g}_1, \mathbf{\bar g}_2)$ and $(P_2, Q_2)=(\mathbf{h}_1, \mathbf{h}_2)$. Applying Lemma~\ref{lifting complete level} yields the required result. 
\end{proof}

\section{Lift of Cyclic elements as image of Engel maps}{\label{cyclic as Engel}}

In this section, we look at cyclic elements in two parts: first, regular semisimple ones, and then non-regular semisimple ones. First, we prove that any non-trivial unipotent matrix in $\SL_2(k)$ satisfies the assumptions mentioned in Proposition~\ref{adelic 2}.
\begin{proposition}{\label{fixed vector 2}}
Let $A\in \SL_2(k)$ be a non-scalar matrix such that $A = [P, Q]$ for some $P$ and $Q$ in $\SL_2(k)$. Then the subgroup $H=\langle P, Q\rangle$ does not have any non-trivial fixed vector under the co-adjoint action on $\mathfrak{sl}_2^*(k)$. 
\end{proposition}
\begin{proof}
Since $A$ is non-scalar, the elements $P, Q$ are both non-scalar, hence cyclic in $\SL_2(k)$. On contrary suppose $H$ has a fixed vector $\phi$ in $\mathfrak{sl}_2^*(k)$. Then, $\mathrm{Ad}_P^*\phi(X)= \phi(X) = \mathrm{Ad}_Q^*\phi(X)$ for all $X\in \mathfrak{sl}_2(k)$. Now, since the bilinear form $\langle, \rangle$ is nondegenerate, the induced map $ \mathfrak{sl}_2(k) \rightarrow \mathfrak{sl}_2^*(k)$ defined by $E\mapsto \langle E,\ -\ \rangle$ is an isomorphism of the vector spaces (see Section~\ref{towards Shalev's conj}). Under this isomorphism, for $\phi\in \mathfrak{sl}_2^*(k)$, we choose $V\in \mathfrak{sl}_2(k)$, that satisfies $\phi = \langle V,\ -\ \rangle$. Thus, using this and the definition $\mathrm{Ad}_P^*\phi(X) = \phi(\mathrm{Ad}_{P^{-1}}(X)) = \langle V, \mathrm{Ad}_{P^{-1}}(X)\rangle$ and similarly for $Q$, we get $\langle \mathrm{Ad}_P(V), X\rangle = \langle V, X \rangle = \langle \mathrm{Ad}_Q(V), X\rangle$ for all $X$. This implies,  $V\in Z(P)\cap Z(Q)$. We need to show that $\phi = 0$, which we do by showing that $V=0$.
 
Now, since $Q$ is cyclic and $V\in Z(Q)$, we get $V = f(Q)$ for some $f(x)\in k[x]$. Then, $V\in Z(P)$ gives $V = PVP^{-1} = f(PQP^{-1}) = f(AQ)$. Now consider $f(x)$ and the characteristic polynomial $\chi_Q(x)$ and write $f(x)= g(x)\chi_Q(x) + h(x)$ for some $g(x), h(x)$ where $h(x)=0$ or $deg(h(x)) < 2$. We claim here that $h(x)=0$. If not, then we write $h(x) = \alpha x + \beta$ for some $\alpha, \beta \in k$. As $PQP^{-1}$ and $Q$ are similar in $\GL_2(k)$, therefore $\chi_Q(x) = \chi_{PQP^{-1}}(x)$.  Then, $V= f(Q) = f(PQP^{-1})=f(AQ)$ implies $ h(Q) = h(AQ) \implies \alpha Q + \beta I = \alpha AQ + \beta I$, that is, $\alpha(I-A)Q=\mathbf{0}$. Therefore, $\alpha (I-A)=\mathbf{0}$. Since $A$ is non-scalar, $I-A$ is nonzero, i.e., at least one of the entries of $I-A$ is nonzero in $k$. This implies $\alpha = 0$. Therefore $V= f(Q) = \beta I$. Now $V\in \mathfrak{sl}_2(k)$ implies $\beta=0$. Therefore, the only possibility is $h(x)=0$, consequently $V = \mathbf{0}$. Hence, $\phi$ is the zero vector in $\mathfrak{sl}_2^*(k)$.
\end{proof}

\begin{lemma}{\label{fixed vector ortho compl}}
Let $k$ be a perfect field of characteristic $\neq 2$ and $A\in \SL_2(k)$ be a non-scalar matrix. Suppose $A=[P, Q]$ for some $P,Q$ in $\SL_2(k)$ and $Z(Q)\cap Z(A)=\{\mathbf{0\}}$ in $\mathfrak{sl}_2(k)$. Then we have the following: 
\begin{enumerate}
\item If $Q$ is not a regular semisimple element then $Z(PQP^{-1})\cap Z(Q)^{\perp}= \{\mathbf{0}\}$ in $\mathfrak{sl}_2(k)$. 
\item If $Q$ is regular semisimple and either $\tr(A)=2$ or
$2 \tr(A)\neq \tr(Q)^2$ then $Z(PQP^{-1})\cap Z(Q)^{\perp}=\{\mathbf{0}\}$ in $\mathfrak{sl}_2(k)$. 
\end{enumerate}
\end{lemma}
\begin{proof}
Let $Y\in Z(PQP^{-1})\cap Z(Q)^{\perp}$. Then, $A=PQP^{-1}Q^{-1}$ gives $Y\in Z(PQP^{-1})=Z(AQ) $.
Now, $AQ$ must be a non-scalar, otherwise $AQ=PQP^{-1}$ gives $Q$ is a scalar and hence $Z(Q)=\mathfrak{sl}_2(k)$ contradicts $Z(Q)\cap Z(A)=\{\mathbf{0\}}$. Therefore, $AQ$ is cyclic and $Y = a AQ + bI$ for some $a, b\in k$. Since,  $Y\in Z(Q)^{\perp}$ we have $\langle Y, \mathfrak{X}\rangle =  tr(Y\mathfrak X) =0$ for all $\mathfrak{X}\in Z(Q)$. This gives, $a\tr(AQ\mathfrak{X}) +b \tr(\mathfrak{X}) = 0$, which further implies $a\tr(AQ\mathfrak{X})=0$ since $\mathfrak{X}\in Z(Q) \subset \mathfrak{sl}_2(k)$. Let us consider two separate cases:
   
\noindent \textbf{Case I: $Q$ is not regular semisimple:} We claim that $\tr(AQ\mathfrak{X})\neq 0$. 
Since $Q$ is not regular semisimple, it is of the form $\left(\begin{array}{cc} 1 & 1 \\ 0 & 1 \end{array}\right)$ or $\left(\begin{array}{cc} -1 & 1 \\ 0 & -1 \end{array}\right)$ up to $\GL_2$-conjugates. Let $J\in \GL_2(k)$ such that $JQJ^{-1}=Q'=\left(\begin{array}{cc} 1 & 1 \\ 0 & 1 \end{array}\right)$. Take $\mathfrak{X}' = J \mathfrak{X}J^{-1} = \delta\left(\begin{array}{cc} 0 & 1 \\
0 & 0 \end{array}\right)$ for some $\delta \neq 0$ in $k$. Write $A' = JAJ^{-1} = \left(\begin{array}{cc} u_1 & u_2 \\ u_3 & u_4 \end{array}\right)$.        
Now $\tr(A'Q'\mathfrak{X}') = \delta \tr (\left(\begin{array}{cc}u_1 & u_2 \\ u_3 & u_4   \end{array} \right) \left(\begin{array}{cc} 1 & 1 \\ 0 & 1 \end{array}\right) \left(\begin{array}{cc} 0 & 1 \\ 0 & 0 \end{array}\right)) = \delta u_3$ and $\tr(A'Q') = u_1+u_3+u_4= u_3+\tr(A') = Tr(AQ) = Tr(PQP^{-1}) = 2$. Therefore, $u_3 = 2- \tr(A)$. 
        
If $u_3 = 0$ then $\mathbf{\tr(A') = 2 =\tr(A)}$. We have $\det(A) = u_1u_4 = 1$ and $tr(A) = u_1 + u_4 = 2$, this gives $u_1 = u_4 = 1$. Since $A'$ is non-scalar, $u_2\neq 0$ and we can write $A' = I + u_2 \delta^{-1}\mathfrak{X}'$. Note that $\mathfrak{X}'$ commutes with $A'$. Therefore, $\mathfrak{X}' \in Z(A') \cap Z(Q')$, which further implies $\mathfrak{X}\in Z(A) \cap Z(Q) = 0$. Thus, $Z(Q)=0$, a contradiction. Thus, $u_3\neq 0$ when $Q'=\left(\begin{array}{cc} 1 & 1 \\ 0 & 1
\end{array}\right)$. A similar argument for the case $Q' = \left(\begin{array}{cc} -1 & 1 \\ 0 & -1 \end{array}\right)$ shows that $u_3 \neq 0$. Therefore when $Q$ is cyclic but not regular semisimple then $\delta u_3 = \tr(A'Q'\mathfrak{X}') = \tr(JAQ\mathfrak{X}J^{-1})=\tr(AQ\mathfrak{X})\neq 0$. 

Now, $\tr(A'Q'\mathfrak{X}') = \tr(AQ\mathfrak{X})\neq 0$ but $a\tr(AQ\mathfrak{X})=0$ gives $a=0$. Hence, $Y=bI$ in $\mathfrak {sl}_2(k)$ gives $Y=0$. This proves the statement.

\noindent \textbf{Case II: $Q$ is regular semisimple:} Over $\Bar{k}$ we can write $Q'=JQJ^{-1}=\left(\begin{array}{cc}   \lambda  & 0 \\ 0 & \lambda^{-1} \end{array}\right)$ for some $J\in \GL_2(\Bar{k})$; where $\lambda\neq \lambda^{-1}$. The elements of $Z(Q')$ in $\mathfrak{sl}_2(\Bar{k})$ are of the form $\mathfrak{X}' = J\mathfrak{X}J^{-1} = \gamma\left(\begin{array}{cc} 1  & 0 \\ 0 & -1 \end{array}\right)$ for some $\gamma\neq 0$ in $\bar k$. Let $A' = JAJ^{-1} = \left(\begin{array}{cc} u_1 & u_2 \\    u_3 & u_4 \end{array}\right)$ in $\SL_2(\bar k)$.  Now $\tr(A'Q'\mathfrak{X}') = \gamma\tr (\left(\begin{array}{cc}
u_1 & u_2 \\ u_3 & u_4 \end{array}\right)\left(\begin{array}{cc} \lambda & 0 \\ 0 & \lambda^{-1} \end{array}\right) \left(\begin{array}{cc} 1 & 0 \\ 0 & -1   \end{array}\right)) = \gamma(u_1\lambda-u_4\lambda^{-1})$ and $\tr(A'Q') = u_1 \lambda + u_4\lambda^{-1}$ which is equal to $\tr(Q') = \lambda + \lambda^{-1}$ because of the commutator relation $[P', Q'] = A'$, where $P'=JPJ^{-1}$. Therefore, we obtain 
\begin{equation}\label{eq5}
\lambda(1 - u_1)+ \lambda^{-1}(1-u_4)=0. 
\end{equation} 
Now we deal with the two subcases.

\noindent $\mathbf{(i)}$ \textbf{When} $\mathbf {\tr(A)=2}$: Then $\tr(A')=2$, i.e., $u_1+ u_4 = 2$, which further implies $(1- u_1) = (u_4 -1)$. This and Equation~(\ref{eq5}) together implies $(1-u_1)(\lambda-\lambda^{-1})=0$. As $\lambda\neq \lambda^{-1}$ therefore $u_1=1$ which gives $u_4=1$. Hence $\tr(A'Q'\mathfrak{X}')=\gamma(\lambda-\lambda^{-1})\neq 0$. 

\noindent $\mathbf{(ii)}$ \textbf{When} $ \mathbf{2\tr(A)\neq} (\tr(Q))^2\ $: It gives $2\tr(A')\neq (\tr(Q'))^2$. Let $\tr(A') = 2\sigma$ for some $\sigma \in k$. Then $u_1-\sigma = \sigma - u_4$. Using this, from the Equation~(\ref{eq5}) we obtain $(u_1-\sigma)(\lambda-\lambda^{-1})=(1-\sigma)\tr(Q')$. This implies $u_1 = (\tr(Q')-2\sigma\lambda^{-1})(\lambda-\lambda^{-1})^{-1}$ and $u_4=-(\tr(Q')-2\sigma\lambda)(\lambda-\lambda^{-1})^{-1}$. Then, $\tr(A'Q'\mathfrak{X}') = \gamma((\tr(Q'))^2-2\tr(A'))(\lambda-\lambda^{-1})^{-1}\neq 0$.
       
Therefore, in all the cases above $\tr(A'Q'\mathfrak{X}')\neq 0$, hence $\tr(AQ\mathfrak{X})\neq 0$. Then $a \tr(AQ\mathfrak{X}) = 0$ implies $a=0$, which further implies $Y=bI$. Now $\tr(Y)=0$ implies $b=0$ hence  $Y=\mathbf{0}$.
\end{proof}

\begin{remark}
When $Q$ is regular semisimple in Lemma~\ref{fixed vector ortho compl} and either $\tr(A)=0$ or $\tr(Q)= 0$, then $2\tr(A)\neq (\tr(Q))^2$ implies $Y=0$.
\end{remark}


\subsection{Regular Semisimples as elements of $\im(e_{m+1})$}{\label{reg-sem}}

In this section, for a regular semisimple element $M\in \SL_2(\R)$, we aim to get a pair $(\widehat{\mathbf{g}}_1, \widehat{\mathbf{g}}_2) \in \SL_2(\R)\times \SL_2(\R)$ such that $M$ can be written as $\R$-conjugate to $e^{\R}_{m+1}(\widehat{\mathbf{g}}_1, \widehat{\mathbf{g}}_2)$. For this, we need the existence of a solution of the equation $\Psi_{m+1}(s, u, v) = \lambda$ in $\R^3$ as a lift of a solution from the field level as described in Section~\ref{sub-trace}. In particular, we need the expression of the trace map corresponding to the $m+1$-th Engel word over $\GL_2(\mathcal{O})$, and we need to restrict our attention to some specific subgroup of $\GL_2(\mathcal{O})$. For a regular semisimple element in $\SL_n(\mathcal{O})$, we translate this idea to each $\mathcal{O}_{\ell}$ level and use the same to achieve lifting results at the $\mathcal{O}$ level. For $m \geq 1$, $\Psi_{m+1}(s, u, v) = \Psi_{m+1}(s_m, v)$ (see Section~\ref{mult-lift}); therefore the existence of a solution $(s_0, u_0, v_0)\in \R^3$ of the equation $\Psi_{m+1}(s, u, v) =\lambda$ guarantees that $(s_m(s_0, u_0, v_0), v_0)$ is a solution of the equation $\Psi_{m+1}(s_m, v) = \lambda$. The following diagram illustrates the process of this lifting (will be helpful in Proposition~ \ref{regsem1}).
\begin{figure}[htbp]
\scriptsize
\[\begin{tikzcd}
\begin{array}{c} \Psi_{m+1}(s,u,v)=\lambda\\(\widehat s_0,\widehat u_0,\widehat v_0) \end{array} && \begin{array}{c} \psi_{m+1}(s,u,v)=\bar \lambda\\(s_0,u_0,v_0) \end{array} \\
\begin{array}{c} (\widehat{\mathbf{g}}_1,\widehat{\mathbf{g}}_2)\in\SL_2(\R)\\ e_{m+1}^{\R}(\widehat{\mathbf{g}}_1,\widehat{\mathbf{g}}_2)\sim_{\R} A \end{array} && \begin{array}{c} (A_0,B_0)\in \SL_2(k)\\ \Bar{e}_{m+1}(A_0,B_0)=\Bar{A} \end{array}
	\arrow["\theta", curve={height=-12pt}, from=1-1, to=1-3]
	\arrow[from=1-1, to=2-1]
	\arrow["lift"', curve={height=-12pt}, from=1-3, to=1-1]
	\arrow["\theta", from=2-1, to=2-3]
	\arrow[from=2-3, to=1-3]
\end{tikzcd}\]
\caption{Lifting of solutions for the Engel word $e_{m+1}$}
\label{fig: lifting com}
\end{figure}

For a given $m\geq 1$, by Proposition~\ref{adelic 2},~\ref{fixed vector 2}, and Lemma~\ref{fixed vector ortho compl}, it is evident that any lift of a regular semisimple matrix $\bar e_{m+1}(h_1,h_2)$ where $h_1, h_2\in \SL_2(k)$, $h_2$ is regular semisimple, and $2\tr(\bar e_{i}(h_1, h_2)) \neq tr(h_2)^2$ for all $i=0, 1, \ldots, m$; must be an element of $\im(e_{m+1}^{\mathcal{O}})$ in $\SL_2(\mathcal{O})$. However, there are examples of regular semisimple elements $A = \bar e_{m+1}(h_1, h_2)$ in $\SL_2(k)$ such that $2\tr(\bar e_{m}(h_1, h_2)) = tr(h_2)^2$ and $h_2$ is regular semisimple.
\begin{example}
In $\SL_2(\mathbb{F}_7)$, let us take $A= \left(\begin{array}{cc} 1 & 6 \\ 1 & 0 \end{array}\right)$. Since $\tr(A)\neq \pm 2$, the element $A$ is regular semisimple in $\SL_2(\mathbb{F}_7)$. There exist elements, for example, $h_1= \left(\begin{array}{cc} 6 & 2 \\ 4 & 5 \end{array} \right)$ and $h_2 = \left(\begin{array}{cc} 2 & 3 \\ 1 & 2    \end{array}\right)$ in $\SL_2(\mathbb{F}_7)$ such that $\bar e_3(h_1, h_2) = A$. It is easy to observe $\bar e_2(h_1, h_2) = \left(\begin{array}{cc} 5 & 2 \\ 0 & 3 \end{array}\right)$. Therefore, $\tr(\bar e_2(h_1, h_2)) = 1$ and $\tr(h_2) = 4$. This implies $2\tr(\bar e_{2}(h_1,h_2)) = tr(h_2)^2$.
\end{example}
\noindent In this section, we show that for this kind of elements $A$, every lift of $A$ in $\SL_2(\mathcal{O})$ must belong to $\im(e_{m+1}^{\mathcal{O}})$ provided $A\in \im(\bar e_{m+1})$ in $\SL_2(k)$.

\begin{lemma}{\label{matsl}}
Let $A\in \GL_2(\mathcal{O}_{\ell + 1})$ be such that $\theta_{\ell}(A) \in \SL_2(\mathcal{O}_{\ell})$ for some $\ell\geq 1$. Then, there exists $Q_1 \in \SL_2( \mathcal{O}_{ \ell + 1})$ and $Q_2 \in \M_2(\mathcal{O}_{\ell +1})$ such that $A = Q_1 + \pi_{\ell+1}^{\ell} Q_2$ where $\pi_{\ell + 1} = \pi +(\pi^{\ell+1})$ is the generator of the maximal ideal $\mathfrak{m}_{\ell+1}$ of $\mathcal{O}_{\ell+1}$. 
\end{lemma}
\begin{proof}
Since $\theta_{\ell}(A) \in \SL_2(\mathcal{O}_{\ell})$, we get $det(A) = 1 + \pi_{\ell+1}^{\ell} v$ for some $v \in \mathcal{O}_{\ell + 1}$ because $\pi_{\ell+1}^{\ell}$ is the generator of the $ker(\theta_{\ell})$. We want to find $B$ such that $C = A + \pi_{\ell+1}^{\ell} B$ with $det(C) = 1$. If such $B$ exists then $det(A) + \pi_{\ell+1}^{\ell} \tr(adj(A)B)=1$. This is because $det(C)=1$ and $\det(C) =\det(A) \det(I + \pi_{\ell+1}^{\ell}A^{-1}B)$. Now note that $A^{-1} =(\det A)^{-1} adj(A)$ and $(\pi_{\ell+1}^{\ell})^2=0$ for $\ell\geq 1$, therefore the result follows. This implies $\pi_{\ell+1}^{\ell}(v + \tr(adj(A)B))=0$. Thus we can choose $B = -2^{-1}v (adj(A))^{-1}$ as $adj(A)$ is invertible. Consider $Q_1 = C = A-2^{-1}\pi_{\ell + 1}^{\ell} v(adj(A))^{-1}$ and $Q_2 = -B = 2^{-1}v(adj(A))^{-1}$.
\end{proof}
\noindent The Lemma~\ref{matsl} can be generalised for $A\in \GL_n(\mathcal{O}_{\ell+1})$ where $n$ is coprime to the characteristic exponent of $k$.

\begin{lemma}{\label{lem: trace GL to trace SL}}
For $\ell\geq 1$, let $\widehat{A}_0$ and $\widehat{B}_0$ be two matrices in $\GL_2(\mathcal{O}_{\ell+1})$, such that $\theta_{\ell}(\widehat{A}_0), \theta_{\ell}(\widehat{B}_0)\in \SL_2(\mathcal{O}_{\ell})$ and $\tr(e_{m+1}^{\mathcal{O}_{\ell+1}}(\widehat{A}_0, \widehat{B}_0)_{\GL}) = \lambda$, for some $\lambda \in \mathcal{O}_{\ell + 1}$. Then, there exist two matrices $\widehat{\mathbf{g}}_1, \widehat{\mathbf{g}}_2$ in $\SL_2(\mathcal{O}_{\ell + 1})$ such that $\tr(e_{m+1}^{\mathcal{O}_{\ell + 1}}(\widehat{\mathbf{g}}_1, \widehat{\mathbf{g}}_2)) = \lambda$, where $\theta_{\ell}(\widehat{\mathbf{g}}_1) = \theta_{\ell}(\widehat{A}_0)$ and $\theta_{\ell}(\widehat{\mathbf{g}}_2) = \theta_{\ell}(\widehat{B}_0)$.
\end{lemma}
\begin{proof}
From the previous Lemma~\ref{matsl}, we can write $\widehat{A}_0 = \mathbf{g}_0 + \pi_{\ell+1}^{\ell} A_1$ and $\widehat{B}_0 = \mathbf{h}_0 + \pi_{\ell+1}^{\ell} B_1$, for some $\mathbf{g}_0, \mathbf{h}_0 \in \SL_2(\mathcal{O}_{\ell+1})$ and $A_1, B_1\in \M_2(\mathcal{O}_{\ell+1})$. Let $\tr(\mathbf{g}_0^{-1}A_1) = \alpha_0$ and $\tr(\mathbf{h}_0^{-1}B_1) = \beta_0$. Consider $\Tilde{A}_1 = A_1 - 2^{-1}\alpha_0\mathbf{g}_0$ and $\Tilde{B}_1 = B_1 - 2^{-1}\beta_0\mathbf{h}_0$. Now consider $\widehat{\mathbf{g}}_1 = \mathbf{g}_0 + \pi_{\ell+1}^{\ell} \Tilde{A}_1$ and $\widehat{\mathbf{g}}_2 = \mathbf{h}_0 + \pi_{\ell+1}^{\ell} \Tilde{B}_1$. Then, $$det(\widehat{\mathbf{g}}_1) = 1+\pi_{{\ell}+1}^{\ell} \tr(\mathbf{g}_0^{-1}\Tilde{A}_1) = 1 + \pi_{\ell+1}^{\ell}(\alpha_0-\alpha_0) = 1.$$ Similarly, $det(\widehat{\mathbf{g}}_2) = 1$. Therefore, $\widehat{\mathbf{g}}_1, \widehat{\mathbf{g}}_2$ are the members of $\SL_2(\mathcal{O}_{\ell + 1})$. Let $P_0 = 2^{-1} \alpha_0 \mathbf{g}_0$ and $Q_0 = 2^{-1} \beta_0 \mathbf{h}_0$, then $\Tilde{A}_1 = A_1 - P_0$ and $\Tilde{B}_1 = B_1 - Q_0$. Now, invoking Lemma~\ref{sum1} by taking $A_0 = \mathbf{g}_0$ and $B_0=\mathbf{h}_0$, we get 
$$L_{m+1}(\Tilde{A}_1, \Tilde{B}_1) = L_{m+1}(A_1, B_1) - L_{m+1}(P_0, Q_0) = L_{m+1}(A_1, B_1)$$ 
as $L(P_0,Q_0)=0$ by Corollary~\ref{scl}. Now, using Lemma~\ref{sum1} we get, 
\begin{eqnarray*}
\tr(e_{m+1}^{\mathcal{O}_{\ell+1}}(\widehat{\mathbf{g}}_1, \widehat{\mathbf{g}}_2)) &= & \tr(e_{m+1}^{\mathcal{O}_{\ell+1}}(\mathbf{g}_0, \mathbf{h}_0)) + \pi_{\ell+1}^{\ell} L_{m+1}(\Tilde{A}_1, \Tilde{B}_1) \\
&=& \tr(e_{m+1}^{\mathcal{O}_{\ell + 1}}(\mathbf{g}_0, \mathbf{h}_0)) + \pi_{\ell + 1}^{\ell} L_{m+1}(A_1, B_1) \\
&=& \tr(e_{m+1}(\widehat{A}_0, \widehat{B}_0)_{\GL})=\lambda.
\end{eqnarray*}
Since the kernel of the natural surjection $\mathcal{O}_{\ell+1} \rightarrow \mathcal{O}_{\ell}$ is generated by $\pi_{\ell+1}^{\ell}$, therefore $\theta_{\ell}(\widehat{\mathbf{g}}_1) = \theta_{\ell}(\mathbf{g}_0) = \theta_{\ell}(\widehat{A}_0)$. A similar argument shows that $\theta_{\ell}(\widehat{\mathbf{g}}_2) = \theta_{\ell}(\mathbf{h}_0) = \theta_{\ell}(\widehat{B}_0)$.
\end{proof}

\begin{lemma}{\label{conjugacycyclic}}
Let $\mathcal{O}$ be a local principal ideal ring, complete with respect to its unique maximal ideal $\mathfrak{m} = (\pi)$, with residue field $k$ of characteristic $\neq 2$. Let $M$ and $P$ be two cyclic matrices in $\GL_2(\mathcal{O})$ with the same characteristic polynomial. Then, $M \sim_{\mathcal{O}} P$.
\end{lemma}
\begin{proof}
First, we show that in this case the null ideals of both matrices $M$ and $P$ are the same, that is, $\langle\chi_M(t) \rangle$ (note $\chi_M(t) = \chi_P(t)$). First, we show that the minimal polynomials are the same too. 
     
Let $J(t)\in N_M$, the null ideal of $M$. Then by \cite[Theorem 2.14]{JacobsonAlgebraBook85} there exist $U(t)$ and $R(t)$ in $\mathcal{O}[t]$ such that $J(t) = U(t) \chi_M(t) + R(t)$ where $R(t) = 0$ or $\deg(R(t)) < \deg(\chi_M(t))$. Under the isomorphism $\mathcal{O}[t] \rightarrow \varprojlim\limits_{j\geq 1} \frac{\mathcal{O}}{\mathfrak{m}^j\mathcal{O}}[t]$ induced by the isomorphism $\mathcal{O}\rightarrow \varprojlim\limits_{j\geq 1} \mathcal{O}/\mathfrak{m}^j\mathcal{O}$, say, the image of $R(t)$ is $(R_1(t), R_2(t),\ldots)$ and the image of $\chi_M(t)$ is $(\chi_{\Bar{M}}(t),\chi_{M_2}(t),\ldots)$. Now for each $i$, $R_i(t)\in N_{M_i}$. Since $\chi_M(t)$ is monic polynomial, $\deg(\chi_M(t)) = \deg(\chi_{M_i}(t))$ for all $i$. By \cite[Corollary 2.3]{PRS2026}, we obtain $R_i(t)$ must be zero polynomial for each $i$. Hence, $R(t)$ must be the zero polynomial in $\mathcal{O}[t]$. This also ensures that the minimal polynomial of both matrices is their characteristic polynomial.

For a matrix $Q\in \M_2(\mathcal{O})$, let us denote by $\M^Q$ the $\mathcal{O}[t]$ module structure on $\mathcal{O}^2$ via the action of $t$ by $Q$. For cyclic $Q$, this module $\M^Q \cong \frac{\mathcal{O}[t]}{\langle \chi_Q(t)\rangle}$ as $\mathcal{O}[t]$ module. As both $M$ and $P$ are cyclic by \cite[Theorem 5.1 (2)]{Snapper1950},  the factorization of $\chi_{M}(t)$ and $P$ in terms of primary components is unique up to associates. Moreover $\chi_M(t)=\chi_{P}(t)$. Therefore $\M^M$ and $\M^{P}$ are isomorphic as $\mathcal{O}[t]$ modules and hence $M\sim _{\mathcal{O}} P$. 
 \end{proof}

At this stage, let us recall the expression of $\tr(e_{m+1}(X, Y)_{\mathfrak{H}})$ from Section~\ref{trace in GL}. By definition of $\mathfrak{H}$, it is obvious that $\theta(d_P) = 1$ for all $P \in \mathfrak{H}$. Therefore, the existence of a solution to the equation 
$$\widetilde F_{m+1}^{\mathcal{A}}(x_1, x_2, x_3, x_4, y_1, y_2, y_3, y_4) = 0$$ 
in $\mathcal{A}^8$ will ensure the existence of a solution of $\Phi_{m+1}(s_{m_{\mathfrak{H}}}, v_{\mathfrak{H}}, d_{Y})=0$ in $\mathcal{A}^3$, and these facts further imply the existence of matrices $X_0, Y_0\in \mathfrak{H}$ which satisfies $\tr(e_{m+1}(X_0, Y_0)_{\mathfrak{H}}) = \lambda$. 
That is, if $(a_1, a_2, a_3, a_4, b_1, b_2, b_3, b_4)\in \mathcal{A}^8$ is a solution of  $\widetilde F_{m+1}^{\mathcal{A}}(x_1, x_2, x_3, x_4, y_1, y_2, y_3, y_4) = 0$ with $\theta(A_0), \theta(B_0)\in \SL_2(k)$ where $A_0=\left(\begin{array}{cc}  a_1 & a_2 \\  a_3 & a_4
\end{array}\right),\ B_0=\left(\begin{array}{cc}  b_1 & b_2 \\
b_3 & b_4 \end{array}\right)$ then it satisfies 
$$\tr(e_{m+1}(A_0, B_0)_{\mathfrak{H}}) = \lambda,$$ 
which is implied by $\Phi_{m+1}(s_{m_{\mathfrak{H}}}, v_{\mathfrak{H}}, d_{Y})=0$.

Keeping this fact in mind, now we write down the main result of this section for regular semisimple elements.
\begin{proposition}{\label{regsem1}}
Let $\mathcal{O}$ be a local principal ideal ring, complete with respect to its unique maximal ideal $\mathfrak{m}=(\pi)$, with residue field $k$ of characteristic $\neq 2$. Let $M$ be a regular semisimple element in $\SL_2(\mathcal{O})$ such that for a given $m\geq 1$, $\overline{M}\in \im(e_{m+1})$ in $\SL_2(k)$. Then, $M$ belongs to $\im(e_{m+1}^{\mathcal{O}})$ in $\SL_2(\mathcal{O})$.
\end{proposition}
\begin{proof}
\textbf{Step I :} It is enough to show the proposition for the length two case. Let $A\in \SL_2(\R)$ be regular semisimple. 
Therefore, $\tr(\Bar{A})$ cannot take the values $\pm 2$. Hence, at this stage, we will not consider those $A\in \SL_2(\R)$ whose trace values are $\pm 2 + m_1$ type for some $m_1\in \mathfrak{m}_2$. 

Now, let $\lambda = tr(A)$ and assume that $X =\left (\begin{array}{cc} x_1& x_2 \\x_3&x_4 \end{array}\right )$ and $Y = \left (\begin{array}{cc} y_1&y_2  \\ y_3&y_4 \end{array} \right )$ be two matrix variables from $\mathfrak{H}$ (consider $\mathcal{A} = \mathcal{O}_2$ here) each with $4$ indeterminate entries from $\R$. Let $\tr(X)=s, \tr(Y)=v, \tr(XY)=u$. Consider the equation $\widetilde F_{m+1}^{\R}(x_1, x_2, x_3, x_4, y_1, y_2, y_3, y_4) = 0$ over $\mathcal{O}_2$. Under the reduction mod $\mathfrak{m}$ let $\tilde \Phi_{m+1}$ goes to $\tilde \psi_{m+1}$ where $\tilde \psi_{m+1}(s_m, v) := \psi_{m+1}(s_m, v) - \bar \lambda$ (see Section~\ref{tr} for the expression of $\psi_{m+1}(s_m,v)$). In Section~\ref{trace in GL}, we have seen 
$$\widetilde \Phi(s_{m_{\mathfrak{H}}}, v_{\mathfrak{H}}, d_{Y}, d_X) = \widetilde F_{m+1}^{\R}(x_1, x_2, x_3, x_4, y_1, y_2, y_3, y_4).$$
Our aim is to show that there exists $\widehat{\mathbf{g}}_1, \widehat{\mathbf{g}}_2 \in \SL_2(\R)$ such that $\tr(e_{m+1}^{\R}(\widehat{\mathbf{g}}_1, \widehat{\mathbf{g}}_2)) = \lambda$. 

It is easy to observe that $\tilde \psi_{m+1}(s_m, v)$ can be written as $\widetilde\psi_{m+1}(s, u, v)$, which can further be written as a polynomial in the $8$ variables (entries of matrix variables in $\SL_2(k)$) and which we rename as $\widetilde f_{m+1}(x_1, x_2, x_3, x_4, y_1, y_2, y_3, y_4)$ for convenience. There exist matrices $A_0 =\left (\begin{array}{cc} a&  b \\ c & d \end{array}\right )$ and $B_0 = \left (\begin{array}{cc} e& u \\ v & w \end{array}\right )$ in $\SL_2(k)$ such that $\Tilde{f}_{m+1}(a, b, c, d, e, u, v, w) = 0$ where $\Tilde{f}_{m+1} = \theta(\Tilde{F}_{m+1}^{\R})$. This implies that there exists $A_0, B_0 \in SL_2(k)$ such that we have a solution $(s_0, u_0, v_0)$ of $\Tilde{\psi}_{m+1}(s, u, v) = 0$ in $k^3$ where $s_0 = tr(A_0), u_0 = tr(A_0B_0)$ and $v_0 = tr(B_0)$. As by assumption $tr(\Bar{A})\notin \{\pm2\}$, thus at $(s_m(s_0, u_0, v_0), v_0)$, the Jacobian corresponding to $\Tilde{\psi}_{m+1}(s_m,v)$ is non-zero by Lemma~\ref{jacobian1}. In other words, $(s_0, u_0, v_0)$ is a simple root of $\Tilde{\psi}_{m+1}(s, u, v) = 0$ in $k^3$ by Lemma~\ref{tr 2}. Therefore, by Lemma~\ref{tr 2}, $(a, b, c, d, e, u, v, w)$ must be a simple root of $\Tilde{f}_{m+1}(x_1, x_2, x_3, x_4, y_1, y_2, y_3, y_4) = 0$ in $k^8$. Then, by using Lemma~\ref{liftingsimple} we get 
$$\Tilde{F}_{m+1}^{\R}(x_1, x_2, x_3, x_4, y_1, y_2, y_3, y_4) = 0$$ 
has a solution in $\R^8$. Therefore, there exists matrices $\widehat{A}_0$ and $\widehat{B}_0$ in $\GL_2(\R)$ (in $\mathfrak{H}$, hence in $\GL_2(\mathcal{O}_2)$, see Section~\ref{trace in GL}) such that $\theta(\widehat{A}_0) = A_0$ and $\theta(\widehat{B}_0) = B_0$ in $\SL_2(k)$ and $\tr(e_{m+1}(\widehat{A}_0, \widehat{B}_0)_{\GL}) = \lambda$. But, at this stage, it is not guaranteed that $\widehat{A}_0$ and $\widehat{B}_0$ are members of $\SL_2(\R)$ or not. Therefore, invoking Lemma~\ref{lem: trace GL to trace SL}, we get two matrices $\widehat{\mathbf{g}}_1, \widehat{\mathbf{g}}_2$ in $\SL_2(\R)$ such that they are the lifts of $A_0$ and $B_0$ respectively and $\tr(e_{m+1}^{\R}(\widehat{\mathbf{g}}_1, \widehat{\mathbf{g}}_2)) = \lambda$.

Hence there exists $\widehat{\mathbf{g}}_1, \widehat{\mathbf{g}}_2$ in $\SL_2(\R)$ such that  $\tr(e_{m+1}^{\R}(\widehat{\mathbf{g}}_1, \widehat{\mathbf{g}}_2)) = \tr(A)$. This implies that $e_{m+1}(\widehat{\mathbf{g}}_1, \widehat{\mathbf{g}}_2)$ and $A$ have the same characteristic polynomial. Moreover,  $\Bar{\lambda} \neq \pm2$ ensures $e_{m+1}(\widehat{\mathbf{g}}_1, \widehat{\mathbf{g}}_2)$ is regular semisimple; see Figure~\ref{fig: lifting com}. 

Now consider the given matrix $M\in \SL_2(\mathcal{O})$.   
Using the description of the canonical isomorphism $\mathcal{T}^*$ in Lemma~\ref{lift regular} we have matrices $\mathbf{g}_1^{(2)}, \mathbf{g}_2^{(2)}\in \SL_2(\mathcal{O}_2)$ such that they are the corresponding lifts of $A_0$ and $B_0$ respectively and they satisfies $\tr(e_{m+1}^{\R}(\mathbf{g}_1^{(2)}, \mathbf{g}_2^{(2)})) = \tr(M_2)$. 

\textbf{Step II :} Take $\mathcal{A} = \mathcal{O}_3$ in the description of $\mathfrak{H}$ and consider the equation $\tr(e_{m+1}(X, Y)_{\mathfrak{H}}) - \tr(M_3) = 0$ where $X,Y$ are matrix variables in $\mathfrak{H}$; see Section~\ref{trace in GL}. Again this means we are considering the equation $$\widetilde F_{m+1}^{\mathcal{O}_3}(x_1, x_2, x_3, x_4, y_1, y_2, y_3, y_4) = 0$$ 
over $\mathcal{O}_3$. Under reduction mod $\theta_2$, the equation boils down to the equation $\widetilde F_{m+1}^{\R}(x_1, x_2, x_3, x_4, y_1, y_2, y_3, y_4) = 0$ over $\mathcal{O}_2$. We have seen earlier that this equation has a solution $\alpha_0=(a_1, a_2, a_3, a_4, b_1, b_2, b_3, b_4)$ in $\mathcal{O}^8_2$ where $\widehat{g}_1=\left(\begin{array}{cc} a_1 & a_2 \\  a_3 & a_4 \end{array}\right),\ \widehat{g}_2=\left(\begin{array}{cc} b_1 & b_2 \\ b_3 & b_4
\end{array}\right)$ in $\SL_2(\R)$. Moreover, as $\alpha_0$ is a lift of a simple root therefore at least one one of the partial derivatives $\{\partial \widetilde F_{m+1}^{\R}/\partial x_{i},\ \partial \widetilde F_{m+1}^{\R}/\partial y_{j}\ \mid \ i,j=1,2,3,4\}$ must take value from $\R^{\times}$ when evaluated at $\alpha_0$. Since $(ker(\theta_2))^2=0$, hence by the similar argument used for lifting simple roots from $k$ to $\mathcal{O}_2$ level and using the Remark~\ref{rmk lifting} (lifting simple roots from $\mathcal{O}_2$ to $\mathcal{O}_3$ level), we obtain, the particular solution of  $\tr(e_{m+1}^{\R}(X, Y)) = \tr(M_2)$. We have obtained this over $\mathcal{O}_2$ via the lifting of simple roots, which can also be lifted to a solution of  $\tr(e_{m+1}^{\mathcal{O}_3}(X, Y))=\tr(M_3)$ over $\mathcal{O}_3$. Therefore, using Lemma~\ref{lem: trace GL to trace SL} and the argument similar to that in the length two case in the previous step, we obtain that there exist matrices $\mathbf{g}_1^{(3)}, \mathbf{g}_2^{(3)}\in \SL_2(\mathcal{O}_3)$ such that they are the corresponding lifts of $\mathbf{g}_1^{(2)}$ and $\mathbf{g}_2^{(2)}$ respectively and they satisfies $\tr(e_{m+1}^{\mathcal{O}_3}(\mathbf{g}_1^{(3)}, \mathbf{g}_2^{(3)})) = \tr(M_3)$. Thus, by repeating use of this argument, we obtain, for any $\ell\geq 2$, there exist $\mathbf{g}_1^{(r+1)}, \mathbf{g}_2^{(r+1)}\in \SL_2(\mathcal{O}_{r+1})$ such that they are the corresponding lifts of $\mathbf{g}_1^{(r)}$ and $\mathbf{g}_2^{(r)}$ respectively (satisfying $\tr(e_{m+1}^{\mathcal{O}_r}(\mathbf{g}_1^{(r)}, \mathbf{g}_2^{(r)}))=\tr(M_{r})$) and it satisfies $\tr(e_{m+1}^{\mathcal{O}_{r+1}}(\mathbf{g}_1^{(r+1)}, \mathbf{g}_2^{(r+1)}))=\tr(M_{r+1})$ for $r=1,2,\ldots,\ell-1$ where $\mathbf{g}_1^{(1)}=A_0,\ \mathbf{g}_2^{(1)}=B_0,\  M_1=\overline{M}$. 

\textbf{Step III :} Now using the isomorphism $\mathcal{O} \rightarrow \varprojlim\limits_{j\geq 1} \mathcal{O}/ \mathfrak{m}^j\mathcal{O}$ we obtain $(\tr(\overline{M}), \tr(M_2), \tr(M_3),\ldots)$ is the isomorphic image of $\tr(M)$. Through the canonical isomorphism $\mathcal{T^*} \colon \M_n(\mathcal{O}) \rightarrow \varprojlim\limits _{j\geq 1}\M_n(\mathcal{O}_j)$ let $\mathbf{g}_1$ and $\mathbf{g}_2$ map to $(A_0, \mathbf{g}_1^{(2)}, \mathbf{g}_1^{(3)},\ldots)$ and $(B_0, \mathbf{g}_2^{(2)}, \mathbf{g}_2^{(3)},\ldots)$, respectively. Note that $\det(\mathbf{g}_1) = 1 =\det(\mathbf{g}_2)$. Therefore, $\mathbf{g}_1, \mathbf{g}_2 \in \SL_2(\mathcal{O})$. Through the isomorphism $\mathcal{O} \rightarrow \varprojlim\limits_{j\geq 1} \mathcal{O}/ \mathfrak{m}^j\mathcal{O}$ we get the isomorphic image of  $\tr(e_{m+1}^{\mathcal{O}}(\mathbf{g}_1, \mathbf{g}_2))$ is exactly equal to the isomorphic image of $\tr(M)$ which is nothing but $(\tr(\overline{M}), \tr(M_1), \tr(M_2),\ldots)$. Therefore, $\tr(M) = \tr(e_{m+1}^{\mathcal{O}}(\mathbf{g}_1, \mathbf{g}_2))$. Consequently, $M$ and $e_{m+1}^{\mathcal{O}}(\mathbf{g}_1, \mathbf{g}_2)$ have the same characteristic polynomial over $\mathcal{O}$. Note that $e_{m+1}^{\mathcal{O}}(\mathbf{g}_1, \mathbf{g}_2)$ and $M$, both are regular semisimple, hence cyclic. Invoking Lemma~\ref{conjugacycyclic} we get $M\sim _{\mathcal{O}}e_{m+1}^{\mathcal{O}}(\mathbf{g}_1, \mathbf{g}_2)$. This proves the required result.
\end{proof}

\subsection{Lifts of cyclic elements which are not regular semisimples}{\label{lift regular}}

Next, we look into the lifting result regarding the $m+1$-th Engel word for cyclic elements of $\SL_2(k)$ with trace $\pm 2$.

\begin{lemma}{\label{lifts of unipotents}}
Let $\mathcal{O}$ be a complete local principal ideal ring, and its residue field $k$, $|k|\geq 5$, is of characteristic $\neq 2$. Then, every lift of $\mathbf{\bar g} = \left(\begin{array}{cc} 1 &  0\\ 1 & 1
\end{array}\right)$ in $\SL_2(\mathcal{O})$ is in $\im (e_{m+1}^{\mathcal{O}})$, for all $m\geq 1$.
\end{lemma}
\begin{proof}
It follows from Lemma~\ref{triang unipotent}, that for $|k|\geq 5$, there exists $\Tilde{Z}_1, \Tilde{Z}_2\in \SL_2(k)$ such that $\mathbf{\bar g} = \bar e_{m+1}(\Tilde{Z}_1, \Tilde{Z}_2)$ for any $m\geq 1$. Moreover, it is clear from the proof that we can take $\Tilde{Z}_i$ to be cyclic. Now $\bar e_{i+1}(\Tilde{Z}_1, \Tilde{Z}_2) = \bar e_1(\bar e_i(\Tilde{Z}_1, \Tilde{Z}_2), \Tilde{Z}_2)$. Further, for for all $i$, $\tr(\bar e_i(\Tilde{Z}_1, \Tilde{Z}_2))=\tr (\bar e_i(Z_1, Z_2))$ as $\bar e_i(\Tilde{Z}_1, \Tilde{Z}_2) \sim_k \bar e_i(Z_1, Z_2)$; see Lemma~\ref{triang unipotent}. Now from the proof of Lemma~\ref{uni2} it is clear that $\tr(\bar e_i(Z_1, Z_2)) = 2$ for each $i=1, \ldots, m$.  Note that, as $\mathbf{\bar g}$ is non-scalar, therefore $\Tilde{Z}_1, \Tilde{Z}_2, \bar e_i(\Tilde{Z}_1,\Tilde{Z}_2)$ must be non-scalar and cyclic for all $i$. Therefore, $H_0= \langle \Tilde{Z}_1, \Tilde{Z}_2\rangle$ and $H_i=\langle \bar e_i(\Tilde{Z}_1, \Tilde{Z}_2), \Tilde{Z}_2\rangle$ have no non-trivial fixed vector in $\mathfrak{sl}_2^*(k)$ for all $i$ by Lemma~\ref{fixed vector 2}. Moreover, Lemma~\ref{fixed vector ortho compl} implies $Z(\bar e_{m-1}(\Tilde{Z}_1, \Tilde{Z}_2)\Tilde{Z}_2 \bar e_{m-1}( \Tilde{Z}_1, \Tilde{Z}_2)^{-1}) \cap Z(\Tilde{Z}_2)^{\perp}$, $\ldots$ , 
$Z(\bar e_{1}(\Tilde{Z}_1, \Tilde{Z}_2)\Tilde{Z}_2 \bar e_{1}( \Tilde{Z}_1, \Tilde{Z}_2)^{-1}) \cap Z(\Tilde{Z}_2)^{\perp}$,
$Z(\Tilde{Z}_1\Tilde{Z}_2\Tilde{Z}_1^{-1})\cap Z(\Tilde{Z}_2)^{\perp}$ are trivial in $\mathfrak{sl}_2(k)$. 
Now, invoking Proposition~\ref{adelic 2}, we get the required result.
\end{proof}

\begin{lemma}{\label{lifts of minus unipotents}}
Let $\mathcal{O}$ be a complete local principal ideal ring with its residue field $k\cong \mathbb{F}_q$ of characteristic $\neq 2$. Let $\mathbf{\bar g}=\left(\begin{array}{cc} -1 &  0 \\ 1 & -1 \end{array}\right)\in \SL_2(k)$ and $m\geq 1$. Suppose $\mathbf{\bar g} \in \im(\bar e_{m+1})$ in $\SL_2(k)$. Then, there exists an integer $q_0(m)$ such that for all $q\geq q_0(m)$ (can be taken as $q_0(m)=2\cdot3^{2^{m+2}}$), any lift of $\mathbf{\bar g}$ in $\SL_2(\mathcal{O})$ is in $\im (e_{m+1}^{\mathcal{O}})$.
\end{lemma}
\begin{proof}
As $\mathbf{\bar g}\in \im(\bar e_{m+1}) $ in $\SL_2(k)$ therefore there exists $h_1, h_2 \in \SL_2(k)$ such that $[\bar e_m(h_1, h_2), h_2] = \mathbf{\bar g}$. Moreover, $\psi_{m+1}(s_m, v) = \tr(\mathbf{\bar g}) = s_m^2-s_mv^2 + 2v^2-2$ where $s_m=\tr(\bar e_m(h_1, h_2))$ and $v=\tr(h_2)$ (see Section~\ref{tr}). Note that $s_m$ can not be $2$ as $\tr(\mathbf{\bar g}) = -2$. Depending upon the properties of $h_2$, we make the following cases.

\textbf{Case I\  (If $h_2$ is not regular semisimple):}  
It is evident that $\bar e_i(h_1, h_2)$ and $h_2$ must be non-scalar for $i=1, 2, \ldots, m$. Moreover, Lemma~\ref{fixed vector ortho compl} shows that 
$$Z(\bar e_{m-1}(h_1, h_2)h_2 \bar e_{m-1}(h_1, h_2))\cap Z(h_2)^{\perp},\ldots, Z(\bar e_{1}(h_1,h_2)h_2 \bar e_{1}(h_1,h_2)) \cap Z(\mathbf{\bar g}_2)^{\perp}, Z(h_1 h_2 h_1^{-1}) \cap Z(h_2)^{\perp}$$ 
are trivial in $\mathfrak{sl}_2(k)$. Invoking Proposition~\ref{adelic 2}, we obtain the result in this case. 
 
\textbf{Case II\  (If $h_2$ is regular semisimple):} The ideas used in the proof of Proposition~\ref{regsem1} will be useful here. In this case, we need to use the lifting strategy we have used in the regular semisimple case. Let $A$ be any lift of $\mathbf{\bar g}$ in $\SL_2(\mathcal{O})$. Then the image of $A$ under the reduction map, which takes it to the length $2$ level, is, say, $A_2\in \SL_2(\mathcal{O}_2)$. Let $\Upsilon = \{(s, v)\in k^2 \mid \text{either}\  v=0\ \text{or}\  v^2=4\}$.
Then, we can choose $h_1$ and $h_2$ with $(\tr(h_1), \tr(h_2)) \notin \Upsilon$ such that $\bar e_{m+1}(h_1, h_2) = \mathbf{\bar g}$; see \cite[Corollary 3.9 and Theorem 5.1]{BandmanGarionGrunewald2012}. This also brings in the necessity of the assumption $q\geq q_0(m) = 2\cdot3^{2^{m+2}}$. (Indeed, in the proof of \cite[theorem 5.1]{BandmanGarionGrunewald2012} one can observe $q\geq  max\{2d_n^4,2j_n^4\}$ where $d_n,j_n\leq 3^{2^n}$.)
     
Now $[\frac{\partial \Tilde{\psi}_{m+1}}{\partial s_m}\ \frac{\partial \Tilde{\psi}_{m+1}}{\partial v}]=\mathbf{0}$ gives
\begin{equation}
2s_m-v^2=0, \ \ \ \ 2v(2-s_m)=0.
\end{equation}
In this case, as $s_m$ cannot take the value $2$, the only possibility is $v=0$. As for regular semisimple $h_2$ that we have chosen, must have $\tr(h_2)\neq 0$ hence the roots of $\tilde{\psi}_{m+1}(s_m, v) = 0$ must be simple. Therefore, using Lemma~\ref{tr 2} and Hensel's lemma we obtain $\widetilde F_{m+1}^{\R}(x_1, x_2, x_3, x_4, y_1, y_2, y_3, y_4) = 0$ (see Section~\ref{trace in GL}, by taking $\lambda = \tr(A_2)$) has a solution in $\mathcal{O}_2^8$. We may consider this for $X = \left (\begin{array}{cc} x_1& x_2 \\ x_3 & x_4 \end{array}\right )$ and $Y = \left (\begin{array}{cc} y_1 & y_2  \\ y_3 & y_4 \end{array}\right )$, matrix variables in $\mathfrak{H}$ (see, Section~\ref{trace in GL}), each with $4$ indeterminate entries from $\R$. Now, using the same process described in the proof of Proposition~\ref{regsem1}, we obtain that there exists $h_1^{(2)}, h_2^{(2)}\in \SL_2(\mathcal{O}_2)$ which are the corresponding lifts of $h_1$ and $h_2$ respectively, satisfies $\tr(e_{m+1}^{\R}(h_1^{(2)}, h_2^{(2)})) = \tr(A_2)$. As $(ker (\theta_2))^2 = 0$ therefore by using the same technique we achieve $\tr(e_{m+1}^{\mathcal{O}_3}(h_1^{(3)}, h_2^{(3)})) = \tr(A_3)$ where  $h_1^{(3)}, h_2^{(3)} \in \SL_2(\mathcal{O}_3)$ which are the corresponding lifts of $h_1^{(2)}$ and $h_2^{(2)}$ respectively. By repeating this argument, we obtain, for any $\ell \geq 2$, there exist $h_1^{(r+1)}, h_2^{(r+1)}\in \SL_2(\mathcal{O}_{r+1})$ such that they are the corresponding lifts of $h_1^{(r)}$ and $h_2^{(r)}$ respectively (satisfying $\tr(e_{m+1}^{\mathcal{O}_r}(h_1^{(r)}, h_2^{(r)})) = \tr(A_{r})$) and these satisfy $\tr(e_{m+1}^{\mathcal{O}_{r+1}}(h_1^{(r+1)}, h_2^{(r+1)})) = \tr(A_{r+1})$ for $r=1, 2, \ldots, \ell-1$ where $h_1^{(1)} = h_1, h_2^{(1)} = h_2, A_1 = \mathbf{\bar g}$. 

Using the isomorphism $\mathcal{O} \rightarrow \varprojlim\limits_{j\geq 1} \mathcal{O}/\mathfrak{m}^j\mathcal{O}$, the isomorphic image of $\tr(A)$ is $(\tr(\mathbf{\bar g}), \tr(A_2), \tr(A_3), \ldots)$. Now, using the canonical isomorphism $\mathcal{T^*} \colon  \M_n(\mathcal{O}) \rightarrow \varprojlim\limits _{j\geq 1}\M_n(\mathcal{O}_j)$, let  $(h_1, h_1^{(2)}, h_1^{(3)}, \ldots)$ and $(h_2, h_2^{(2)}, h_2^{(3)}, \ldots)$ be the images of ${h}^{\circ}_1$ and $h^{\circ}_2$ respectively. Note that $\det(h^{\circ}_1) = 1 =\det(h^{\circ}_2)$. Therefore $h^{\circ}_1, h^{\circ}_2\in \SL_2(\mathcal{O})$. Via the isomorphism $\mathcal{O} \rightarrow \varprojlim\limits_{j\geq 1} \mathcal{O}/ \mathfrak{m}^j\mathcal{O}$ we get the isomorphic image of  $\tr(e_{m+1}^{\mathcal{O}}(h^{\circ}_1, h^{\circ}_2))$ is exactly equal to the isomorphic image of $\tr(A)$ which is nothing but $(\tr(\mathbf{\bar g}), \tr(A_2), \tr(A_3), \ldots)$. Therefore $\tr(A) = \tr(e_{m+1}^{\mathcal{O}}(h^{\circ}_1, h^{\circ}_2))$. Therefore, $A$ and $e_{m+1}^{\mathcal{O}}(h^{\circ}_1, h^{\circ}_2)$ have the same characteristic polynomial over $\mathcal{O}$. Note that $e_{m+1}^{\mathcal{O}}(h^{\circ}_1, h^{\circ}_2)$ and $A$ both are cyclic by definition. 
Invoking Lemma~\ref{conjugacycyclic} we get
 $A\sim _{\mathcal{O}}e_{m+1}^{\mathcal{O}}(h^{\circ}_1, h^{\circ}_2)$. This proves the required result.
\end{proof}

\section{Engel maps on $\SL_2(\mathcal O)$ and $\text{PSL}_2(\mathcal O_2)$: Main results}{\label{mainth}}

In this section, we prove our main results.
\begin{theorem}{\label{th: Engel second}}
Let $\mathcal{O}$ be a local principal ideal ring, complete with respect to its unique maximal ideal $\mathfrak{m} = (\pi)$ and its residue field $k\cong \mathbb{F}_q$ of characteristic $\neq 2$. Let $e_{m+1}(x, y)$ be the $(m+1)$-th Engel word in the free group $\mathcal F_2=\langle x, y\rangle$ for $m\geq 1$. Consider the induced word maps $e_{m+1}^{\mathcal{O}}$ on $\SL_2(\mathcal{O})$ and $\bar e_{m+1}$ on $\SL_2(k)$.
Let $A\in \SL_2(k)$ be a non-scalar element. Then, there exists a constant $q_0(m)$ such that for all $q \geq q_0(m)$, all lifts of $A$ in $\SL_2(\mathcal{O})$ are in the image $\im (e_{m+1}^{\mathcal{O}})$. The constant $q_0(m)$ can be taken as $2\cdot3^{2^{m+2}}$.
\end{theorem}
\begin{proof}
It follows from~\cite[Corollary 3.9 and Corollary 5.4]{BandmanGarionGrunewald2012} that there exists $q_0(m)$ such that for all $q\geq q_0(m)$, the element $A \in \im(\bar e_{m+1})$ in $\SL_2(k) $ for all $m\geq 1$. Thus, there exist $B_1, B_2\in \SL_2(k)$ such that $A = \bar e_{m+1}(B_1, B_2)$. Note that this property remains invariant under conjugation, i.e, for $P\in \GL_2(k)$ we have $P^{-1}AP = \bar e_{m+1}(P^{-1}B_1P, P^{-1}B_2P)$ where both of $P^{-1}B_iP\in \SL_2(k)$. Therefore, to study what elements are in the image for the lifts of non-scalar elements of $\SL_2(k)$, it is enough to study them for the representatives of the $\GL$-conjugates. All the non scalar matrices in $\SL_2(k)$ are either regular semisimple or $\left(\begin{array}{cc}
1 & 0 \\ 1 & 1 \end{array}\right)$ or $\left(\begin{array}{cc}
-1 & 0 \\ 1 & -1 \end{array}\right)$ up to $\GL_2(k)$ conjugate. The question that lifts of these elements are in the image of the Engel word map $e_{m+1}^{\mathcal{O}}$ is proved in Proposition~\ref{regsem1}, Lemma~\ref{lifts of unipotents}, and Lemma~\ref{lifts of minus unipotents}. This proves the result. 
\end{proof}

\begin{proposition}{\label{prop: lift of I Engel}}
Let $\mathcal{O}_2$ be a local principal ideal ring of length $2$ with its unique maximal ideal $\mathfrak{m}=(\pi_2)$. Suppose its residue field $k\cong \mathbb{F}_q$ is of characteristic $\neq 2$ and $|k| \geq 7$. Then, for any $m\geq 1$, any lift of $I$ is in $\im(e_{m+1}^{\mathcal{O}_2})$. 
\end{proposition}
\begin{proof}
From Table~\ref{tab:simplified}, the lifts of $I$ in $\SL_2(\mathcal{O}_2)$ are of the form $I, \ \left(\begin{array}{cc} 1 & 0 \\ \pi_2 & 1 \end{array} \right), \left(\begin{array}{cc} 1 + \pi_2 & 0 \\ 0 & 1 - \pi_2
\end{array} \right), \left(\begin{array}{cc} 1 & \pi_2 \epsilon b \\ \pi_2 b & 1 \end{array}\right)$ up to $\GL_2(\mathcal{O}_2)$-conjugates. We deal with each case below.

We begin with choosing certain obvious lifts of $I$. 
\begin{enumerate}
\item For any  $m\geq 1$, we have $I = \bar e_{m+1}(\bar g_1, \bar g_2)$ where $\bar g_1 = diag(\lambda, \lambda^{-1}) = \bar g_2$ in $\SL_2(k)$ for some $\lambda\neq \lambda^{-1}$. Take any lift of $\lambda$ in $\mathcal{O}_2$ and denote it by $\lambda^{(2)}$. Consider $g_1^{(2)} = diag(\lambda^{(2)},(\lambda^{(2)})^{-1}) = g_2^{(2)}$ in $\SL_2(\mathcal{O}_2)$. These $g_1^{(2)}$ and $g_2^{(2)}$ are the lifts of $\bar g_1$ and $\bar g_2$, respectively. It is easy to see $e_{m+1}^{\mathcal{O}_2}(g_1^{(2)},g_2^{(2)})=I$ in $\SL_2(\mathcal{O}_2)$, for any  $m\geq 1$.
\item Now, we can write $I = e_{m+1}(\bar h_1, \bar h_2)$ in $\SL_2(k)$ for any $m\geq 1$, where $\bar h_1 = \left(\begin{array}{cc} 1 & 1 \\ 0 & 1 \end{array}\right) = \bar h_2$. Consider $h_1^{(2)} = \left(\begin{array}{cc} 1 & 1 \\ 0 & 1 \end{array}\right) = h_2^{(2)}$ in $\SL_2(\mathcal{O}_2)$. These are the lifts of $\bar {h}_{1} $ and $\bar {h}_{2}$, respectively and $e_{m+1}^{\mathcal{O}_2}(h_1^{(2)}, h_2^{(2)}) = I$. 
\end{enumerate}    
    
Now, for $X, Y\in \mathfrak{sl}_2(\mathcal{O}_2)$ we have 
\begin{eqnarray*}
e_{m+1}^{\mathcal{O}_2}(g_1^{(2)}(I + \pi_2X), g_2^{(2)}(I + \pi_2Y)) &=& e_{m+1}^{\mathcal{O}_2}(g_1^{(2)}, g_2^{(2)})\left(I + \pi_2De_{m+1}^{\mathcal{O}_2}(X, Y)\right)= I + \pi_2De_{m+1}^{\mathcal{O}_2}(X, Y)\\ 
e_{m+1}^{\mathcal{O}_2}(h_1^{(2)}(I+\pi_2X),h_2^{(2)}(I+\pi_2Y)) &=& e_{m+1}^{\mathcal{O}_2}(h_1^{(2)},h_2^{(2)}) \left(I+\pi_2De_{m+1}^{\mathcal{O}_2}(X,Y)\right) = I+\pi_2De_{m+1}^{\mathcal{O}_2}(X,Y)
\end{eqnarray*}
for any matrix variables $X, Y \in \mathfrak{sl}_2(\mathcal{O}_2)$. Note that these two expressions for $De_{m+1}^{\mathcal{O}_2}(X, Y)$ are different because their expressions involves $g_1^{(2)}, g_2^{(2)}$ and $h_1^{(2)}, h_2^{(2)}$ respectively. We have set up a framework to apply Proposition~\ref{adelic 2}. Now, let us analyze the following case.

\textbf{Step I:} Claim: The elementary matrices $E_{1,2}$ and $E_{2,1}$ lie in the image of $\overline{De}_{m+1}$, the derivative map corresponding to $\bar e_{m+1}$ at $(\bar g_1,\bar g_2)$.

As $\bar g_2$ and $\bar h_2$ both are cyclic in $\SL_2(k)$ therefore $Z(\bar g_2)$ and $Z(\bar h_2)$ both are one dimensional subspaces of $\mathfrak{sl}_2(k)$. Moreover, $Z(\bar g_2)= span \left(\begin{array}{cc}  1 & 0 \\
0 & -1 \end{array}\right)$ and $Z(\bar h_2) = span \left(\begin{array}{cc}  0 & 1 \\ 0 & 0 \end{array}\right)$. As $dim(Z(\bar g_2)^{\perp}) = 2 = dim (Z(\bar h_2)^{\perp})$ therefore $Z(\bar g_2)^{\perp} = span\left\{ \left(\begin{array}{cc}0 & 1 \\  0 & 0 \end{array}\right),\ \left(\begin{array}{cc} 0 & 0 \\ 1 & 0 \end{array}\right)\right\}$ and $Z(\bar h_2)^{\perp} = span\left\{ \left(\begin{array}{cc} 1 & 0 \\ 0 & -1 \end{array}\right),\ \left(\begin{array}{cc} 0 & 0 \\
1 & 0 \end{array}\right)\right\}$ with respect to the bilinear form $\langle, \rangle$ defined earlier. 
    
It is obvious to see that $\bar e_j(\bar g_1,\bar g_2) = 1 = \bar e_j(\bar h_1, \bar h_2)$ for $1\leq j\leq m$ and $Z(\bar g_2) \bigcap Z(\bar g_2)^{\perp} = \{\mathbf{0}\} = Z(\bar h_2) \bigcap Z(\bar h_2)^{\perp}$. Therefore, 
$$Z({\bar g}_1{\bar g}_2{\bar g}_1^{-1}) \cap Z({\bar g}_2)^{\perp}, Z(\bar e_1({\bar g}_1, {\bar g}_2){\bar g}_2\bar e_1({\bar g}_1, {\bar g}_2)^{-1}) \cap Z({\bar g}_2)^{\perp}, \ldots, Z(\bar e_{m-2}({\bar g}_1,{\bar g}_2){\bar g}_2\bar e_{m-2}({\bar g}_1,{\bar g}_2)^{-1}) \cap Z({\bar g}_2)^{\perp}$$ and $Z(\bar e_{m-1}({\bar g}_1,{\bar g}_2){\bar g}_2\bar e_{m-1}({\bar g}_1, {\bar g}_2)^{-1}) \cap Z({\bar g}_2)^{\perp}$  are trivial in $\mathfrak{sl}_2(k)$. 
Similarly, we can show that 
$$Z({\bar h}_1{\bar h}_2{\bar h}_1^{-1})\cap Z({\bar h}_2)^{\perp},  Z(\bar e_1({\bar h}_1,{\bar h}_2){\bar h}_2\bar e_1({\bar h}_1,{\bar h}_2)^{-1})\cap Z({\bar h}_2)^{\perp}, \ldots, Z(\bar e_{m-2}({\bar h}_1,{\bar h}_2){\bar h}_2\bar e_{m-2}({\bar h}_1,{\bar h}_2)^{-1})\cap Z({\bar h}_2)^{\perp}$$ and the subset $Z(\bar e_{m-1}({\bar h}_1,{\bar h}_2){\bar h}_2\bar e_{m-1}({\bar h}_1,{\bar h}_2)^{-1})\cap Z({\bar h}_2)^{\perp}$  are trivial in $\mathfrak{sl}_2(k)$.

Let $T_1\in \im (\overline{De}_{m+1})^{\perp}$ at $(\bar g_1,\bar g_2)$. The condition $Z(\bar e_{m-1}({\bar g}_1, {\bar g}_2){\bar g}_2\bar e_{m-1}({\bar g}_1, {\bar g}_2)^{-1}) \cap Z({\bar g}_2)^{\perp} = \{\mathbf{0}\}$ gives $U = (I - \mathrm{Ad}_{\mathbf{\bar g}_2^{-1}})(V) = \mathbf{0}$ as in Proposition~\ref{adelic 2} (also in Lemma~\ref{lem-adelic 2}). Therefore, $V = \mathrm{Ad}_{\bar e_m^{-1}\mathbf{\bar g}_2^{-1}}(T_1)\in Z(\bar g_2)\bigcap Z(\bar e_m)$ (by Equation~(\ref{eqmain2}) in the proof of Proposition~\ref{adelic 2}). This further implies $T_1\in Z(\bar g_2)$ as $\bar e_m=\bar e_m(\bar g_1,\bar g_2) = I$. Hence $\im (\overline{De}_{m+1})^{\perp} \subset Z(\bar g_2)$. Moreover, from the recurrence formula $\overline{De}_{m+1}\colon  \mathfrak{sl}_2(k) \times \mathfrak{sl}_2(k) \rightarrow \mathfrak{sl}_2(k)$ defined by $$(X,Y)\mapsto \left(((\overline{De}_m(X,Y)^{\bar{g}_2} -\overline{De}_m(X,Y))^{\bar e_m^{-1}}+(Y^{\bar e_m^{-1}}-Y)\right)^{\bar{g}_2^{-1}}$$ 
and from the conditions $\bar e_m({\bar g}_1, {\bar g}_2)=I,\ {\bar g}_1={\bar g}_2$, it is clear that $Z(\bar g_2)\subset \im (\overline{De}_{m+1})^{\perp}$ at $(\bar g_1, \bar g_2)$. Therefore, $\im (\overline{De}_{m+1})^{\perp} = Z(\bar g_2)$ for the choice $(\bar g_1, \bar g_2)$. Now, $dim (Z(\bar g_2))+dim (Z(\bar g_2)^{\perp})=dim (\mathfrak{sl}_2(k))$ and $Z(\bar g_2)\bigcap Z(\bar g_2)^{\perp} = \{\mathbf{0}\}$ together implies $\mathfrak{sl}_2(k) = Z(\bar g_2)\oplus Z(\bar g_2)^{\perp}$. As $dim (\im (\overline{De}_{m+1})) = 2$ therefore  $\im (\overline{De}_{m+1})=Z(\bar g_2)^{\perp}$ in this case. Hence, the elementary matrices $E_{1,2}$ and $E_{2,1}$ lie in the image of $\overline{De}_{m+1}$, the derivative map corresponding to $\bar e_{m+1}$ at $(\bar g_1,\bar g_2)$. This proves the claim. 
    
\textbf{Step II:} Now, we exactly determine the $\im(De_{m+1}^{\mathcal{O}_2})$. From the previous claim, there exists $\alpha_1, \alpha_2, \beta_1, \beta_2 \in \mathfrak{sl}_2(k)$ such that $\overline {De}_{m+1}(\alpha_1, \alpha_2) = E_{1,2}$ and $\overline{De}_{m+1}(\beta_1, \beta_2) = E_{2,1}$. Now consider any two lifts $\alpha_1^{(2)}$ of $\alpha_1$ and $\alpha_2^{(2)}$ of $\alpha_2$ in $\mathfrak{sl}_2(\mathcal{O}_2)$. Similarly, consider lifts $\beta_1^{(2)}$ of $\beta_1$ and $\beta_2^{(2)}$ of $\beta_2$ in $\mathfrak{sl}_2(\mathcal{O}_2)$. Note that this kind of lifts always exist because, using the surjective map $\M_2(\mathcal{O}_2) \rightarrow \M_2(k)$, we can see that any lift of a traceless matrix achieves non-unit trace in $\mathcal{O}_2$. Therefore, $De_{m+1}^{\mathcal{O}_2}(\alpha_1^{(2)}, \alpha_2^{(2)}) = E^{\mathcal{O}_2}_{1,2} + \pi_2C_1$ and $De_{m+1}^{\mathcal{O}_2}(\beta_1^{(2)}, \beta_2^{(2)}) = E^{\mathcal{O}_2}_{2,1} + \pi_2 C_2$ for some $C_1, C_2 \in \M_2(\mathcal{O}_2)$ (here $E^{\mathcal{O}_2}_{i,j}$ is the elementary matrix in $\M_2(\mathcal{O}_2)$ of which $(i,j)$-th entry is $1$ and zero elsewhere). This implies $\pi_2De_{m+1}^{\mathcal{O}_2}(\alpha_1^{(2)}, \alpha_2^{(2)}) = \pi_2E^{\mathcal{O}_2}_{1, 2}$ and $\pi_2 De_{m+1}^{\mathcal{O}_2}(\beta_1^{(2)}, \beta_2^{(2)}) = \pi_2 E^{\mathcal{O}_2}_{1,2}$. Now, $De_{m+1}^{\mathcal{O}_2}$ is a $\mathcal{O}_2$ linear map and $\mathfrak{sl}_2(\mathcal{O}_2)$ is a finitely generated $\mathcal{O}_2$-module. As $\mathcal{O}_2$ is Noetherian,  the $\mathcal{O}_2$-submodule $\im(De_{m+1}^{\mathcal{O}_2})$ of $\mathfrak{sl}_2(\mathcal{O}_2)$ is finitely generated. Moreover, $E^{\mathcal{O}_2}_{1,2} + \pi_2C_1,\  E^{\mathcal{O}_2}_{2,1} + \pi_2 C_2 \in \im(De_{m+1}^{\mathcal{O}_2})$, therefore $$\im(De_{m+1}^{\mathcal{O}_2}) = span_{\mathcal{O}_2}\{ E^{\mathcal{O}_2}_{1,2} + \pi_2C_1,\ E^{\mathcal{O}_2}_{2,1} + \pi_2C_2\}$$ 
by Nakayama's lemma (\cite[proposition 2.8]{Atiyahmacdonald}). 
    
\textbf{Step III:} Now we are ready to prove the result. The identity $I$ is always in the $\im(e_{m+1}^{\mathcal{O}_2})$. Now, we have $|k|\geq 7$,  therefore, for $A_2=\left(\begin{array}{cc} 1 & 0 \\ \pi_2 & 1 \end{array}\right)$ we use Lemma~\ref{triang unipotent} to get 
$A_2\in \im(e_{m+1}^{\mathcal{O}_2})$.

Next, let us consider the element $A_2 = \left(\begin{array}{cc} 1 & \pi_2\epsilon b \\ \pi_2 b & 1 \end{array}\right)$, where $\epsilon, b \in \mathcal{O}_2^{\times}$ and $\epsilon$ is non-square. Then
\begin{eqnarray*}
A_2 &=& I+\epsilon b\pi_2E_{1,2}^{\mathcal{O}_2}+ b\pi_2E_{2,1}^{\mathcal{O}_2} \\ &=& I+\pi_2\left(\epsilon b(E_{1,2}^{\mathcal{O}_2}+\pi_2C_1)+b(E_{2,1}^{\mathcal{O}_2}+\pi_2C_2)\right); \ \text{as}\ \pi_2^2=0.
\end{eqnarray*}
As $\epsilon b(E_{1,2}^{\mathcal{O}_2} + \pi_2C_1) + b(E_{2,1}^{\mathcal{O}_2} + \pi_2C_2) \in span_{\mathcal{O}_2} \left\{E^{\mathcal{O}_2}_{1,2} + \pi_2 C_1,\ E^{\mathcal{O}_2}_{2,1} + \pi_2C_2\}\right\}$ therefore there exist $P_{\mathcal{O}_2}, P_{\mathcal{O}_2}' \in \mathfrak{sl}_2(\mathcal{O}_2)$ such that $\epsilon b(E_{1,2}^{\mathcal{O}_2} + \pi_2C_1)+b(E_{2,1}^{\mathcal{O}_2}+\pi_2C_2) = De_{m+1}^{\mathcal{O}_2}(P_{\mathcal{O}_2}, P_{\mathcal{O}_2}')$. Hence, 
\begin{eqnarray*}
e_{m+1}^{\mathcal{O}_2}(g_1^{(2)}(I + \pi_2P_{\mathcal{O}_2}), g_2^{(2)}(I + \pi_2P_{\mathcal{O}_2}'))  = e_{m+1}^{\mathcal{O}_2}(g_1^{(2)},g_2^{(2)})\left(I+\pi_2De_{m+1}^{\mathcal{O}_2}(P_{\mathcal{O}_2}, P_{\mathcal{O}_2}')\right) = A_2.
\end{eqnarray*}
Therefore, $A_2\in \im(e_{m+1}^{\mathcal{O}_2})$.

Similarly, for the derivative map corresponding to $e_{m+1}$ at $(h_1^{(2)}, h_2^{(2)})$, we obtain $\im (\overline{De}_{m+1})^{\perp} = Z(\bar h_2)$, consequently $\im (\overline{De}_{m+1}) = Z(\bar h_2)^{\perp}$. This further implies $\im(De_{m+1}^{\mathcal{O}_2}) = span_{\mathcal{O}_2} \{E^{\mathcal{O}_2}_{1,1} -E_{2,2}^{\mathcal{O}_2} + \pi_2H_1, \ E^{\mathcal{O}_2}_{2,1} + \pi_2H_2\}$ for some $H_1, H_2\in \M_2(\mathcal{O}_2)$. The argument is similar to that we have used for derivative map corresponding to $e_{m+1}$ at $(g_1^{(2)}, g_2^{(2)})$ in this case also, and we obtain that there exist $Q_{\mathcal{O}_2}, Q_{\mathcal{O}_2}', U_{\mathcal{O}_2}, U_{\mathcal{O}_2}' \in \mathfrak{sl}_2(\mathcal{O}_2)$ such that $E_{1,1}^{\mathcal{O}_2}-E_{2,2}^{\mathcal{O}_2} + \pi_2H_1 = De_{m+1}^{\mathcal{O}_2}(Q_{\mathcal{O}_2}, Q_{\mathcal{O}_2}')$ and  $E_{2,1}^{\mathcal{O}_2} + \pi_2H_2 = De_{m+1}^{\mathcal{O}_2}(U_{\mathcal{O}_2}, U_{\mathcal{O}_2}')$. Take, $A_2=\left(\begin{array}{cc} 1 + \pi_2 & 0 \\ 0 & 1-\pi_2 \end{array}\right)$. One can write it as $I + \pi_2(E_{1,1}^{\mathcal{O}_2} -E_{2,2}^{\mathcal{O}_2})$ which can further be written as $I + \pi_2(E_{1,1}^{\mathcal{O}_2} -E_{2,2}^{\mathcal{O}_2} + \pi_2H_1)$. Therefore, $e_{m+1}^{\mathcal{O}_2}(h_1^{(2)}(I + \pi_2Q_{\mathcal{O}_2}), h_2^{(2)}(I + \pi_2Q_{\mathcal{O}_2}')) = e_{m+1}^{\mathcal{O}_2}(h_1^{(2)}, h_2^{(2)})\left(I + \pi_2De_{m+1}^{\mathcal{O}_2}(Q_{\mathcal{O}_2}, Q_{\mathcal{O}_2}')\right) = A_2$. Therefore $A_2\in \im(e_{m+1}^{\mathcal{O}_2})$ in this case also. Hence for any $m\geq 1$, all the lifts of $I$ in $\SL_2(\mathcal{O}_2)$, must lie inside $\im(e_{m+1}^{\mathcal{O}_2})$.
\end{proof}

\subsection{Projective Special Linear group over Local Rings }

Let $\mathcal{A}$ be a local (commutative) ring with unity and $\mathcal{M}$ be an $\mathcal{A}$ module. The general linear group $\GL(\mathcal{M})$ is the group of invertible $\mathcal{A}$-linear transformations on $\mathcal{M}$. For any $a\in \mathcal{A}^{\times}$ consider the linear transformation $aI_{\mathcal{M}}\in \GL(\mathcal{M})$ defined by $x\mapsto xr$ for all $x\in \mathcal{M}$. Let $\mathrm{RL}(\mathcal{M}) = \{aI_{\mathcal{M}} \mid a\in \mathcal{A}^{\times}\}$. Note that $\mathrm{RL}(\mathcal{M})$ forms a normal subgroup of $\GL(\mathcal{M})$. The natural surjection $\mathrm{P}\colon \GL(\mathcal{M}) \rightarrow \GL(\mathcal{M})/\mathrm{RL}(\mathcal{M})$ is called the projection map of $\GL(\mathcal{M})$ and its image is called the \emph{projective general linear group of $\mathcal{M}$} and is denoted by $\mathrm{PGL}(\mathcal{M})$. For any subgroup $G$ of $\GL(\mathcal{M})$, $PG\cong G/(G\cap \mathrm{RL}(\mathcal{M}))$. Now assume that $\mathcal{M}$ is a free $\mathcal{A}$-module of finite rank $n$. Then there is an isomorphism (through the choice of a basis) between $\GL(\mathcal{M})$ and $\GL_n(\mathcal{A})$ which takes $\mathrm{RL}(\mathcal{M})$ onto $\mathrm{RL}_n(\mathcal{A}) = \{aI_n \mid a\in \mathcal{A}^{\times}\}$; see \cite[section 1.2]{HahnOMeara}. The matrix analogue of the projection map is $\mathrm{P}\colon \GL_n(\mathcal{A}) \rightarrow \GL_n(\mathcal{A})/ \mathrm{RL}_n(\mathcal{A})$. The image of $\SL_n(\mathcal{A})$ under this map is called the \emph{ projective special linear group over $\mathcal{A}$} and is denoted by $\mathrm{PSL}_n(\mathcal{A}) \cong \SL_n(\mathcal{A})/(\SL_n(\mathcal{A})\cap \mathrm{RL}_n(\mathcal{A}))$. 

\begin{lemma}{\label{psl}}
Let $\mathcal{O}$ be a local principal ideal ring, complete with respect to its unique maximal ideal $\mathfrak{m}=(\pi)$ and has residue field $k$ with characteristic $\neq 2$. Then,  $\mathrm{PSL}_2(\mathcal{O}) \cong \SL_2(\mathcal{O})/\{\pm I_2\}$.
\end{lemma}
\begin{proof}
It is enough to show that $\SL_n(\mathcal{O}) \cap \mathrm{RL}_n(\mathcal{O}) = \{\pm I_2\}$ in this case. Let $\lambda I_2 \in \SL_n(\mathcal{O}) \cap \mathrm{RL}_n(\mathcal{O})$. Therefore, $\bar \lambda$ is a solution of $x^2-1 = 0$ in $k$. Hence $\bar \lambda = \pm 1$. Denote $\mathcal{O}/ \mathfrak{m}^j\mathcal{O} =\mathcal{O}_j$ for $j\geq 1$. It is easy to see (check the calculations in Family (1) case in Section~\ref{O_2conj}) that the solutions of $x^2 - 1 = 0$ in $\mathcal{O}_2$ are $\pm 1$. Similarly, using the surjection $\theta_2 \colon \mathcal{O}_3 \rightarrow \mathcal{O}_2$, if $1 + m$ is a lift of $1$ in $\mathcal{O}_3$ that satisfies $x^2=1$ in $\mathcal{O}_3$ where $m\in ker(\theta_2)$, then $m=0$ as $(ker (\theta_2))^2=0$. The other solution of $x^2-1=0$ in $\mathcal{O}_3$, which is a lift of $-1$, is again $-1$. Since $(ker (\theta_j))^2 = 0$ for any $j\geq 1$, the only solutions of $x^2=1$ over $\mathcal{O}_{j+1}$, which are the lifts of $\pm 1$, are $\pm 1$ only. Hence, by Hensel's lemma (\cite[theorem 7.18]{EisenbudCommAlgebra}) and the canonical isomorphism $\mathcal{A}\rightarrow \varprojlim\limits_{j\geq 1} \mathcal{A}/\mathfrak{m}^j\mathcal{A}$, we obtain that the only solutions of $x^2 = 1$ in $\mathcal{O}$, that are the lifts of $1$ and $-1$ are $1$ and $-1$, respectively. Hence, $\lambda = \pm 1$ and $\SL_n(\mathcal{O}) \cap \mathrm{RL}_n(\mathcal{O}) = \{\pm I_2\}$. 
\end{proof}

\begin{theorem}{\label{th: Engel PSL2}}
Let $\mathcal{O}_2$ be a local principal ideal ring of length $2$, with its residue field $k\cong \mathbb{F}_q$ of characteristic $\neq 2$. Let $e_{m+1}(x,y)$ be the $(m+1)$-th Engel word in the free group $\mathcal F_2=\langle x,y\rangle$ and $m\geq 1$. Then, there exists a constant $q_0(m)$ such that for all $q\geq q_0(m)$, the $(m+1)$-th Engel word map $e_{m+1}^{\mathcal{O}_2}$ is surjective on  $\mathrm{P}\SL_2(\mathcal{O}_2)$. One can take $q_0(m) = 2\cdot3^{2^{m+2}}$.   
\end{theorem}
\begin{proof}
For any positive integer $m\geq 1$, let $(e_{m+1})_{\mathrm{PSL}_2(\mathcal{O}_2)}$ denotes the $(m+1)$-th Engel word on $\mathrm{PSL}_2(\mathcal{O}_2)$. By Lemma~\ref{psl} it follows that $\mathrm{PSL}_2(\mathcal{O}_2) \cong \SL_2(\mathcal{O}_2)/ \{\pm I_2\}$. Under the assumption $q\geq q_0(m)$, Proposition~\ref{regsem1}, Lemma~\ref{lifts of unipotents}, and Lemma~\ref{lifts of minus unipotents} together ensure that all the lifts of non-scalars must be in $\im ((e_{m+1})_{\mathrm{PSL}_2(\mathcal{O}_2)})$ for any $m\geq 1$. This only leaves the case of lift of $I$. The lifts of $I$ are in $\im ((e_{m+1})_{\mathrm{PSL}_2(\mathcal{O}_2)})$ by Proposition~\ref{prop: lift of I Engel}. Consequently, $(e_{m+1})_{\mathrm{PSL}_2(\mathcal{O}_2)}$ is surjective.
\end{proof}

We conclude this section with a comment regarding 
higher Engel words and the assumption $|k| \geq 5$ in Section~\ref{commutator-local ring-length two}. For higher Engel words, it may not always be possible that any lift of $-I$ lies in $\im(e_{m+1}^{\mathcal{O}})$ for every $m\geq 1$. This follows from \cite[Proposition 4.7]{BandmanGarionGrunewald2012}.
\begin{example}[When the residue field is $\mathbb{F}_3$]{\label{F_3 case}}

Consider $A=\left(\begin{array}{cc} 1  &  0\\ 1  & 1
\end{array}\right)$ and $B = \left(\begin{array}{cc} 2  &0  \\
1 & 2 \end{array}\right)$ be the elements in $SL_2(\mathbb{F}_3)$ with $o(A)=3$ and $o(B)=12$, respectively. The commutator subgroup $[\SL_2(\mathbb{F}_3), \SL_2(\mathbb{F}_3)]$ has order $8$. Therefore, $A$ and $B$ can never be written as a single commutator or as a finite product of commutators in $\SL_2(\mathbb{F}_3)$. Hence, $A$ and $B$ can never be written as a product of a finite number of commutators in $\SL_2(\R)$ when the residue field of $\R$ is $\mathbb{F}_3$. 
\end{example}

\printbibliography
\vspace{2em}
\end{document}